\documentclass[11pt,oneside]{amsart}

\usepackage{fullpage}

\usepackage{amsthm,amsfonts,amssymb,amsmath,amsxtra}
\usepackage[all]{xy}
\SelectTips{cm}{}
\usepackage{tikz}
\usepackage{verbatim}
\usepackage{mathrsfs}

\RequirePackage{xspace}
\RequirePackage{etoolbox}
\RequirePackage{varwidth}
\RequirePackage{enumitem}
\RequirePackage{tensor}
\RequirePackage{mathtools}
\RequirePackage{longtable}
\RequirePackage{multirow}
\RequirePackage{makecell}
\RequirePackage{bigints}

\usepackage[backref=page]{hyperref} %

\newcommand{\fkr}{\ensuremath{\mathfrak{r}}\xspace}

\newcommand{\BC}{\ensuremath{\mathbb{C}}\xspace}

\newcommand{\BE}{\ensuremath{\mathbb{E}}\xspace}
\newcommand{\BF}{\ensuremath{\mathbb{F}}\xspace}

\newcommand{\BL}{\ensuremath{\mathbb{L}}\xspace}
\newcommand{\BM}{\ensuremath{\mathbb{M}}\xspace}
\newcommand{\BN}{\ensuremath{\mathbb{N}}\xspace}

\newcommand{\BP}{\ensuremath{\mathbb{P}}\xspace}
\newcommand{\BQ}{\ensuremath{\mathbb{Q}}\xspace}

\newcommand{\BT}{\ensuremath{\mathbb{T}}\xspace}

\newcommand{\BV}{\ensuremath{\mathbb{V}}\xspace}
\newcommand{\BW}{\ensuremath{\mathbb{W}}\xspace}
\newcommand{\BX}{\ensuremath{\mathbb{X}}\xspace}
\newcommand{\BY}{\ensuremath{\mathbb{Y}}\xspace}
\newcommand{\BZ}{\ensuremath{\mathbb{Z}}\xspace}

\newcommand{\CC}{\ensuremath{\mathcal{C}}\xspace}
\newcommand{\CD}{\ensuremath{\mathcal{D}}\xspace}
\newcommand{\CE}{\ensuremath{\mathcal{E}}\xspace}

\newcommand{\CH}{\ensuremath{\mathcal{H}}\xspace}

\newcommand{\CM}{\ensuremath{\mathcal{M}}\xspace}
\newcommand{\CN}{\ensuremath{\mathcal{N}}\xspace}
\newcommand{\CO}{\ensuremath{\mathcal{O}}\xspace}

\newcommand{\CT}{\ensuremath{\mathcal{T}}\xspace}

\newcommand{\CV}{\ensuremath{\mathcal{V}}\xspace}

\newcommand{\CY}{\ensuremath{\mathcal{Y}}\xspace}
\newcommand{\CZ}{\ensuremath{\mathcal{Z}}\xspace}

\newcommand{\nat}{{\natural}}

\newcommand{\corr}{\mathrm{corr}}

\newcommand{\del}{\operatorname{\partial Orb}}

\newcommand{\disc}{{\mathrm{disc}}}

\DeclareMathOperator{\End}{End}

\newcommand{\Fil}{\ensuremath{\mathrm{Fil}}\xspace}

\newcommand{\GL}{\mathrm{GL}}

\DeclareMathOperator{\Gr}{Gr}

\DeclareMathOperator{\Hom}{Hom}

\newcommand{\id}{\ensuremath{\mathrm{id}}\xspace}

\DeclareMathOperator{\im}{im}

\DeclareMathOperator{\length}{length}
\DeclareMathOperator{\Lie}{Lie}

\DeclareMathOperator{\Nm}{Nm}

\DeclareMathOperator{\Orb}{Orb}

\DeclareMathOperator{\Res}{Res}
\newcommand{\rs}{\ensuremath{\mathrm{rs}}\xspace}

\DeclareMathOperator{\Spec}{Spec}
\DeclareMathOperator{\Spf}{Spf}

\newcommand{\val}{{\mathrm{val}}}

\newcommand{\Ver}{{\mathrm{Vert}}}

\newcommand{\U}{\mathrm{U}}

\DeclareMathOperator{\vol}{vol}

\newcommand{\wit}{\widetilde}

\newcommand{\pair}[1]{\langle {#1} \rangle}

\newcommand{\ov}{\overline}
\newcommand{\incl}{\hookrightarrow}
\newcommand{\lra}{\longrightarrow}

\newcommand{\la}{\langle}
\newcommand{\ra}{\rangle}
\newcommand{\lv}{\lvert}
\newcommand{\rv}{\rvert}

\newenvironment{altenumerate}
   {\begin{list}
      {(\theenumi) }
      {\usecounter{enumi}
       \setlength{\labelwidth}{0pt}
       \setlength{\labelsep}{0pt}
       \setlength{\leftmargin}{0pt}
       \setlength{\itemsep}{\the\smallskipamount}
       \renewcommand{\theenumi}{\roman{enumi}}
      }}
   {\end{list}}

\newenvironment{altitemize}
   {\begin{list}
      {$\bullet$}
      {\setlength{\labelwidth}{0pt}
	   \setlength{\itemindent}{5pt}
       \setlength{\labelsep}{5pt}
       \setlength{\leftmargin}{0pt}
       \setlength{\itemsep}{\the\smallskipamount}
      }}
   {\end{list}}

\setitemize[0]{leftmargin=*,itemsep=\the\smallskipamount}
\setenumerate[0]{leftmargin=*,itemsep=\the\smallskipamount}

\renewcommand{\to}{%
   \ifbool{@display}{\longrightarrow}{\rightarrow}%
   }
\let\shortmapsto\mapsto
\renewcommand{\mapsto}{%
   \ifbool{@display}{\longmapsto}{\shortmapsto}%
   }
\newlength{\olen}
\newlength{\ulen}
\newlength{\xlen}
\newcommand{\xra}[2][]{%
   \ifbool{@display}%
      {\settowidth{\olen}{$\overset{#2}{\longrightarrow}$}%
       \settowidth{\ulen}{$\underset{#1}{\longrightarrow}$}%
       \settowidth{\xlen}{$\xrightarrow[#1]{#2}$}%
       \ifdimgreater{\olen}{\xlen}%
          {\underset{#1}{\overset{#2}{\longrightarrow}}}%
          {\ifdimgreater{\ulen}{\xlen}%
             {\underset{#1}{\overset{#2}{\longrightarrow}}}
             {\xrightarrow[#1]{#2}}}}%
      {\xrightarrow[#1]{#2}}
   }
\makeatother
\newcommand{\xyra}[2][]{%
   \settowidth{\xlen}{$\xrightarrow[#1]{#2}$}%
   \ifbool{@display}%
      {\settowidth{\olen}{$\overset{#2}{\longrightarrow}$}%
       \settowidth{\ulen}{$\underset{#1}{\longrightarrow}$}%
       \ifdimgreater{\olen}{\xlen}%
          {\mathrel{\xymatrix@M=.12ex@C=3.2ex{\ar[r]^-{#2}_-{#1} &}}}%
          {\ifdimgreater{\ulen}{\xlen}%
             {\mathrel{\xymatrix@M=.12ex@C=3.2ex{\ar[r]^-{#2}_-{#1} &}}}
             {\mathrel{\xymatrix@M=.12ex@C=\the\xlen{\ar[r]^-{#2}_-{#1} &}}}}}%
      {\mathrel{\xymatrix@M=.12ex@C=\the\xlen{\ar[r]^-{#2}_-{#1} &}}}%
   }
\makeatletter
\newcommand{\xla}[2][]{%
   \ifbool{@display}%
      {\settowidth{\olen}{$\overset{#2}{\longleftarrow}$}%
       \settowidth{\ulen}{$\underset{#1}{\longleftarrow}$}%
       \settowidth{\xlen}{$\xleftarrow[#1]{#2}$}%
       \ifdimgreater{\olen}{\xlen}%
          {\underset{#1}{\overset{#2}{\longleftarrow}}}%
          {\ifdimgreater{\ulen}{\xlen}%
             {\underset{#1}{\overset{#2}{\longleftarrow}}}
             {\xleftarrow[#1]{#2}}}}%
      {\xleftarrow[#1]{#2}}
   }
\newcommand{\isoarrow}{%
   \ifbool{@display}{\overset{\sim}{\longrightarrow}}{\xrightarrow\sim}%
   }
\renewcommand{\lra}{%
   \ifbool{@display}{\longleftrightarrow}{\leftrightarrow}%
   }

\newcommand{\barE}{{\ov\BE}}
\newcommand{\Fb}{{\breve F}}
\newcommand{\OFb}{{O_{\breve F}}}

\newcommand{\rd}{\mathrm{d}}

\newcommand{\wt}{\wit}

\DeclareFontFamily{U}{matha}{\hyphenchar\font45}
\DeclareFontShape{U}{matha}{m}{n}{
      <5> <6> <7> <8> <9> <10> gen * matha
      <10.95> matha10 <12> <14.4> <17.28> <20.74> <24.88> matha12
      }{}
\DeclareSymbolFont{matha}{U}{matha}{m}{n}
\DeclareFontFamily{U}{mathx}{\hyphenchar\font45}
\DeclareFontShape{U}{mathx}{m}{n}{
      <5> <6> <7> <8> <9> <10>
      <10.95> <12> <14.4> <17.28> <20.74> <24.88>
      mathx10
      }{}
\DeclareSymbolFont{mathx}{U}{mathx}{m}{n}

\DeclareMathSymbol{\obot}         {2}{matha}{"6B}

\newtheorem{theorem}[subsubsection]{Theorem}
\newtheorem{proposition}[subsubsection]{Proposition}
\newtheorem{lemma}[subsubsection]{Lemma}
\newtheorem {conjecture}[subsubsection]{Conjecture}
\newtheorem{corollary}[subsubsection]{Corollary}

\theoremstyle{definition}
\newtheorem{definition}[subsubsection]{Definition}
\newtheorem{example}[subsubsection]{Example}

\newtheorem{remark}[subsubsection]{Remark}

\numberwithin{equation}{subsection}

\newcommand{\tN}{{\wit{\mathcal{N}}^{[1]}_n}}
\newcommand{\Na}{\mathcal{N}^{[1]}_n}
\newcommand{\Nared}{\mathcal{N}^{[1],\mathrm{red}}_n}

\newcommand{\Nns}{\CN_n^{[1],{\mathrm{ns}}}}
\newcommand{\Nbl}{\CN_n^{[1],\circ}}
\newcommand{\Nlk}{\CN_n^{[1],\dagger}}
\newcommand{\tNns}{{\wit{\CN}^{[1],\mathrm{ns}}_n}}

\newcommand{\tNbl}{{\wit{\CN}}^{[1],\circ}_n}
\newcommand{\tNlk}{{\wit{\CN}}^{[1],\dagger}_n}
\newcommand{\Zns}{\Zx^\mathrm{ns}}

\newcommand{\A}{\mathsf{A}}
\newcommand{\B}{\mathsf{B}}
\newcommand{\Zx}{\mathcal{Z}(u)}

\newcommand{\PL}{\mathbb{P}_\Lambda}

\usepackage{xcolor}

\newcommand{\Nr}{\mathcal{N}^{[r]}_n}
\newcommand{\tNr}{{\wit{\mathcal{N}}^{[r]}_n}}
\newcommand{\Nrs}{\mathcal{N}^{[r,s]}_n}
\newcommand{\Nrp}{\mathcal{N}^{[r+\varepsilon]}_{n+1}}
\newcommand{\Nrzp}{\mathcal{N}^{[r+\varepsilon,0]}_{n+1}}
\newcommand{\Ns}{\mathcal{N}^{[s]}_n}
\newcommand{\Nrz}{\mathcal{N}^{[r,0]}_n}
\newcommand{\Nn}{\mathcal{N}_{n}}
\newcommand{\N}{\mathcal{N}_{n+1}}

\newcommand{\Mnr}{{\wit{\mathcal{M}}^{[r]}_n}}

\newcommand{\nr}{\mathcal{N}^{[r]}_{n}}
\newcommand{\tnr}{{\wit{\mathcal{N}}^{[r]}_{n}}}
\newcommand{\tn}{{\wit{\mathcal{N}}^{[1]}_{n}}}

\newcommand{\n}{\mathcal{N}_{n+1}}

\newcommand{\x}{{u}}
\newcommand{\fp}{\varphi}
\newcommand{\Fp}{\phi}
\newcommand{\F}{\mathbf{F}}
\newcommand{\V}{\mathbf{V}}

\newcommand{\KM}{K'^\dagger_{n+1}}

\newcommand{\kN}{K_{n+1}}

\newcommand{\knr}{K_{n}^{[r]}}

\newcommand{\tknr}{\wit{K}_{n}^{[r]}}

\newcommand{\GW}{G_{W_1}}

\newcommand{\kb}{\bar k}

\newcommand{\Bc}{\mathrm{BC}}

\title[Quasi-canonical AFL and  AT conjectures at parahoric]{Quasi-canonical AFL and Arithmetic Transfer conjectures at parahoric levels}

\author{Chao Li}
\address{Columbia University, Department of Mathematics, 2990 Broadway,	New York, NY 10027, USA}
\email{chaoli@math.columbia.edu} 
\author{Michael Rapoport}
\address{Mathematisches Institut der Universit\"at Bonn, Endenicher Allee 60, 53115 Bonn, Germany, and University of Maryland, Department of Mathematics, College Park, MD 20742, USA}
\email{rapoport@math.uni-bonn.de}
\author{Wei Zhang}
\address{Massachusetts Institute of Technology, Department of Mathematics, 77 Massachusetts Avenue, Cambridge, MA 02139, USA}
\email{weizhang@mit.edu}

 \date{\today}

\begin{document}

\begin{abstract}
In the first part of the paper, we formulate several arithmetic transfer conjectures, which are variants of the arithmetic fundamental lemma conjecture in the presence of ramification. The ramification comes from the choice of non-hyperspecial parahoric level structure. We prove a  \emph{graph version} of these arithmetic transfer conjectures, by relating it to the \emph{quasi-canonical arithmetic fundamental lemma}, which we also establish. We relate some of the arithmetic transfer conjectures to the arithmetic fundamental lemma conjecture for the whole Hecke algebra in our recent paper \cite{LRZ}. As a consequence, we prove these conjectures in some simple cases. In the second part of the paper, we elucidate the structure of an integral model of a certain member of the \emph{almost selfdual} Rapoport--Zink tower, thereby proving conjectures in \cite{Kudla2012}  and \cite{LZ22}.  This result allows us to verify the hypotheses of the graph version of the arithmetic transfer  conjectures in a particular case.
\end{abstract}

\maketitle{}
\tableofcontents{}

\section{Introduction}

Inspired by the Jacquet--Rallis approach \cite{JR} to the global Gan--Gross--Prasad conjecture, 
the third author proposed a relative trace formula approach to the arithmetic Gan--Gross--Prasad conjecture. In this context, he formulated the \emph{arithmetic fundamental lemma} (AFL) conjecture \cite{Zha12}. The AFL conjecturally relates the special value of the derivative of an orbital integral to an arithmetic intersection number on a \emph{Rapoport--Zink formal moduli space of $p$-divisible groups} (RZ-space) attached to a unitary group.   The AFL formula takes the following form. 

   Let $p$ be an odd prime number. Let $F_0$ be a finite extension of $\BQ_p$ and let $F/F_0$ be an unramified quadratic extension. Let $W_1$ be a non-split $F/F_0$-hermitian space of dimension $n+1$ and let $W_1^\flat$ be the perp-space of a vector $u_1\in W_1$ of unit length (the \emph{special vector}). Let $G'=\Res_{F/F_0}(\GL_n\times\GL_{n+1})$ and $G_{W_1}=\U(W_1^\flat)\times\U(W_1)$. Then the following identity holds for all matching regular semi-simple elements $\gamma\in G'(F_0)$ 
and 
 $g\in G_{W_1}(F_0)$,
\begin{equation*}\label{IntroAFL}
\left\langle \Delta, g\Delta\right\rangle_{\CN_{n, n+1}}\cdot\log q=- \frac{1}{2}\del(\gamma, {\mathbf 1}) .
\end{equation*}
Here $q$ denotes the cardinality of the residue field of $F_0$. On the RHS, $\del(\gamma, {\mathbf 1})$  is the special value of the derivative of the weighted orbital integral of the unit element in the spherical Hecke algebra $\CH_{K^{\prime \flat}\times K'}$ of the natural hyperspecial compact subgroup $K'^\flat\times K'$ of $\GL_n(F)\times\GL_{n+1}(F)$.  We note that, in contrast to \cite{Zha12},  the natural transfer factor  $\omega(\gamma)$   of \cite{RSZ2} has been incorporated in the definition of  $\del(\gamma, {\mathbf 1})$.  On the LHS appears the intersection number of the diagonal cycle $\Delta$ of the product RZ-space $\CN_{n, n+1}=\CN_n\times\CN_{n+1}$ with its translate under the automorphism of $\CN_{n, n+1}$ induced by $g$. Here, for any $n$,  $\CN_n$ is the Rapoport--Zink moduli space of framed \emph{basic} principally polarized $p$-divisible groups with action of $O_F$ of signature $(1, n-1)$. Both sides  of the identity only depend on the orbits of $\gamma$, resp. $g$, under natural group actions.

The AFL conjecture is now known to hold in general, cf. W. Zhang \cite{Zha21}, Mihatsch--Zhang \cite{MZ}, Z. Zhang \cite{ZZha}.  These proofs are global in nature. Local proofs of the AFL are known for $n=1,2$ (W. Zhang \cite{Zha12}),  and for minuscule elements (Rapoport--Terstiege--Zhang \cite{RTZ}, He--Li--Zhu  \cite{HLZ}).

It is essential for the AFL conjecture that one is dealing with a situation that is unramified in every possible sense, i.e., the quadratic extension $F/F_0$ defining the unitary group is unramified, and the special vector has unit length, and the function appearing in the derivative of the orbital integral is the characteristic function of a hyperspecial maximal open compact subgroup. The AFL has to be modified when these unramifiedness hypotheses are dropped. In the context of the \emph{fundamental lemma} (FL) conjecture of Jacquet--Rallis, this question leads naturally to the \emph{smooth transfer} (ST) conjecture, proved by the third author in the non-archimedean case \cite{Z14}. In the arithmetic context, this question naturally leads  to the problem of formulating  \emph{arithmetic transfer} (AT) conjectures. A number of such AT conjectures are formulated in \cite{RSZ2}. These conjectures are proved in a small number of cases, cf. \cite{RSZ1}, \cite{RSZ2}. 

The limiting factor to formulating such conjectures is the geometric side of the conjecture. Indeed, for formulating an AT conjecture (at least in the naive sense), one has to make sure that the ambient Rapoport--Zink space is regular (otherwise, the Serre definition of the intersection multiplicity is not applicable), and this strongly limits the possibilities, see \cite{HPR}.  In the present paper, we impose on the quadratic extension $F/F_0$ to be \emph{unramified} but allow the polarization in the RZ-moduli problem to be non-principal. More precisely, let $\CN_n^{[r]}$ be the formal moduli space of formal $O_{F_0}$-modules with action of $O_F$ of signature $(1, n-1)$ and a compatible polarization of type $r$, i.e., the kernel is killed by $\varpi$ (the uniformizer of $F_0$) and is of order $q^{2r}$. Thus  $\CN_n^{[0]}=\CN_n$. In  \cite{RSZ2}, we considered on the geometric side the natural closed embedding of $\CN_n^{[0]}$ into $\CN_{n+1}^{[1]}$ and its graph $\Delta$ in the product $\CN_n^{[0]}\times\CN_{n+1}^{[1]}$. We then formed  the intersection product of $\Delta$ with its translate under the (regular semi-simple) automorphism $g$ of $\CN_n^{[0]}\times\CN_{n+1}^{[1]}$ and related it to the derivative of an orbital integral at a matching element $\gamma$. As highlighted above, the intersection product makes sense, since $\CN_n^{[0]}\times\CN_{n+1}^{[1]}$ is regular. More generally, in \cite{ZZha}, Z.~Zhang considers for any $r\geq 0$ the natural embedding of $\CN_n^{[r]}$ into  $\CN_n^{[r']}$, where $r'=r$ or $r'=r+1$ and forms an intersection on $\CN_n^{[r]}\times\CN_n^{[r']}$. However, this last product is not regular, unless $r=0$; therefore, when $r\geq 1$, this product is replaced in loc.~cit. by a blow-up. Using this blow-up,  an AT identity is formulated  and indeed proved. 

 In the present paper, we want to replace the pair $(r, r')=(0, 1)$ by the pair $(1, 0)$, or more generally by $(r, 0)$ for arbitrary $r>0$. The product $\CN_n^{[r]}\times\CN_{n+1}^{[0]}$ is regular so that  intersection products make sense on this space. A problem arises however from  the fact that there is no natural embedding of $\CN_n^{[r]}$ into $\CN_{n+1}^{[0]}$. Rather, one has to replace this embedding by a diagram linking $\CN_n^{[r]}$ to $\CN_{n+1}^{[0]}$,
 \begin{equation}\label{diagN}
 \begin{aligned}
 \xymatrix{&\tNr \ar[rd]^{\pi_2}  \ar[ld]_{\pi_1} &\\ \Nr &&  \CN_{n+1}^{[0]}.}
 \end{aligned}
 \end{equation} 
  In the generic fiber (i.e., for the corresponding rigid-analytic spaces), the morphism $\pi_1$ is part of  the RZ-tower corresponding to suitable open compact subgroups of a unitary group of dimension $n$. In other words, $\tNr$ is an integral model of a certain member of the RZ-tower. When $r$ is even, the new space $\tNr $ is simply the RZ-space $\CN_n^{[r, 0]}$ (corresponding to a parahoric subgroup in a quasisplit unitary group) and the map $\pi_1$ is then just the obvious transition map. When $r$ is odd, the new space $\tNr$ (corresponding to a non-parahoric in a non-quasisplit unitary group) is very mysterious.  One of our main results concerns the structure of $\wt\CN_n^{[1]}$. The following theorem is a simplified version of Theorem \ref{conj:KRSZ} in the text. 
 \begin{theorem}\label{Introstrr=1}
     \begin{altenumerate}
    \item\label{item:KR1} The formal scheme $\tN$ is regular of dimension $n$.
  \item\label{item:KR2} The morphism $\pi_1$ is finite flat of degree $q+1$, \'etale away from the  \emph{closed balloon locus} $\CN_n^{[1], \bullet}$, and totally ramified along $\CN_n^{[1], \bullet}$.      The  closed balloon locus  $\CN_n^{[1], \bullet}$ is a Cartier divisor which is a disjoint sum of copies of $\BP^{n-1}$, enumerated by the self-dual lattices  in the split $F/F_0$-hermitian space of dimension $n$.  
   \item\label{item:KR3} The morphism $\pi_2$ is proper and factors through the Kudla--Rapoport divisor $\CZ(u)$ of $\CN_{n+1}$ corresponding to the \emph{special vector} $u$ of valuation one. The resulting morphism $\tN\to \CZ(u)$ is a blow-up  in a zero-dimensional reduced subscheme $\CZ(u)^{\rm cent}$ and the exceptional divisor $\CN_n^{[1], {\rm exc}}$ in $\tN$ is a reduced Cartier divisor which maps isomorphically to $\CN_n^{[1], \bullet}$ under $\pi_1$.
   
\emph{ Here, the blow-up  is meant in the generalized sense of EGA, i.e., a blow-up in an ideal sheaf with support in $\CZ(u)^{\rm cent}$.}
     \end{altenumerate}
  \end{theorem}

Let us illustrate this theorem in the case $n=2$. In this case, we have
    \begin{itemize}
    \item $\CN_2^{[1]}$ is isomorphic to the Drinfeld half plane (this is the \emph{alternative interpretation} of the Drinfeld halfplane in \cite{KRalt}). The special fiber is a union of $\mathbb{P}^1$'s with dual graph a $(q+1)$-valent tree.  These $\BP^1$'s are of two kinds: even and odd; if two such $\BP^1$'s intersect, they are of different parity. 
    \item $\CN_2^{[1], \bullet}$ consists of all even $\mathbb{P}^1$'s.   The preimage under $\pi_1$ of an even $\BP^1$ is nonreduced (a ``fat'' $\mathbb{P}^1$ with multiplicity $q+1$). The preimage of an odd $\BP^1$ under $\pi_1$ is a Fermat curve of degree $q+1$.
    \item The special fiber of $\Zx$ consists of Fermat curves of degree $q+1$ intersecting at points in $\CZ(u)^{\rm cent}$ (the centers of the blow-up morphism $\pi_2$), and, conversely, all such intersection points are contained in $\CZ(u)^{\rm cent}$. Each Fermat curve contains  $q+1$ intersection points and $q+1$ Fermat curves pass through each intersection point.
    \item       The preimage of a  point of $\CZ(u)^{\rm cent}$ under $\pi_2$ is an exceptional divisor  of $\pi_2$ and can be identified with the underlying reduced scheme of a fat $\BP^1$ corresponding to a specified even $\BP^1$ in $\CN_2^{[1]}$. 
        \end{itemize}
    Figure       \ref{fig:n=2} illustrates the morphisms $\pi_1$ and $\pi_2$ in (\ref{eq:tNr3}) (for $n=2$ and $r=1$) on the special fibers locally around a blow-up point of $\Zx$.
    \begin{figure}[h]
      \centering
      \includegraphics[scale=.75]{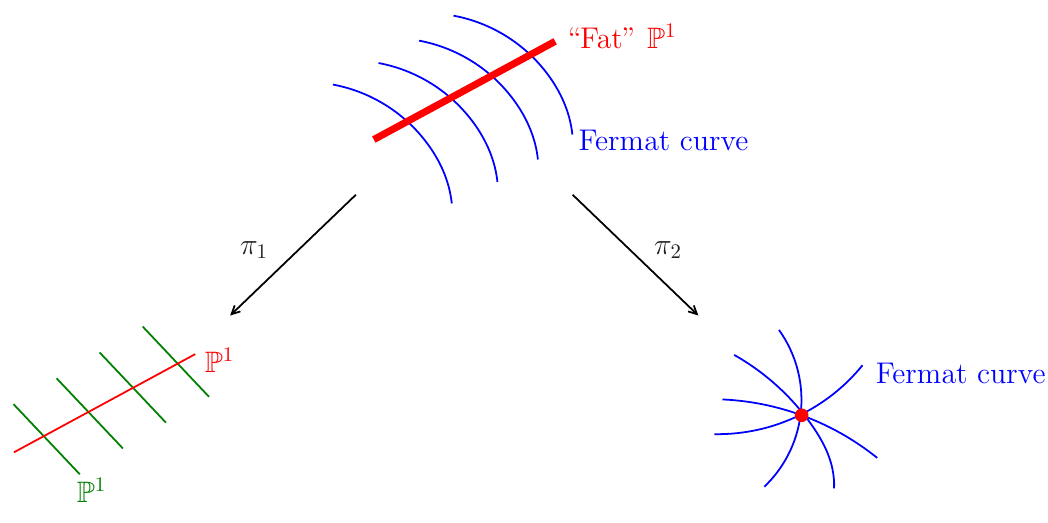}
      \caption{$n=2$}
      \label{fig:n=2}
    \end{figure}

 Theorem \ref{Introstrr=1} was conjectured in the unpublished manuscript \cite{Kudla2012} of Kudla and the second author and was used in the extension by the first and the third author of the Kudla--Rapoport intersection conjecture to the \emph{almost self-dual} case in \cite{LZ22}. 
 
  Returning to the case of general $r$, we obtain two closed embeddings,
 \begin{equation*}
 \begin{aligned}
&\quad\quad\quad\quad&\tNr \subset {\Nr\times\N},\\
&\quad\quad\quad\quad& \tNr \subset {\tNr\times\N}.\\
\end{aligned}
\end{equation*}
 The first is given by $(\pi_1,\pi_2)$.   The second is the graph of $\pi_2$. On the RHS of both inclusions, there is a (compatible) action of $G_{W_1}(F_0)=\U(W_1^\flat)(F_0)\times\U(W_1)(F_0)$. Here the perp-space is taken for the special vector $u_1\in W_1$ of length $\varpi^\varepsilon$, where $\varepsilon=\varepsilon(r)\in\{0, 1\}$ is the parity of $r$. Correspondingly, there are two intersection numbers arising in this context,   given as follows,
 \begin{equation*}
 \begin{aligned}
{\rm (i)}&\quad\quad\quad\quad &\left\langle \tNr, g\tNr\right\rangle_{\Nr\times\N}&:=\chi(\Nr\times \N, \tNr\cap^\BL g\tNr ),\\
{\rm (ii)}&\quad\quad\quad\quad &\left\langle \tNr, g\tNr\right\rangle_{\tNr\times\N}&:=\chi(\tNr\times \N, \tNr\cap^\BL g\tNr).\\
\end{aligned}
\end{equation*}
 Here $g\in G_{W_1}(F_0)$.  The first expression makes sense, since $\Nr$ is regular.  The second expression makes sense if $\tNr$ is regular. We conjecture (in a more precise way, cf. Conjecture \ref{conjreg}) that this is always the case. This conjecture holds for $r=1$ and when $r$ is even. 
 
 Here the variant (i) leads to our AT conjecture of type $(r, 0)$, by which we indicate the type of the vertex lattices defining the parahoric level of the relevant RZ spaces. The variant (ii) leads to the graph version of our AT conjecture. Let us first state our result on the graph version, which is reasonably complete (Corollary \ref{defifct}, Proposition \ref{prop:an-explicit-transfer even r}, Theorem \ref{conjt}). In the statement below, there appear compact open subgroups $\widetilde{K}_n^{[r]}\subset K_n^{[r]}$ of $\U(W_0^\flat)$ (the first a non-parahoric for odd $r$, the second a  maximal parahoric) and $K_{n+1}$ of $\U(W_0)$ (a hyperspecial maximal parahoric). Here $W_0$ denotes the split hermitian space of dimension $n+1$ and $W_0^\flat$ the perp-space for the special vector $u_0$ of length $\varpi^\varepsilon$. Then $\tNr$, resp. $\CN_n^{[r]}$, resp. $\CN_{n+1}$ are the members of the RZ tower corresponding to the open compact subgroups $\widetilde{K}_n^{[r]}$, resp. $K_n^{[r]}$, resp. $K_{n+1}$ . On the $\GL$-side, we have  analogous open compact subgroups $\widetilde{K}_n^{\prime [r]}\subset K'_n\subset \GL_n(F)$  and $K'_{n+1}\subset\GL_{n+1}(F)$.
 \begin{theorem}\label{Intro:graph}
  Let $\varphi'_r\in C_c^\infty(G')$ be as follows,
 \begin{equation*}
  \varphi'_r=\begin{cases}c_r c_r' (q^{2(n+1)}-1) \mathbf{1}_{\widetilde{K}_n^{\prime [r]}\times K'_{n+1}}+c_r ((-1)^{n+1}+1)\mathbf{1}_{G'(O_{F_0})},& \text{when $r$ is odd}\\
 c_rc_r'\mathbf{1}_{\widetilde{K}_n^{\prime [r]}\times K_{n+1}'}, & \text{when $r$ is even.}
\end{cases} 
 \end{equation*}
Then $\varphi'_r$ is a transfer of $(c_r^2\cdot {\bf 1}_{ \widetilde{K}_n^{[r]}\times K_{n+1}}, 0)\in C_c^\infty(G_{W_0})\times C_c^\infty(G_{W_1})$.  

Assume that $\wt{\CN}_n^{[r]}$ is regular. Then, if $\gamma\in G'(F_0)_\rs$ is matched with  $g\in G_{W_1}(F_0)_\rs$, 
\begin{equation*}
  \left\langle \tnr, g\tnr\right\rangle_{\tnr\times\n}\cdot\log q= -\frac{1}{2}\del\big(\gamma,  \varphi'_r \big)- \Orb\big(\gamma,  \fp'_{r, \corr} \big) ,
\end{equation*}
where
$$
\fp'_{r,\corr}=\begin{cases} c_r\cdot (n+1)\, {\bf 1}_{G'(O_{F_0})}\cdot \log q
, & $n$ \text{ is even, and $r$ is odd}\\
0, & $n$ \text{ is odd or $r$ is even.}
\end{cases}
$$
 \end{theorem}
  Here  $G'(O_{F_0})$ denotes a certain maximal compact subgroup of $G'(F_0)$. The integers $c_r$ and $c'_r$ are related to the normalizations of measures, see \eqref{defc}. The proof of Theorem \ref{Intro:graph} is by reduction (via the factorization of $\pi_2$ through the inclusion of the special divisor $\CZ(u_1)$ of $\CN_{n+1}$) to the AFL when $r$ is even, resp. to the quasi-canonical AFL when $r$ is odd, see Theorem \ref{conjt}. The topic of this latter variant of the AFL is another main theme of the paper, which we discuss next.

  Let $\CZ(u)\subset \CN_{n+1}$ be the Kudla--Rapoport divisor for a special vector of valuation $\varepsilon\in \{ 0,1\}$, cf. \cite{Kudla2011}. When $\varepsilon=0$, the structure of $\CZ(u)=\Delta\simeq \CN_n$ is clear: it is a regular formal scheme of dimension $n$ which is formally smooth over $\Spf O_{\breve F}$. When $\varepsilon=1$, the structure of $\CZ(u)$ is given by the following theorem (Theorem \ref{structZ}).
  \begin{theorem}\label{intro-structZ}
Let $u$ be a special vector of valuation one. The formal scheme $\CZ(u)$ is regular of dimension $n$, and formally smooth over $\Spf O_{\breve F}$ outside a zero-dimensional closed subset of $\CZ(u)^{\rm red}$. 
\end{theorem}
Note that  $\CZ(u)$ is its own difference divisor, which implies the regularity of $\CZ(u)$ (regularity holds for any difference divisor \cite{Ter}, cf. also \cite{Zhu}). When $n=1$, the  formal scheme $\CZ(u)$ (in $\CN_2$) is the quasi-canonical divisor of conductor one introduced in \cite{Kudla2011}, which is in turn closely related to the quasi-canonical  lifting of level one of Gross.

 The quasi-canonical AFL arises from the closed embedding,
 \begin{equation*}
 \begin{aligned}
&\quad\quad\quad\quad& \Zx\subset{\Zx\times\N}.
\end{aligned}
\end{equation*}
 By the regularity of $\CZ(u)$,  there is the well-defined intersection number for $g\in G_{W_1}(F_0)$,
 \begin{equation*}
 \begin{aligned}
&\quad\quad\quad\quad &\left\langle \Zx, g\Zx\right\rangle_{\Zx\times\N}&:=\chi(\Zx\times \N, \Zx\cap^\BL g\Zx ).
\end{aligned}
\end{equation*}
   Note that the AFL gives an analytic expression for this when $\varepsilon=0$. When $\varepsilon=1$, the corresponding statement is the following theorem (Theorem \ref{thmzAFL}). Recall the non-parahoric $\tilde{K}_n^{[1]}$ of $\U(W_0^\flat)(F_0)$ corresponding to the member $\wt{\CN}_n^{[1]}$ of the RZ-tower mentioned above. Also, let $K_{n+1}\subset \U(W_0)(F_0)$ be the stabilizer of a selfdual lattice.   
\begin{theorem}\label{Intro:qcAFL}
Let $c_1'=(q^2+1)(q^2-1)$.  Consider the function 
\begin{equation*}
\varphi'_1= c_1'(q^{2(n+1)}-1) \mathbf{1}_{\widetilde{K}_n^{\prime [1]}\times K'_{n+1}}+ ((-1)^{n+1}+1)\mathbf{1}_{G'(O_{F_0})}\in C_c^\infty(G') .
\end{equation*}
 Then $\varphi'_1$ is a transfer of $({\bf 1}_{\tilde {K}_n^{[1]}\times \kN},0)\in C_c^\infty(G_{W_0})\times C_c^\infty(G_{W_1})$ and, if $\gamma\in G'(F_0)_\rs$ is matched with  $g\in G_{W_1}(F_0)_\rs$, then
\begin{equation*}
  \left\langle \Zx, g\Zx\right\rangle_{\Zx\times\n}\cdot\log q= -\frac{1}{2}\del\big(\gamma,  \varphi'_1 \big)- \Orb\big(\gamma,  \fp'_{1, \corr} \big) ,
\end{equation*}
where 
$$
\fp'_{1, \corr}=\begin{cases} \frac{1}{c'_1}(n+1) {\bf 1}_{G'(O_{F_0})}\cdot\log q
, & $n$ \text{ is even, }\\
0, & $n$ \text{ is odd.}
\end{cases}
$$
\end{theorem}

We now return to the AT problem, pertaining to the intersection number $\left\langle \tNr, g\tNr\right\rangle_{\Nr\times\N}$ above. Here we have only partial results. Quite generally, we have the following conjecture, see Conjecture \ref{conjrodd all}.

\begin{conjecture}\label{conjATIntro}
\begin{altenumerate}
\item\label{item:conjr1Intro} There exists a transfer $\varphi'\in C_c^\infty(G')$ of $(c_r^2\cdot {\bf 1}_{K_n^{[r]}\times K_{n+1}},0)\in C_c^\infty(G_{W_0})\times C_c^\infty(G_{W_1})$ such that, if $\gamma\in G'(F_0)_\rs$ is matched with  $g\in G_{W_1}(F_0)_\rs$, then
 \begin{equation*}
    \left\langle \tnr, g\tnr\right\rangle_{\nr\times\n} \cdot\log q=-\frac{1}{2}\del\big(\gamma,  \varphi' \big).
\end{equation*}
\item\label{item:conjr2Intro} For any transfer $\varphi'\in C_c^\infty(G')$ of $(c_r^2\cdot {\bf 1}_{K_n^{[r]}\times K_{n+1}},0)\in C_c^\infty(G_{W_0})\times C_c^\infty(G_{W_1})$, there exists $\fp'_\corr\in C_c^\infty(G')$  such that if $\gamma\in G'(F_0)_\rs$ is matched with  $g\in G_{W_1}(F_0)_\rs$, then
\begin{equation*}
     \left\langle \tnr, g\tnr\right\rangle_{\nr\times\n}  \cdot\log q= -\frac{1}{2}\del\big(\gamma,  \varphi' \big) -  \Orb\big(\gamma,  \fp'_\corr \big).
\end{equation*}
\end{altenumerate}
\end{conjecture}

Under additional hypotheses, we can give candidates for the function $\varphi'$ in part (i), cf. Conjecture \ref{conjreven}. We  recall from \cite[\S3.6]{LRZ}  the base change homomorphism between spherical Hecke algebras,
 $$\Bc: \CH_{K^{\prime}_n}\otimes_{\BQ}\CH_{K^{\prime}_{n+1}}\rightarrow \CH_{K_{n}} \otimes_{\BQ}\CH_{K_{n+1}}.$$
 Recall from \cite[\S4]{LRZ} the atomic  Hecke function in $\CH_{K_n}$, defined as the  convolution,
\begin{equation*}\label{def phirintro}
\varphi_r:=\vol(\knr)^{-1}{\bf 1}_{K_n\knr}\ast {\bf 1}_{\knr K_n}.
\end{equation*}
\begin{conjecture}\label{conjrevenIntro}
Assume that $r$ is even. Let $\varphi'$ be any element in $\CH_{K^{\prime}_n}\otimes_{\BQ}\CH_{K^{\prime}_{n+1}}$ such that 
$$\Bc(\varphi')= \varphi_r\otimes {\bf 1}_{ K_{n+1}},$$ (then $\varphi'$ is a transfer  of $(\vol(K_n^{[r,0]})^{-2}\,{\bf 1}_{K_n^{[r]}\times K_{n+1}},0)\in C_c^\infty(G_{W_0})\times C_c^\infty(G_{W_1})$).
If $\gamma\in G'(F_0)_\rs$ is matched with  $g\in G_{W_1}(F_0)_\rs$, then \begin{equation*}
    \left\langle \tnr, g\tnr\right\rangle_{\nr\times\n} \cdot\log q=-\frac{1}{2}\del\big(\gamma,  \varphi' \big).
\end{equation*}
\end{conjecture}

 When $r=0$,  Conjecture \ref{conjrevenIntro} recovers the arithmetic fundamental lemma. For arbitrary even $r$, we show that Conjecture \ref{conjrevenIntro} follows from the AFL  in \cite{LRZ} (see Conjecture \ref{AFL Hk}) for certain (non unit) elements in the spherical Hecke algebra, cf. Corollary \ref{CorAFLwh}. Unfortunately, we know of no case of even $r\geq 2$, where Conjecture \ref{AFL Hk} is proved. 

When $r$ is odd, we only can give $\varphi'$ in the following special case, cf.  Conjecture \ref{conjrodd}.  Define the analogous atomic  Hecke function
\begin{equation*}
\varphi_{0}^{[n+1]}:= {\bf 1}_{K_{n+1}^{[n+1]}K_{n+1}}\ast {\bf 1}_{K_{n+1} K_{n+1}^{[n+1]}}\in \CH_{K_{n+1}^{[n+1]}},
\end{equation*}
where we note that $\CH_{K_n^{[n]}}$ and $\CH_{K_{n+1}^{[n+1]}}$ are both spherical Hecke algebras(!).

\begin{conjecture}\label{conjroddIntro}
Let $r$ be odd and assume $r=n$. Let $\varphi'$ be any element in $\CH_{K^{\prime}_n}\otimes_{\BQ}\CH_{K^{\prime}_{n+1}}$ such that 
$$\Bc(\varphi')=  {\bf 1}_{K_n^{[n]}}\otimes\varphi_0^{[n+1]}\in \CH_{K_n^{[n]}}\otimes_\BQ\CH_{K_{n+1}^{[n+1]}}.$$ 
\label{item:conjr1}
If $\gamma\in G'(F_0)_\rs$ is matched with  $g\in G_{W_1}(F_0)_\rs$, then
 \begin{equation*}
    \left\langle \tnr, g\tnr\right\rangle_{\nr\times\n} \cdot\log q=-\frac{1}{2} \del\big(\gamma,  \varphi' \big).
\end{equation*}
\end{conjecture}
Again, as for even $r$, Conjecture \ref{conjroddIntro} follows from the AFL in \cite{LRZ} (see Conjecture \ref{AFL Hk}) for certain (non unit) elements in the spherical Hecke algebra, cf. Corollary \ref{cor odd}. For $n=1$,  Conjecture  \ref{AFL Hk}  holds, cf. \cite[Thm. 7.5.1]{LRZ}. Hence we obtain the following theorem.
 
 \begin{theorem} \label{thm n=r=1Intro}
Conjecture \ref{conjroddIntro} holds when $n=r=1$.
\end{theorem}
It would be very interesting to construct an explicit candidate for $\varphi'$ in Conjecture \ref{conjATIntro}, at least in the case $r=1$ but when $n$ is arbitrary. We hope to return to this problem in future work. 

There is another kind of AT problem, related to a diagram similar to \eqref{diagN}, but relating this time $\CN_n$ to $\CN_{n+1}^{[r]}$, cf.  \eqref{eq Mnr},
\begin{equation}
\begin{aligned}
\xymatrix{&\Mnr \ar[rd]^{\pi'_2}  \ar[ld]_{\pi'_1} &\\ \CN_n&&  \CN_{n+1}^{[r]}.}
\end{aligned}\end{equation} 
The geometry of $\Mnr$ is in a sense the opposite of  that of $\tNr$. When $r$ is even and $r\leq n$, then $\Mnr$ is very singular but contains   a closed formal subscheme $\wt{\CM}_n^{[r], +}$ which is isomorphic to $\CN_n^{[0, r]}$, hence  is regular with semi-stable reduction, cf. \S \ref{ss:Mnr}. There is another  closed formal subscheme $\wt{\CM}_n^{[r], -}$ about which we know very little (e.g., if it  is regular); then  $\wt{\CM}_n^{[r]}$ is the union of  $\wt{\CM}_n^{[r], +}$  and $\wt{\CM}_n^{[r], -}$. When $r$ is odd, then $\Mnr\simeq\CN_n^{[0, r-1]}\simeq \wt{\CN}_n^{[r-1]}$ is regular with semi-stable reduction.  

We consider the  intersection number arising in this context,   given as follows,
 \begin{equation*}
 \left\langle \Mnr, g\Mnr\right\rangle_{\CN_n\times\CN_{n+1}^{[r]}}:=\chi(\CN_n\times\CN_{n+1}^{[r]}, \Mnr\cap^\BL g\Mnr ) , \quad g\in G_{W_1}(F_0).
\end{equation*}
 This leads to our AT conjecture of type $(0,r)$. We have the following general conjecture (cf. Conjecture \ref{conj 0 r +}), analogous to Conjecture \ref{conjATIntro}.  
 \begin{conjecture}\label{conj 0 r -Intro}
Let $r$ be  such that $0\leq r\leq n+1$, with parity $\varepsilon=\varepsilon(r)$. Let $W_\varepsilon$ be the hermitian space of dimension $n+1$ with invariant $(-1)^\varepsilon$, and  denote by $W_{\varepsilon+1}$ the hermitian space of the same dimension $n+1$ and with opposite invariant. As before, the perp-spaces $W_\varepsilon^\flat$ and $W_{\varepsilon+1}^\flat$ are formed using special vectors of length $\varpi^\varepsilon$. Also, recall the function $\varphi_r^{[\varepsilon]}\in \CH_{\U(W_\varepsilon)}$ from \eqref{def phir}.

\begin{altenumerate}
\item There exists $\varphi'\in C_c^\infty(G')$ with transfer $({\bf 1}_{K_n^{[0]}}\otimes \varphi_r^{[\varepsilon]},0)\in C_c^\infty(G_{W_\varepsilon})\times C_c^\infty(G_{W_{\varepsilon+1}})$ such that, if $\gamma\in G'(F_0)_\rs$ is matched with  $g\in G_{W_1}(F_0)_\rs$, then
 \begin{equation*}
 \left\langle \wt\CM_n^{[r]}, g\wt\CM_n^{[r]}\right\rangle_{\CN_n^{[0]}\times\CN_{n+1}^{[r]}} 
 \cdot\log q=-\frac{1}{2}\del\big(\gamma,  \varphi' \big).
\end{equation*}
\item\label{item:conjr2} For any $\varphi'\in C_c^\infty(G')$ transferring to   $({\bf 1}_{K_n^{[0]}}\otimes \varphi_r^{[\varepsilon]},0)\in C_c^\infty(G_{W_\varepsilon})\times C_c^\infty(G_{W_{\varepsilon+1}})$, there exists $\fp'_\corr\in C_c^\infty(G')$  such that, if $\gamma\in G'(F_0)_\rs$ is matched with  $g\in G_{W_1}(F_0)_\rs$, then
\begin{equation*}
 \left\langle \wt\CM_n^{[r]}, g\wt\CM_n^{[r]}\right\rangle_{\CN_n^{[0]}\times\CN_{n+1}^{[r]}} 
      \cdot\log q= -\frac{1}{2}\del\big(\gamma,  \varphi' \big) - \Orb\big(\gamma,  \fp'_\corr \big).
\end{equation*}
\end{altenumerate}
\end{conjecture}

When $r$ is even, we can give candidates for the function $\varphi'$ in part (i), cf Conjecture \ref{conjreven 0 r}. 
\begin{conjecture}\label{conjreven 0 rIntro}
Let $0\leq r\leq n+1$, with $r$  even. Let $\varphi'$ be any element in $\CH_{K^{\prime}_n}\otimes_{\BQ}\CH_{K^{\prime}_{n+1}}$ such that 
$$\Bc(\varphi')= {\bf 1}_{ K_{n}}\otimes  \varphi_r \in \CH_{K_n}\otimes_\BQ\CH_{K_{n+1}} .$$ 
If $\gamma\in G'(F_0)_\rs$ is matched with  $g\in G_{W_1}(F_0)_\rs$, then \begin{equation*}
    \left\langle  \wt\CM_n^{[r]}, g \wt\CM_n^{[r]}\right\rangle_{\nr\times\n} \cdot\log q=-\frac{1}{2}\del\big(\gamma,  \varphi' \big).
\end{equation*}
\end{conjecture}

Again, when $r=0$,  Conjecture \ref{conjreven 0 rIntro} recovers the arithmetic fundamental lemma. For arbitrary even $r$, we show that Conjecture \ref{conjreven 0 rIntro} follows from the AFL  in \cite{LRZ} (see Conjecture \ref{AFL Hk}) for certain (non unit) elements in the spherical Hecke algebra, cf. Corollary \ref{CorAFLwh 0 r}. When $r$ is even and $r=n+1$, then there is a close relation between  Conjectures \ref{conjreven 0 rIntro} and \ref{conjroddIntro}. For  even $r$ with $2\leq r\leq n$, there are also variants  of Conjecture \ref{conjreven 0 rIntro} involving the closed formal subschemes $\wt{\CM}_n^{[r], +}$ and $\wt{\CM}_n^{[r], -}$ of $\wt{\CM}_n^{[r]}$.  If $r$ is odd, we do not have a candidate for the function $\varphi'$, unless $r=1$ in which case we recover the AT conjecture made in \cite[\S10]{RSZ2}, see \S\ref{s: case (01)}.

The following table summarizes all the cases of AT conjectures in this paper. 
\medskip

{\setlongtables
\renewcommand{\arraystretch}{1.25}
\begin{longtable}{|c|c|c|c|c|}
\hline
\begin{varwidth}{\linewidth}
   \centering
  Type
\end{varwidth} 

&
\begin{varwidth}{\linewidth}
   \centering
 Ambient space
\end{varwidth} 
   & \begin{varwidth}{\linewidth}
   \centering
   The cycle
\end{varwidth} 
   &   \begin{varwidth}{\linewidth}
         \centering
        AT \\ Conjecture
      \end{varwidth}
   &  \begin{varwidth}{\linewidth}
         \centering
      	Relation to
      	AFL for \\ 
      	spherical Hecke\\ Conj.
\ref{AFL Hk}  
      \end{varwidth} 
   \\
\hline
 $(r,0)$:  $r$ even & $\CN_{n}^{[r]}\times \CN_{n+1}^{[0]}$&   $\wt \CN_{n}^{[r]} \simeq \CN_n^{[0,r]}$&
  Conj. \ref{conjreven} & $ \varphi_r\otimes {\bf 1}_{K_{n+1}}$ (Cor. \ref{CorAFLwh})
   \\
 \hline
 $(r,0)$:  $r$ odd& $\CN_{n}^{[r]}\times \CN_{n+1}^{[0]}$&   $\wt \CN_{n}^{[r]} $ &  Conj. \ref{conjrodd all}
	& None
	\\
	 \hline
  $(r,0)$:  $r=n$ odd & $\CN_{n}^{[n]}\times \CN_{n+1}^{[0]}$&   $\wt \CN_{n}^{[n]} $ &  Conj. \ref{conjrodd}
	& $ {\bf 1}_{K_n}\otimes\varphi_{n+1}$ (Cor. \ref{cor odd})

	\\
\hline
 $(0,r')$:  $r'$ even & $\CN_{n}^{[0]}\times \CN_{n+1}^{[r']}$&   $\wt \CM_{n}^{[r']}$&
  Conj. \ref{conjreven 0 r} & $  {\bf 1}_{K_{n}}\otimes \varphi_{r'}$ (Cor. \ref{CorAFLwh 0 r})
   \\
   \hline
 $(0,r')$:  $r'$ even & $\CN_{n}^{[0]}\times \CN_{n+1}^{[r']}$&   $\wt \CM_{n}^{[r'],+}$&
  Conj. \ref{conj 0 r mixed}
 & None
   \\  \hline
 $(0,r')$:  $r'$ even & $\CN_{n}^{[0]}\times \CN_{n+1}^{[r']}$&   $\wt \CM_{n}^{[r'],-}$&
  Conj. \ref{conj 0 r mixed}
 & None
   \\
 \hline
  $(0,r')$:  $r'=n+1$ even & $\CN_{n}^{[0]}\times \CN_{n+1}^{[n+1]}$&   $\wt \CM_{n}^{[n+1]}\simeq\wt \CN_{n}^{[n]} $ &  Conj. \ref{conjrodd}
	& $ {\bf 1}_{K_n}\otimes\varphi_{n+1}$ (Cor. \ref{cor odd})
	\\
\hline
 $(0,r')$:  $r'$ odd& $\CN_{n}^{[0]}\times \CN_{n+1}^{[r']}$&   $\wt \CM_{n}^{[r']}\simeq \CN_{n}^{[0,r'-1]} $ &  Conj. \ref{conj 0 r +}
	& None
	\\
	 \hline
	  $(0,r')$:  $r'=1$& $\CN_{n}^{[0]}\times \CN_{n+1}^{[1]}$&   $\wt \CM_{n}^{[1]}\simeq \CN_{n}^{[0]} $ &  Thm. \ref{thm AT type 01}
	&$ {\bf 1}_{K_n}\otimes{\bf 1}_{K_{n+1}}$
	\\
	 \hline
\end{longtable}

 \medskip
 We note that there are two extreme cases $(r=n,0)$ and $(0,r'=n+1)$ in the table; they are in fact equivalent under the duality isomorphisms \eqref{eq: dual iso}.

Let us comment on the scope of our AT conjectures. We indeed seem to have exhausted all possible cases under a couple of natural constraints, as we explain now. In addition to the Kudla--Rapoport $\CZ$-divisor, we also use the $\CY$-divisor, cf. \cite[\S 5.2]{ZZha} (and see \S\ref{ss:Mnr}).  The AFL conjecture concerns the diagram
\begin{equation}
\begin{aligned}
\xymatrix{&\CN_{n}^{[0]} \ar[rd]  \ar[ld] &\\ \CN_{n}^{[0]}&&  \CN_{n+1}^{[0]}, }
\end{aligned}
\end{equation} 
where $\CN_{n}^{[0]}\simeq \CZ(u_0)$, for a special vector of unit norm. Our spaces and cycles are variants of this diagram built on the following two considerations:
\begin{altenumerate}
\item [1) ] we would like the ambient space to be a product of RZ spaces of {\em maximal parahoric levels} and to be regular.
\item  [2) ] the $\CZ$-divisor, resp. the $\CY$-divisor $ \CZ(u)^{[r]} \hookrightarrow \CY(u)^{[r]}$ on $\CN_{n+1}^{[r]}$ should be isomorphic  to a lower-dimensional RZ space of maximal parahoric level. More precisely, we have the following \emph{exceptional isomorphisms}
\begin{equation}\label{eq: ex iso}
\begin{cases}
 \CZ(u)^{[r]}\simeq \CN_{n}^{[r]},& v(u)=0,\\
  \CY(u)^{[r]}\simeq \CN_{n}^{[r-1]},& v(u)=-1.
  \end{cases}
  \end{equation} 
 These are also called \emph{exceptional special divisors}.   It is conceivable that exceptional special divisors on $\CN_n^{[r]}$ are characterized by the property that they are regular formal schemes (besides the case  $\CZ(u)^{[0]}$ with $v(u)=1$).
 \end{altenumerate}
 
We are led to the formation of ``pull-back" diagrams of exceptional special divisors along the natural projection maps   $\CN_{n+1}^{[r_1,r_2]}\to \CN_{n+1}^{[r_i]}$ from  RZ spaces of (non-maximal) parahoric levels, where $0\leq r_1,r_2\leq n+1$ and $r_1\equiv r_2\mod 2$:  
\begin{equation}
\label{eq Mnr}
  \begin{aligned}
  \xymatrix{ & \wt\CZ_2 \ar@{_(->}[d]\ar[r] \ar@{}[rd]|{\square} &  \CC^{[r_2]}=\CZ(u_0)^{[r_2]} \text{ or } \CY(u_0) ^{[r_2]} \ar@{_(->}[d]
  \\ \wt\CZ_1  \ar@{^(->}[r]\ar[d] \ar@{}[rd]|{\square}  &  \CN_{n+1}^{[r_1,r_2]} \ar[d] \ar[r] &  \CN_{n+1}^{[r_2]} \\ 
\CC^{[r_1]}=  \CZ(u_0)^{[r_1]} \text{ or } \CY(u_0) ^{[r_1]} \ar@{^(->}[r] &   \CN_{n+1}^{[r_1]} .&}
  \end{aligned}
\end{equation}
Here $\CC^{[r_i]}$ are as in \eqref{eq: ex iso}.
By
the symmetry interchanging $r_1$ and $r_2$, it suffices to consider one of the two cartesian squares, say the bottom-left one. We would like to consider the cartesian product $\wt\CZ_1$ as our cycle and the product $\CC^{[r_1]}\times  \CN_{n+1}^{[r_2]}$ as the ambient space. The regularity of the product space happens if and only if (at least) one of the two factors is smooth over $\Spf O_{\breve F}$.  We distinguish  two cases.

{\em The case when $\CC^{[r_1]}$ is smooth}.  Then either $\CC^{[r_1]} = \CZ(u_0)^{[r_1]}\simeq \CN_{n}^{[0]}$ with  $v(u_0)=0$ and $ r_1=0$, or $\CC^{[r_1]} = \CY(u_0)^{[r_1]}\simeq \CN_{n}^{[0]}$ with  $v(u_0)=-1$ and $ r_1=1$. Summarizing these two possibilities and renaming $r_2$ as $r$ and recalling the parity $\varepsilon=\varepsilon(r)=r_1$, the lower left diagram in \eqref{eq Mnr} becomes the cartesian diagram \eqref{eq Mnr} defining the space $\Mnr$,
\begin{equation*}
  \begin{aligned}
  \xymatrix{\Mnr \ar@{^(->}[r] \ar[d] \ar@{}[rd]|{\square}  &  \CN_{n+1}^{[r,\varepsilon]} \ar[d] \\ 
  \CN_n^{[0]} \ar@{^(->}[r] &   \CN_{n+1}^{[\varepsilon]}.}
  \end{aligned}
\end{equation*}

{\em The case when  $\CN_{n+1}^{[r_2]}$ is smooth}. Then $r_2=0$ or $r_2=n+1$. There is a duality isomorphism $\CN_{n+1}^{[r]}\simeq \CN_{n+1}^{[n+1-r]}$ \cite[\S5.1]{ZZha} which interchanges the two exceptional special divisors  in \eqref{eq: ex iso}, cf.  \cite[Prop. 5.7]{ZZha}. To avoid repetitions, we thus assume $r_2=0$. Then $r_1$ is even. We let $r=r_1$ in the case $\CC^{[r_1]}=\CZ(u_0)^{[r_1]}\simeq \CN_n^{[r_1]}$ (and $v(u_0)=0$), resp. let $r=r_1-1$ in the case  $\CC^{[r_1]}=\CY(u_0)^{[r_1]}\simeq \CN_n^{[r_1-1]}$ (and $v(u_0)=-1$), so that we always have $r+\varepsilon=r_1$. Then the lower left diagram in \eqref{eq Mnr} becomes the  cartesian diagram \eqref{eq:tNr1} defining the space $\wt\CN_n^{[r]}$, 
\begin{equation*}
  \begin{aligned}
  \xymatrix{\tNr \ar@{^(->}[r] \ar[d] \ar@{}[rd]|{\square}  & \Nrzp \ar[d] \\ \Nr \ar@{^(->}[r] & \Nrp.}
  \end{aligned}
\end{equation*}

Therefore the list of our cases of AT conjectures is exhaustive, if we only consider pull-backs of KR divisors, further subject to the natural conditions 1) and 2)  above.
Moreover, from the above interpretation using the pull-backs of $\CZ$- or $\CY$-divisors, our cycles $\Mnr $ and $\tNr $ may be viewed as special cases of (yet to be defined) KR divisors on a RZ space $\CN_{n+1}^{[r_1,r_2]}$ of non-maximal parahoric level. Finally, we observe that the image of the composition $ \Mnr \to   \CN_{n+1}^{[r,\varepsilon]}\to \CN_{n+1}^{[r]}$ is a KR divisor which admits a decomposition as a sum of two (Cartier) divisors. This decomposition leads naturally to closed formal subschemes $ \wt \CM_n^{[r],\pm }$ and to refinements of AT conjectures, explaining all cases in the table.

Note that there are more AT conjectures if we allow ourselves to take a certain class of resolutions of singularities, for example  those formulated (and proved!) by Z. Zhang \cite{ZZha}.}
 
 Let us outline the layout of the paper. The paper consists of two parts. In the first part, we discuss the quasi-canonical arithmetic fundamental lemma, the graph version of the AT conjecture of type $(r, 0)$ and the AT conjectures,  and the evidence for them. In the second part, we discuss the geometry of $\tn$ and the structure of the special divisor $\CZ(u)$ for a vector $u$ of length $\varpi$. The two parts are independent of each other (of course, there are results  in the case $r=1$ in the first part which apply only due to the results in the second part). 

In more detail, the layout of part 1 of the paper is as follows. In \S \ref{sec:geometric-side}, the relevant RZ spaces are introduced and the intersection numbers appearing on the geometric side are defined. In \S \ref{sec:analytic-side}, the analytic side is detailed (transfer, matching, etc.). In \S \ref{s:FLplusAFL} we recall the FL and the AFL. In \S \ref{sec:transfer}, we construct functions which have the correct transfer for the graph version of the arithmetic transfer conjectures. Using these functions, we formulate and prove in \S \ref{s:qcAFL}  the quasi-canonical FL and AFL. In \S \ref{s:graphv}, we deduce the graph version of the AT conjecture. The sections \ref{s:ATCeven}--\ref{s:ATCgen} are devoted to the AT conjecture. In \S \ref{s:ATCeven} we construct a function with the correct transfer and formulate the AT conjecture  when $r$ is an even integer. In \S \ref{s:ATCodd} we do the same for the case when $r$ is odd and equals $n$. In \S \ref{s:ATCgen}, we formulate an AT conjecture in a vague form, for arbitrary  $r$. 

Part 2 of the paper starts with \S \ref{s:spaceN},  in which the space $\wt{\CN}_n^{[1]}$ is introduced, and in which  our results concerning its  geometric structure  are formulated. The last two sections are devoted to the proofs of these results. 

We thank Zhiyu Zhang for his comments on an earlier version of the paper. We thank SLMath for its hospitality to all three of us during the Spring 2023 semester on ``Algebraic cycles, L-values, and Euler systems"  when part of this work was done. We thank the referee for his/her comments. WZ thanks the department of Mathematics at Harvard University for their hospitality where part of the paper was written during his visit in Spring 2024. CL was supported by the NSF grant DMS \#2101157.  MR was supported by the Simons foundation. WZ was supported by the NSF grant DMS \#1901642 and the Simons Foundation.

\section{Notations}

Let $p>2$ be a prime.  Let $F_0$ be a finite extension of $\mathbb{Q}_p$, with ring of integers $O_{F_0}$, residue field $k=\mathbb{F}_q$ of size $q$, and uniformizer $\varpi$. Let $F$ be the unramified quadratic extension of $F_0$, with ring of integers $O_{F}$ and residue field $k_F$. Let $\val:F\to \BZ\cup\{\infty\}$ be the valuation on $F$. Let $|\cdot|_F:F\rightarrow \mathbb{R}_{\ge0}$ (resp. $|\cdot|: F_0\rightarrow \mathbb{R}_{\ge0}$) be the normalized absolute value on $F$ (resp. $F_0$). Let $\eta=\eta_{F/F_0}: F_0^\times \rightarrow\{\pm1\}$ be the quadratic character associated to $F/F_0$.  Let $\sigma$ be the nontrivial Galois automorphism of $F/F_0$. Let $\Fb$ be the completion of the maximal unramified extension of $F$, and $\OFb$ its ring of integers, and $\bar k$ its residue field. Fix $\delta\in O_F^\times$ such that $\sigma(\delta)=-\delta$.

Let $W$ be a (non-degenerate) $F/F_0$-hermitian space with hermitian form $(\ ,\ )$. We write $\val(x):=\val((x,x))$ for any $x\in W$. Recall that a (non-degenerate) $F/F_0$-hermitian space is determined up to isomorphism by its dimension $n$ and its discriminant $\disc(W)=(-1)^{{n \choose 2}}\det(W)\in F_0^\times/\mathrm{Nm}_{F/F_0}F^\times$ (\cite[Theorem 3.1]{J}). We say $W$ is \emph{split} if $\disc(W)=1\in F_0^\times/\mathrm{Nm}_{F/F_0}F^\times$, and \emph{nonsplit} otherwise.

Let $L\subseteq W$ be an $O_F$-lattice of rank $n$. We denote by $L^\vee$ its dual lattice under $(\ ,\ )$. We say that $L$ is \emph{integral} if $L\subseteq L^\vee$. If $L$ is integral, define its \emph{fundamental invariants} to be the unique sequence of integers $(a_1,\ldots, a_n)$ such that $0\leq a_1\le \cdots\le a_n$, and $L^\vee/L\simeq \oplus_{i=1}^n O_F/\varpi^{a_i}$ as $O_F$-modules; define its \emph{valuation} to be $\val(L)\coloneqq \sum_{i=1}^na_i$; and define its \emph{type}, denoted by $t(L)$, to be the number of nonzero terms in its fundamental invariants $(a_1,\ldots, a_n)$.

We say $L$ is \emph{minuscule} or a \emph{vertex lattice} if it is integral and $L^\vee\subseteq \varpi^{-1}L$. Note that $L$ is a vertex lattice of type $t$ if and only if it has fundamental invariants $(0^{(n-t)},1^{(t)})$, if and only if $L\subseteq^t L^\vee\subseteq \varpi^{-1}L$, where $\subseteq^t$ indicates that the $O_F$-colength is equal to $t$. The set of vertex lattices of type $t$ (resp. of type $\ge t$, resp. all vertex lattices) in $W$ is denoted by $\Ver^t(W)$ (resp. $\Ver^{\ge t}(W)$, resp. $\Ver(W)$). We say $L$ is \emph{self-dual} if $L=L^\vee$, or equivalently $L$ is a vertex lattice of type 0. We say $L$ is \emph{almost self-dual} if $L$ is a vertex lattice of type 1.  Since $F/F_0$ is unramified, if $W$ is split then $\val(L)$ is even, any vertex lattice has even type and $W$ contains a self-dual lattice; if $W$ is nonsplit then $\val(L)$ is odd, any vertex lattice has odd type and $W$ contains an almost self-dual lattice.

Let $X$ be a formal scheme.   For closed formal subschemes $\CZ_1,\cdots,\CZ_m$ of $X$, denote by $\cup_{i=1}^m\mathcal{Z}_i$ the formal scheme-theoretic union, i.e., the closed formal subscheme with ideal sheaf $\cap_{i=1}^m\mathcal{I}_{\mathcal{Z}_i}$, where $\mathcal{I}_{\mathcal{Z}_i}$ is the ideal sheaf of $\mathcal{Z}_i$.  A closed formal subscheme on $X$ is called a Cartier divisor if it is defined by an {\em invertible} ideal sheaf.

 When $X$ is noetherian and $Y$ is a closed formal subscheme, denote by $K_0^Y(X)$ the Grothendieck group (modulo quasi-isomorphisms) of finite complexes of coherent locally free $\mathcal{O}_X$-modules, acyclic outside $Y$ (i.e., the homology sheaves are formally supported on $Y$). As defined in \cite[(B.1), (B.2)]{Zha21},  denote by $\mathrm{F}^i K_0^Y(X)$ the (descending) codimension filtration on $K_0^Y(X)$, 
and denote by $\Gr^{i}K_0^Y(X)$ its $i$-th graded piece. As in \cite[App. B]{Zha21}, the definition of $K_0^Y(X)$, $\mathrm{F}^i K_0^Y(X)$ and $\Gr^{i}K_0^Y(X)$ can be extended to locally noetherian formal schemes $X$ by writing $X$ as an increasing union of open noetherian formal subschemes. Similarly, we let $K_0'(X)$ denote the Grothendieck group of coherent sheaves of $\mathcal{O}_X$-modules. Now let $X$ be regular. Then there is a natural isomorphism $K_0^Y(X)\simeq K_0'(Y)$. For closed formal subschemes $\CZ_1,\cdots,\CZ_m$ of $X$, denote by $\CZ_1\cap^\BL_X\cdots\cap^\BL_X\CZ_m$ (or simply $\CZ_1\cap^\BL\cdots\cap^\BL\CZ_m$ if the ambient space is clear) the derived tensor product $\CO_{\CZ_1}\otimes^\BL_{\CO_X}\cdots \otimes^\BL_{\CO_X}\CO_{\CZ_m}$, viewed as an element in $K_0^{\CZ_1\cap\cdots\cap \CZ_m}(X)$.

For $\mathcal{F}$ a finite complex of coherent $\CO_X$-modules, we define its Euler--Poincar\'e characteristic $$\chi(X, \mathcal{F}):=\sum_{i,j}(-1)^{i+j}\length_{\OFb}H^i(X,H_j(\mathcal{F}))$$
if the lengths are all finite. Assume that $X$ is regular with pure dimension $n$. If $\mathcal{F}_i\in \mathrm{F}^{r_i}K_0^{\mathcal{Z}_i}(X)$ with $\sum_i r_i\geq n$, then by \cite[(B.3)]{Zha21} we know that $\chi(X, \bigotimes_i^\mathbb{L}\mathcal{F}_i)$ depends only on the image of $\mathcal{F}_i$ in $\Gr^{r_i}K_0^{\mathcal{Z}_i}(X)$.

For two formal schemes $X,Y$ over $\Spf\OFb$, write $X\times Y:=X\times_{\Spf\OFb}Y$ for short. 

For an algebraic variety $Y$ over a finite extension $F$ of $\BQ_p$, we write $C_c^\infty(Y)$ for $C_c^\infty(Y(F))$.

\part{AT conjectures and the quasi-canonical AFL}

\section{The geometric side}\label{sec:geometric-side}

In this section,  we introduce the relevant RZ spaces and then define the intersection numbers appearing on the geometric side of the arithmetic transfer conjectures.

\subsection{Rapoport--Zink spaces $\Nn$ of self-dual level}\label{sec:rapoport-zink-spaces}
Let $S$ be a $\Spf \OFb$-scheme. Consider a triple $(X, \iota,\lambda)$ where
\begin{altenumerate}
\item $X$ is a formal $\varpi$-divisible $O_{F_0}$-module over $S$ of relative height $2n$ and dimension $n$,
\item $\iota: O_F\rightarrow\End(X)$ is an action of $O_F$ extending the $O_{F_0}$-action and satisfying the Kottwitz condition of signature $(1,n-1)$: for all $a\in O_F$, the characteristic polynomial of $\iota(a)$ on $\Lie X$ is equal to $(T-a)(T-\sigma(a))^{n-1}\in \mathcal{O}_S[T]$,
\item $\lambda: X\rightarrow X^\vee$ is a principal polarization on $X$ whose Rosati involution induces the automorphism $\sigma$ on $O_F$ via $\iota$.
\end{altenumerate}

Up to $O_F$-linear quasi-isogeny compatible with polarizations, there is a unique such triple $(\mathbb{X}, \iota_{\mathbb{X}}, \lambda_{\mathbb{X}})$ over $S=\Spec \kb$, where $\BX$ is isoclinic. Let $\Nn=\mathcal{N}_{F/F_0, n}$ be the (relative) \emph{unitary Rapoport--Zink space of self-dual level}, which is a formal scheme over $\Spf\OFb$ representing the functor sending each $S$ to the set of isomorphism classes of tuples $(X, \iota, \lambda, \rho)$, where the \emph{framing} $\rho: X\times_S \bar S\rightarrow \mathbb{X}\times_{\Spec \kb}\bar S$ is an $O_F$-linear quasi-isogeny of height 0 such that $\rho^*((\lambda_\mathbb{X})_{\bar S})=\lambda_{\bar S}$. Here $\bar S\coloneqq S_{\kb}$ is the special fiber.

The Rapoport--Zink space $\Nn$ is formally locally of finite type and formally smooth of relative dimension $n-1$ over $\Spf \OFb$ (\cite{RZ96}, \cite[Prop. 1.3]{Mihatsch2016}).

\subsection{The hermitian space $\mathbb{V}_n$}\label{sec:herm-space-mathbbv}

Let $\mathbb{E}$ be the formal $O_{F_0}$-module of relative height 2 and dimension 1 over $\Spec \kb$. Then $D\coloneqq \End_{O_{F_0}}^\circ(\mathbb{E}):=\End_{O_{F_0}}(\mathbb{E}) \otimes \mathbb{Q}$ is the quaternion division algebra over $F_0$. We fix an $F_0$-embedding $\iota_\mathbb{E}:F\rightarrow D$, which makes $\mathbb{E}$ into a formal $O_F$-module of relative height 1. We fix an $O_{F_0}$-linear principal polarization $\lambda_\mathbb{E}: \mathbb{E}\xrightarrow{\sim} \mathbb{E}^\vee$. Then $(\mathbb{E}, \iota_\mathbb{E},\lambda_\mathbb{E})$ is a hermitian $O_F$-module of signature $(1,0)$. We have $\mathcal{N}_1\simeq \Spf \OFb$ and there is a unique lifting (\emph{the canonical lifting}) $\mathcal{E}$ of the formal $O_F$-module $\mathbb{E}$ over $\Spf \OFb$, equipped with its $O_F$-action $\iota_\mathcal{E}$, its framing $\rho_\mathcal{E}: \mathcal{E}_{\kb}\xrightarrow{\sim}\mathbb{E}$, and its principal polarization $\lambda_\mathcal{E}$ lifting $\rho_\mathcal{E}^*(\lambda_\mathbb{E})$. Define $\barE$ to be the same $O_{F_0}$-module as $\mathbb{E}$ but with $O_F$-action given by $\iota_{\barE}\coloneqq \iota_\mathbb{E}\circ \sigma$, and $\lambda_{\barE}\coloneqq \lambda_{\mathbb{E}}$, and similarly define $\bar{\mathcal{E}}$ and $\lambda_{\bar{\mathcal{E}}}$.

Denote by $\mathbb{V}_n=\Hom_{O_F}^\circ(\barE, \mathbb{X})$ be the space of special quasi-homomorphisms. Then $\mathbb{V}_n$ carries a $F/F_0$-hermitian form: for $x,y\in \mathbb{V}_n$, the pairing $(x,y)\in F$ is given by  $$ (\barE\xrightarrow{x} \mathbb{X}\xrightarrow{\lambda_\mathbb{X}} {\mathbb{X}^\vee}\xrightarrow{y^\vee} \barE^\vee\xrightarrow{\lambda_\mathbb{E}^{-1}}\barE)\in\End_{O_{F}}^\circ(\barE)=\iota_\barE(F)\simeq F.$$ The hermitian space $\mathbb{V}_n$ is the unique (up to isomorphism) non-degenerate non-split $F/F_0$-hermitian space of dimension $n$. The unitary group $\U(\mathbb{V}_n)(F_0)$ acts on the framing hermitian $O_F$-module $(\mathbb{X}, \iota_\mathbb{X},\lambda_\mathbb{X})$ (via the identification in \cite[Lem. 3.9]{Kudla2011}) and hence acts on the Rapoport--Zink space $\mathcal{N}_n$ via $g(X,\iota,\lambda,\rho)=(X, \iota,\lambda, g\circ\rho)$ for $g\in \U(\mathbb{V})(F_0)$.

For any $0\ne x\subseteq \mathbb{V}_n$, define the \emph{Kudla--Rapoport divisor} or \emph{special divisor} $\mathcal{Z}(x)\subseteq \Nn$ to be the closed formal subscheme which represents the functor sending each $S$ to the set of isomorphism classes of tuples $(X, \iota, \lambda,\rho)$ such that the quasi-homomorphism $$\rho^{-1}\circ x\circ \rho_{\bar{\mathcal{E}}}: \bar{\mathcal{E}}_S\times_S\bar S\xrightarrow{\rho_{\bar{\mathcal{E}}}} \bar{\mathbb{E}}\times_{\Spec\kb}\bar S\xrightarrow{x}\mathbb{X}\times_{\Spec\kb}\bar S\xrightarrow{\rho^{-1}} X\times_S \bar S$$  extends to a homomorphism $\bar{\mathcal{E}}_S\rightarrow X$ (\cite[Def. 3.2]{Kudla2011}). Then $\mathcal{Z}(x)$ only depends on the $O_F$-lattice $\langle x\rangle$ spanned by $x$, and it is a Cartier divisor on $\Nn$ and is flat over $\Spf\OFb$ (\cite[Prop.~3.5]{Kudla2011}). In the body of the paper, the index $n$ of the RZ-space is in fact often replaced by $n+1$.

\subsection{Rapoport--Zink spaces $\Nr$ of maximal parahoric level}\label{sec:rapoport-zink-spaces-1}

Let $r$ be an integer such that $0\le r\le n$. Let $S$ be a $\Spf \OFb$-scheme. Consider a triple $(Y, \iota,\lambda)$ analogous to \S\ref{sec:rapoport-zink-spaces} but the principal polarization $\lambda$ is replaced by a polarization $\lambda: Y\rightarrow Y^\vee$ such that $\ker \lambda\subseteq Y[\varpi]$ and has order $q^{2r}$. Up to $O_F$-linear quasi-isogeny compatible with polarizations, there is a unique such triple $(\BY, \iota_{\BY}, \lambda_{\BY})$ over $S=\Spec \kb$ such that $\BY$ is isoclinic. Let $\Nr$ be the (relative) \emph{unitary Rapoport--Zink space of maximal parahoric level}, which is a formal scheme over $\Spf\OFb$ representing the functor sending each $S$ to the set of isomorphism classes of tuples $(Y, \iota, \lambda, \rho)$, where the \emph{framing} $\rho: Y\times_S \bar S\rightarrow \BY\times_{\Spec \kb}\bar S$ is an $O_F$-linear quasi-isogeny of height 0 such that $\rho^*((\lambda_\BY)_{\bar S})=\lambda_{\bar S}$.

The Rapoport--Zink space $\Nr$ is formally locally of finite type, regular, of relative dimension $n-1$ and of semistable reduction over $\Spf \OFb$ (\cite{RZ96}, \cite{Go}, \cite[Thm. 1.2]{Cho2018}). By definition $\Nn^{[0]}\simeq \Nn$ for $r=0$.

Analogous to \S\ref{sec:herm-space-mathbbv}, denote by $\mathbb{W}_n=\BW_n^{[r]}=\Hom_{O_F}^\circ(\barE, \BY)$  the space of special quasi-homomorphisms. Then $\mathbb{W}_n$ carries a $F/F_0$-hermitian form: for $x,y\in \mathbb{W}_n$, the pairing $(x,y)\in F$ is given by  $$ (\barE\xrightarrow{x} \BY\xrightarrow{\lambda_\BY} {\BY^\vee}\xrightarrow{y^\vee} \barE^\vee\xrightarrow{\lambda_\mathbb{E}^{-1}}\barE)\in\End_{O_{F}}^\circ(\barE)=\iota_\barE(F)\simeq F.$$ The hermitian space $\mathbb{W}_n^{[r]}$ is the unique (up to isomorphism) non-degenerate split (resp. nonsplit) $F/F_0$-hermitian space of dimension $n$ if $r$ is odd (resp. $r$ is even). Analogous to \S\ref{sec:herm-space-mathbbv}, the unitary group $\U(\mathbb{W}_n)(F_0)$ acts on the framing hermitian $O_F$-module $(\BY, \iota_\BY,\lambda_\BY)$ and hence acts on the Rapoport--Zink space $\Nr$ via $g(Y,\iota,\lambda,\rho)=(Y, \iota,\lambda, g\circ\rho)$ for $g\in \U(\mathbb{W})(F_0)$. Note that $\CN_n^{[0]}=\CN_n$ and $\BW_n^{[0]}=\BV_n$. 

\subsection{Rapoport--Zink spaces $\Nrs$ of parahoric level} 

 Let $0\le s\le r\le n$ be integers of the same parity. Let $\BY^{[r]}$ (resp. $\BY^{[s]}$) be the framing $p$-divisible groups of $\Nr$ (resp. $\Ns$). Fix an $O_F$-linear isogeny  $\alpha: \BY^{[r]}\rightarrow \BY^{[s]}$ compatible with polarizations such that $\ker \alpha\subseteq \BY^{[r]}[\varpi]$ and has degree $q^{r-s}$. The existence of $\alpha$ follows from Dieudonn\'e theory. 

Consider the functor sending a $\Spf \OFb$-scheme $S$ to the set of isomorphism classes of tuples $(Y^{[r]},\iota^{[r]},  \lambda^{[r]},\rho^{[r]}, Y^{[s]},\iota^{[s]},\lambda^{[s]},\rho^{[s]})$, where
\begin{itemize}
\item $(Y^{[r]},\iota^{[r]},\lambda^{[r]},\rho^{[r]})\in \Nr(S)$,
\item $(Y^{[s]},\iota^{[s]},\lambda^{[s]},\rho^{[s]})\in \Ns(S)$,
\end{itemize}
such that $(\rho^{[s]})^{-1}\circ\alpha\circ \rho^{[r]}: Y^{[r]} \times_S \bar S\rightarrow Y^{[s]}\times_S\bar S$ lifts to an isogeny $\wit\alpha: Y^{[r]}\rightarrow Y^{[s]}$. Note that if $\wit\alpha$ exists then it is unique  and  $\ker \wit\alpha\subseteq Y^{[r]}[\varpi]$ and has degree $q^{r-s}$. This functor is represented by a formal scheme $\Nrs$ known as the (relative) \emph{unitary Rapoport--Zink space of parahoric level}. The Rapoport--Zink space $\Nrs$ is formally locally of finite type, regular, of relative dimension $n-1$ and of semistable reduction over $\Spf \OFb$ (\cite{RZ96}, \cite{Go}). By definition there are natural projections 
\begin{equation}\label{eq Nrs}
\xymatrix{&\Nrs \ar[rd]  \ar[ld] &\\ \Nr &&  \Ns.}
\end{equation}

\subsection{The space $\tNr$} Let $0\le r\le n$. Set
\begin{equation}\label{eq:epsilon}
\varepsilon=
\begin{cases}
  0, & r \text{ is even}, \\
  1, & r\text{ is odd}.
\end{cases}  
\end{equation}
 Let $\BY$ (resp. $\mathbb{X}$) be the framing $p$-divisible group of $\Nr$ (resp. $\N$). Then we may choose $(\BY\times\barE, \iota_\BY\times \iota_\barE,\lambda_\BY\times \varpi^\varepsilon \lambda_{\barE})$ (resp. $(\BX, \iota_\BX, \lambda_\BX)$) as a framing object of $\Nrp$ (resp. $\N$).   We have a natural closed immersion
 \begin{equation}\label{em r r+}
\Nr\hookrightarrow \Nrp, \quad (Y, \iota, \lambda,\rho)\mapsto (Y\times \bar{\mathcal{E}}_S, \iota\times \iota_{\bar{\mathcal{E}}_S}, \lambda \times \varpi^\varepsilon\lambda_{\bar{\mathcal{E}}_S}, \rho\times \rho_{\bar{\mathcal{E}}_S}).
\end{equation}
Consider  $\Nrzp$ by fixing an $O_F$-linear isogeny 
\begin{equation}\label{defalpha}
\alpha\colon \BY\times\barE\to \BX,
\end{equation}
such that $\ker\alpha\subset (\BY\times\barE)[\varpi]$ and $\alpha^*(\lambda_{\BX})=\lambda_\BY\times \varpi^\varepsilon\lambda_{\barE}$. We define the formal scheme $\tNr$ by the cartesian diagram
\begin{equation}
  \label{eq:tNr1}
  \begin{aligned}
  \xymatrix{\tNr \ar@{^(->}[r] \ar[d] \ar@{}[rd]|{\square}  & \Nrzp \ar[d] \\ \Nr \ar@{^(->}[r] & \Nrp.}
  \end{aligned}
\end{equation}
 By construction, $\tNr$ is the closed formal subscheme of $\Nr\times \N$ parameterizing tuples $$(Y, \iota_Y, \lambda_Y,  \rho_Y,  X, \iota_X, \lambda_X, \rho_X)$$ such that $\rho_X^{-1}\circ \alpha\circ (\rho_Y\times \rho_{\bar{\mathcal{E}}_S})$ lifts to an isogeny $\wit\alpha\colon Y\times\ov\CE\to X$. Then $\wit\alpha$ is uniquely determined, has degree $q^{r+\varepsilon}$ and $\ker\wit\alpha\subset (Y\times\ov\CE)[\varpi]$. We denote by  $\pi_1, \pi_2$ the two natural projections, 
 \begin{equation}\label{pi1pi2}
 \begin{aligned}
 \xymatrix{&\tNr \ar[rd]^{\pi_2}  \ar[ld]_{\pi_1} &\\ \Nr &&  \N.}
 \end{aligned}
 \end{equation}
Here $\pi_1$, resp. $\pi_2$, maps the tuple $(Y, \iota_Y, \lambda_Y,  \rho_Y,  X, \iota_X, \lambda_X, \rho_X)$ to the first half, resp. the second half, of the tuple.  Both projection maps  $\pi_1, \pi_2$ are proper.

The structure of $\tNr$ is quite mysterious in general. The following conjecture concerns the local structure. 
\begin{conjecture}\label{conjreg}
The formal scheme $\tNr$ is regular,  with special fiber a \emph{tame divisor with normal crossings}, i.e. for every $x\in \tNr(\kb)$, there exists a regular system of parameters $X_1, \ldots , X_{n}$ for the complete local ring $\hat\CO_{\tNr, x}$ such that $\varpi=\prod X_1^{m_1}\cdot\ldots \cdot X_{n}^{m_{n}}$, where the $m_i$ are prime to $p$.
\end{conjecture}

\begin{proposition}\label{regreven}
  Conjecture \ref{conjreg} holds when $r$ is even.
\end{proposition}

\begin{proof}
  In this case $\varepsilon=0$ and we have $\tNr\simeq \Nrz$ is regular and has semistable reduction (using the local model diagram which connects the question to the analogous problem for the local model for $(\GL_n, \varpi_1^\vee)$, solved in \cite{Go}).
\end{proof}

Conjecture \ref{conjreg} seems more difficult when $r$ is odd. We will study the special case $r=1$ in more detail and prove in   Theorem \ref{conj:KRSZ} a more precise result regarding the structure of $\tN$. This was conjectured by Kudla and Rapoport, cf. \cite{Kudla2012}, see also \cite[Conj. 10.4.1]{LZ22}. In particular, when $r=1$ we will prove Conjecture \ref{conjreg}  in Theorem \ref{thm:singtN}. 

\subsection{Arithmetic intersection number for AT of type $(r,0)$}\label{sec:arithm-inters-numbAT}

The morphisms $\pi_1$ and $\pi_2$ in \eqref{pi1pi2} combine to give a closed embedding,
\[
\tNr\hookrightarrow \Nr\times\N .
\]
Since $\Nr$ is regular and $\N$ is formally smooth over $\Spf\OFb$, we know that $\Nr\times \N$ is regular. Hence it makes sense to define arithmetic intersection numbers of closed formal subschemes on the ambient space $\Nr\times \N$,
\begin{equation}
  \label{eq:Int1}
  \left\langle \tNr, g\tNr\right\rangle_{\Nr\times\N}:=\chi(\Nr\times \N, \tNr\cap^\BL g\tNr ),
\end{equation}
 where $g\in \U(\mathbb{W}_n^{[r]})(F_0)\times \U(\mathbb{V}_{n+1})(F_0)$. The arithmetic intersection number is finite as long as $\tNr\cap g\tNr$ is a proper scheme over $\Spf\OFb$, which is the case if $g$ is regular semisimple by the standard argument, cf. \cite[proof of Lem. 6.1]{Mihatsch2016}. We will formulate  arithmetic transfer conjectures (Conjectures \ref{conjreven}, \ref{conjrodd} and \ref{conjrodd all})  relating (\ref{eq:Int1}) to the derivative of certain orbital integrals.

\subsection{Arithmetic intersection number for the graph variant of AT of type $(r, 0)$}\label{sec:vari-arithm-intersGV} We also consider a variant of the arithmetic intersection number (\ref{eq:Int1}) by replacing the factor $\Nr$ in the regular ambient space $\Nr\times \N$ by the related space  $\tNr$.  We have the natural closed immersion $$ \tNr\hookrightarrow \tNr\times\N ,$$  given by the graph of  $\pi_2$.

 Let $g \in \U(\mathbb{W}_n^{[r]})(F_0)\times \U(\mathbb{V}_{n+1})(F_0)$. The  variant of (\ref{eq:Int1}) we consider is
\begin{equation}
  \label{eq:Int3}
  \left\langle \tNr, g\tNr\right\rangle_{\tNr\times\N}:=\chi(\tNr\times \N, \tNr\cap^\BL g\tNr).
\end{equation}
Note that the diagonal action of the group $\U(\mathbb{W}_n^{[r]})(F_0)$ on $\Nr\times\N$ stabilizes $\tNr$ and hence $\U(\mathbb{W}_n^{[r]})(F_0)$ acts on $\tNr$.  Thus the group $\U(\mathbb{W}_n^{[r]})(F_0)\times \U(\mathbb{V}_{n+1})(F_0)$ naturally acts on the product $\tNr \times\N$.

Assuming  that $\tNr$ is regular (comp. Conjecture \ref{conjreg}),  and since $\N$ is formally smooth over $\Spf\OFb$, the ambient space  $\tNr\times \N$ is regular. Hence under this assumption the arithmetic intersection number above makes sense, and is  finite when the formal scheme-theoretic intersection $\tNr\cap g\tNr$ is a  proper scheme over $\Spf\OFb$. We will give a formula, relating \eqref{eq:Int3} to the derivative of orbital integrals, cf. Theorem \ref{conjt}.

\subsection{Relation of $\tNr$ with the special divisor $\Zx$}\label{sec:relat-with-spec} Let $\x\in \mathbb{V}_{n+1}=\Hom^\circ_{O_F}(\barE, \BX)$ be the restriction of $\alpha$ in \eqref{defalpha} to the second factor $\barE$.   Let $\Zx\subseteq \N$ be the special divisor associated to $\x$. By construction, the projection $\pi_2: \tNr\rightarrow\N$ factors through $\Zx$. Thus we have a (not necessarily cartesian) commutative diagram
\begin{equation}
  \label{eq:tNr2}
  \begin{aligned}
  \xymatrix{\tNr \ar@{^(->}[r] \ar[d]_{\pi_2}  & \Nrzp \ar[d] \\ \Zx \ar@{^(->}[r] & \N.}
\end{aligned}\end{equation}
Combining (\ref{eq:tNr1}) and (\ref{eq:tNr2}) we obtain a commutative diagram,
\begin{equation}
  \label{eq:tNr3}
\begin{aligned}\xymatrix{& \tNr \ar@{}[d]|{\square} \ar@{^(->}[r] \ar[ld]_(.75){\pi_1} \ar[rd]^(.75){\pi_2} & \Nrzp  \ar[ld]|\hole  \ar[rd]  & \\ \Nr  \ar@{^(->}[r] & \Nrp & \Zx \ar@{^(->}[r] & \N.}   
\end{aligned}\end{equation}
\subsection{Arithmetic intersection number for the quasi-canonical AFL}\label{sec:qcAFL-arithm-inters}

Consider the natural closed immersion
$$
\Zx\hookrightarrow \Zx \times \N 
$$
given by the graph of the inclusion of the special divisor $\CZ(u)$ in $\CN_{n+1}$ Note that $(u,u)=\varpi^\varepsilon$  and hence $\CZ(u)$ is the quasi-canonical divisor of level zero (when $\varepsilon=0$), resp. one (when $\varepsilon=1$). Since $(\x,\x)$ has valuation $\varepsilon\in\{0,1\}$, we know by \cite{Ter} that $\Zx$ is regular (note that $\Zx$ is its own difference divisor).  Since  $\N$ is formally smooth over $\Spf\OFb$, the ambient space $\Zx\times \N$ is regular. Let $\BW_n$ be the orthogonal complement of $u\in\BV_{n+1}$. Then the action of the group $\U(\BW_n)(F_0)\subset\U(\BV_{n+1})(F_0)$ on $\CN_{n+1}$ leaves $\CZ(u)$ invariant. Hence $\U(\BW_n)(F_0)\times\U(\BV_{n+1})(F_0)$ acts on the product $\Zx \times\N$. Therefore we may consider the intersection number of the quasi-canonical divisor $\CZ(u)$ with its translate under  $g\in \U(\BW_n)(F_0)\times\U(\BV_{n+1})(F_0)$,
\begin{equation}
  \label{eq:Int2}
  \left\langle \Zx, g\Zx\right\rangle_{\Zx\times\N}:=\chi(\Zx\times \N, \Zx\cap^\BL g\Zx ).
\end{equation}

When $\varepsilon=0$,  \eqref{eq:Int2} is, under the identification of $\CZ(u)$ with $\CN_n$, the intersection number occurring in the AFL. When  $\varepsilon=1$, we consider \eqref{eq:Int2} as the intersection number occurring in the \emph{quasi-canonical} AFL. This  number  is  finite when the formal scheme-theoretic intersection $\Zx\cap g\Zx$ is a  proper scheme over $\Spf\OFb$. We will give a formula, relating \eqref{eq:Int2} to the derivative of orbital integrals, cf. Theorem \ref{thmzAFL}. When $\varepsilon=0$, this formula reduces to the AFL.

\subsection{The space $\Mnr$} \label{ss:Mnr}
Let $0\le r\le n+1$.  Let $\varepsilon\in\{0,1\}$ be of the same parity as $r$ so that the space  $\CN_{n+1}^{[r,\varepsilon]}$ is defined.
We have a natural embedding (cf. \eqref{em r r+})
$$
 \CN_n^{[0]} \to   \CN_{n+1}^{[\varepsilon]},  \quad (Y, \iota, \lambda,\rho)\mapsto (Y\times \bar{\mathcal{E}}_S, \iota\times \iota_{\bar{\mathcal{E}}_S}, \lambda \times \varpi^\varepsilon\lambda_{\bar{\mathcal{E}}_S}, \rho\times \rho_{\bar{\mathcal{E}}_S}).
$$
Using this embedding,  similar to  \eqref{eq:tNr1}  but using the other projection  $ \CN_{n+1}^{[r,\varepsilon]}\to   \CN_{n+1}^{[\varepsilon]}$ (which maps the tuple $(Y^{[r]},\iota^{[r]},  \lambda^{[r]},\rho^{[r]}, Y^{[\varepsilon]},\iota^{[\varepsilon]},\lambda^{[\varepsilon]},\rho^{[\varepsilon]})$ to the second half of the tuple), we form another Cartesian product, denoted by $\Mnr$,
\begin{equation}
\label{eq Mnr}
  \begin{aligned}
  \xymatrix{\Mnr \ar@{^(->}[r] \ar[d] \ar@{}[rd]|{\square}  &  \CN_{n+1}^{[r,\varepsilon]} \ar[d] \\ 
  \CN_n^{[0]} \ar@{^(->}[r] &   \CN_{n+1}^{[\varepsilon]}.}
  \end{aligned}
\end{equation}
Now, instead of \eqref{pi1pi2} we can consider  
\begin{equation}
 \begin{aligned}
 \xymatrix{&\Mnr \ar[rd]^{\pi'_2}  \ar[ld]_{\pi'_1} &\\ \CN_n^{[0]} &&  \CN^{[r]}_{n+1},}
 \end{aligned}
 \end{equation}
 where the first map $\pi'_1$ is the left vertical map in  \eqref{eq Mnr} and the second map $\pi'_2$ is  the composition $\Mnr  \incl \CN_{n+1}^{[ r,\varepsilon ]} \to   \CN^{[r]}_{n+1}$. Explicitly, fix  an $O_F$-linear isogeny 
\begin{equation}\label{defalpha'}
\alpha\colon \BY'\to \BX\times\barE,
\end{equation}
such that $\ker\alpha\subset \BY'[\varpi]$ and $\alpha^*(\lambda_{\BX}\times \varpi^\varepsilon\lambda_{\barE})=\lambda_{\BY'}$. Then the space $\Mnr$ is the closed formal sublocus of $\CN_n^{[0]}\times \CN^{[r]}_{n+1}$ parametrizing  tuples 
 \[
 (X, \iota_X, \lambda_X, \rho_X, Y', \iota_{Y'}, \lambda_{Y'}, \rho_{Y'})
 \]
 such that $(\rho_X\times \rho_{\bar\CE_S})^{-1}\circ\alpha \circ \rho_{Y'}$ lifts to an isogeny $\tilde\alpha: Y'\to X\times \bar\CE$ .
 
 Let us discuss the geometry of $\Mnr$. We distinguish the case where $r$ is even from the case where $r$ is odd. 
 
 Let $r$ be even. If $r\leq n$, there is a closed embedding $\CN_n^{[0,r]} \hookrightarrow \Mnr$, sending $(Y\to X)$ to $(X, Y\times\bar\CE)$ (we are dropping all auxiliary structure from the notation). Let $\x\in \Hom^\circ_{O_F}( \BY', \barE)$ be the projection of $\alpha$ in \eqref{defalpha'} to the second factor. By taking duals and using the polarization, we may identify $\Hom^\circ_{O_F}( \BY', \barE)$ with $\mathbb{V}_{n+1}$. In this way, we obtain the special vector $u$ of length one   in 
 $\BV_{n+1}$. Recall the $\CZ$-divisor and the $\CY$-divisor $ \CZ(u)^{[r]} \hookrightarrow \CY(u)^{[r]}$ on $\CN_{n+1}^{[r]}$, comp. \cite[\S 5.2]{ZZha}.  They are both relative Cartier divisors. Furthermore, $ \CZ(u)^{[r]}$ is a regular Cartier divisor  isomorphic to $\CN_n^{[r]}$, cf. \cite[Prop. 5.20]{ZZha}).  If $r=0$, then $ \CZ(u)^{[r]}= \CY(u)^{[r]}$. If $r\geq 2$, then  $\CY(u)^{[r]}$ is a reducible Cartier divisor, as follows from the non-emptiness of the difference divisor $\CD(u)=\CY(u)^{[r]}\setminus \CY(u/\varpi)^{[r]}$ (this can be seen for $n=2$ from the theory of quasi-canonical divisors \cite{Kudla2011}, and for $n\geq 3$ from the non-emptiness of the underlying reduced scheme, as  can be checked using the Bruhat--Tits stratification in \cite{Cho2019} and \cite{ZZha}). We obtain a cartesian diagram 
\begin{equation}
\label{eq MnrNr0}
  \begin{aligned}
  \xymatrix{\CN_n^{[0, r]} \ar@{^(->}[r] \ar[d] \ar@{}[rd]|{\square}  &  \Mnr \ar[d] \\ 
  \CZ(u)^{[r]} \ar@{^(->}[r] &   \CY(u)^{[r]}.}
  \end{aligned}
\end{equation}
Here the right downward arrow associates to $(Y'\to X\times \bar\CE)$ the second projection $(Y'\to \bar\CE)$ (this is the analogue of  $\pi_2$ in \eqref{eq:tNr2} for $\Mnr$ instead of $\tNr$). The left downward arrow associates to $(Y\to X)$ the map $(\bar\CE\to Y\times\bar\CE)$. It follows that $\Mnr$ is formally reducible when $r\geq 2$, hence $\Mnr$ is not regular  for $r\geq 2$,  i.e., there is no analogue of Proposition \ref{regreven} for $\Mnr$.

 Write $\CY(u)^{[r]}$ as the sum of effective  Cartier divisors 
 \[
 \CY(u)^{[r]}=\CY(u)^{[r], +}+\CY(u)^{[r], -} ,
 \]
 where $\CY(u)^{[r], +}=\CZ(u)^{[r]}$. Correspondingly, we obtain by pullback of $\CY(u)^{[r], +}=\CZ(u)^{[r]}$, resp. of  $\CY(u)^{[r], -}$  under the map $\Mnr\to \CZ(u)^{[r]}$ the closed formal subschemes  $\wt{\CM}^{[r], +}= \CN_n^{[0, r]}$, resp.   $\wt{\CM}^{[r], -}$. Then $\wt{\CM}^{[r], +}$ is regular and  its generic fiber is the member of the RZ-tower of $\CN_n^{[0]}$ corresponding to the parahoric  subgroup $K_n^{[r], +}$ of \S \ref{ss: cpt open 0 r}. It is conceivable that $\wt{\CM}^{[r], -}$ is also regular and that its generic fiber is the member of the RZ-tower of $\CN_n^{[0]}$ corresponding to the parahoric subgroup $K_n^{[r], -}$  of \S \ref{ss: cpt open 0 r}.

Now let $r$ be even and $r=n+1$. Then we have a natural isomorphism
$\wt{\CM}_n^{[n+1]}\isoarrow \wt{\CN}_n^{[n]}$ making the following diagram commutative, 
\begin{equation}
  \label{relMN}
  \begin{aligned}
  \xymatrix{ \wt{\CM}_n^{[n+1]}\ar[r]^{\simeq} \ar[d]_{\pi_1'\times\pi'_2}  & \wt{\CN}_n^{[n]}\ar[d]^{\pi_1\times \pi_2} \\ \CN_{n}^{[0]}\times \CN_{n+1}^{[n+1]} \ar[r]^{\simeq} & \CN_{n}^{[n]}\times\CN_{n+1}^{[0]}.}
\end{aligned}\end{equation} Here the isomorphisms in the bottom line 
\begin{equation}\label{eq: dual iso}
\CN_{n}^{[n]}\simeq \CN_{n}^{[0]},\text{ resp. }\CN_{n+1}^{[n+1]}\simeq \CN_{n+1}^{[0]}
\end{equation} are the \emph{rescaling isomorphisms} which send e.g. $(X, \iota, \lambda)\in \CN_{n}^{[n]}$ to $(X, \iota, \lambda_0)\in \CN_{n}^{[0]}$, where $\lambda =\varpi\lambda_0$.  The isomorphism  $\wt{\CM}_n^{[n+1]}\isoarrow \wt{\CN}_n^{[n]}$ sends  $(Y'\to X\times \bar\CE)\in \wt{\CM}_n^{[n+1]}$ to $(Y\times \bar\CE\to X')$, where $X'$ is the dual of the rescaled module of $Y'$, and $Y$ is the dual of the rescaled module $X$. 

Now let $r$ be odd. When $r=1$ (hence $\varepsilon=1 $) then, trivially,  $\CN_{n+1}^{[r,\varepsilon]}\simeq \CN_{n+1}^{[1]}$ and $\Mnr\simeq  \CN_n^{[0]}$. In general
\begin{equation}\label{surpr}
\Mnr\simeq \CN_n^{[0, r-1]} ,
\end{equation} such that $\pi'_1:\Mnr\to \CN_n^{[0]}$ coincides with the natural projection $\CN_n^{[0, r-1]}\to \CN_n^{[0]}$. The isomorphism is given by sending $(Y'\to X\times \bar\CE)\in\Mnr$ to the element $(Y\to X)\in \CN_n^{[0, r-1]}$, where $Y\in \CN_n^{[r-1]}$ is the dual of $(Y')^{ \vee}/\im(\bar\CE^\vee\to (Y')^{ \vee})$, and where the isogeny  $Y\to X$ is the dual of the composition $X^\vee\to (Y')^\vee\to Y^{ \vee}$. Note that, by Proposition \ref{regreven}, the isomorphism \eqref{surpr} may also be written as $\Mnr\simeq \wt{\CN}_n^{[r-1]}$, in analogy with the isomorphism $\wt{\CM}_n^{[n+1]}\isoarrow \wt{\CN}_n^{[n]}$ for $n$ odd, occurring in \eqref{relMN}. 

However, for $r$   even  with $r\leq n$, there is no identification of $\Mnr$ with an $\wt{\CN}_n$-space. Maybe there is a better chance with the space $\wt{\CM}^{[r], -}$ introduced above.

In conclusion, when $r$ is odd, then $\Mnr$ is regular  with semi-stable reduction and the generic fiber of $\Mnr$ is a member of the RZ-tower associated to $\CN_n^{[0]}$ (in fact corresponding to the  parahoric subgroup $K_n^{[r]}$, cf. \S \ref{ss: cpt open 0 r}).

{\subsection{Arithmetic intersection number for AT of type $(0, r)$}\label{sec:arithm-inters-numbAT(r,0)}

The arithmetic intersection number in this context is defined as
 \begin{equation}\label{eq Int (0,r)}
\left\langle \Mnr, g\Mnr\right\rangle_{\CN_n^{[0]}\times\CN_{n+1}^{[r]}}:=\chi(\CN_n^{[0]}\times\CN_{n+1}^{[r]}, \Mnr\cap^\BL g\Mnr ) , \quad g\in G_{W_{\epsilon(r)+1}}(F_0)_\rs .
\end{equation}
Here we take the convention that $W_{\epsilon(r)+1}=W_0$ when $\epsilon(r)=1$. 
Note that, by \eqref{relMN}, the intersection number \eqref{eq Int (0,r)},  for $r=n+1$ even,  coincides with the intersection number \eqref{eq:Int1} for $r=n$. 
In the special case $r=1$ (hence $\varepsilon=1 $), we have $\CN_{n+1}^{[r,\varepsilon]}\simeq \CN_{n+1}^{[1]}$ and $\Mnr\simeq  \CN_n^{[0]}$;  in this case, $\Mnr$ is the graph of the closed embedding $\CN_n^{[0]}\hookrightarrow \CN_{n+1}^{[1]}$, and this intersection number  has been considered in \cite{RSZ2}, where an AT statement is proved. See \S\ref{s: case (01)} for more details.

In the case when $r$ is even and $r\leq n$, there are the following variants of \eqref{eq Int (0,r)},
\begin{equation}\label{eq Intvar (0,r)}
 \begin{aligned}
\left\langle \wt\CM_n^{[r],+}, g \wt\CM_n^{[ r],+}\right\rangle_{\CN_n^{[0]}\times\CN_{n+1}^{[r]}}&:=\chi(\CN_n^{[0]}\times\CN_{n+1}^{[r]},  \wt\CM_n^{[ r],+}\cap^\BL g  \wt\CM_n^{[ r],+} ) , \\
\left\langle \wt\CM_n^{[r],-}, g \wt\CM_n^{[ r],-}\right\rangle_{\CN_n^{[0]}\times\CN_{n+1}^{[r]}}&:=\chi(\CN_n^{[0]}\times\CN_{n+1}^{[r],-},  \wt\CM_n^{[ r],-}\cap^\BL g  \wt\CM_n^{[ r],-} ) , \\
\left\langle \wt\CM_n^{[r],+}, g \wt\CM_n^{[ r],-}\right\rangle_{\CN_n^{[0]}\times\CN_{n+1}^{[r]}}&:=\chi(\CN_n^{[0]}\times\CN_{n+1}^{[r]},  \wt\CM_n^{[ r],+}\cap^\BL g  \wt\CM_n^{[ r],-} ) , 
\end{aligned}
\end{equation}
where  $g\in G_{W_1}(F_0)_\rs$.

\section{The analytic side}\label{sec:analytic-side}

\subsection{Groups}\label{ss:gps} We recall the group-theoretic setup of \cite[\S2]{RSZ1} in both homogeneous and inhomogeneous settings. Let $n\geq 1$. In the homogeneous setting, set
\begin{equation}
G':=\Res_{F/F_0}(\GL_{n}\times\GL_{n+1}),
\end{equation}
a reductive algebraic group over $F_0$. Let $W$ be  a $F/F_0$-hermitian space of dimension $n+1$.  Fix $u\in W$  a non-isotropic vector (the {\it special vector}), and let $W^\flat=\langle u\rangle^\perp$. Set
\begin{equation}
G_W=\U(W^\flat)\times \U(W),
\end{equation}
a reductive algebraic group over $F_0$. We have the notion of a {\it regular semi-simple element}, for $\gamma\in G'(F_0)$ and for $g\in G_W(F_0)$ and the notion of {\it matching} $\gamma\leftrightarrow g$, cf. \cite[\S2]{RSZ1}. The  notions of regular semi-simple elements are with respect to the actions of  reductive algebraic groups over $F_0$, namely   $H'_{1, 2}=H_1'\times H_2' :=\Res_{F/F_0} (\GL_{n})\times ( \GL_{n}\times\GL_{n+1})$
 on $G'$, resp. $H_W:=\U(W^\flat)\times \U(W^\flat)$ on $G_W$.
  It is important to note that the first action is arranged \emph{after} the choice of the special vector $u\in W$. The sets of regular semi-simple elements are denoted by $G'(F_0)_\rs$ and $G_W(F_0)_\rs$ respectively. 
  Let $W_0$, $W_1$ be the two isomorphism classes of $F/F_0$-hermitian spaces of dimension $n+1$. Note that, unlike the convention in \cite{RSZ1}, for the moment we do not assume $W_0$ to be split. Take the special vectors  $u_0\in W_0$ and $u_1\in W_1$ to have the same norm (not necessarily a unit). Then we have a bijection of regular semisimple orbits,
\begin{equation}\label{decrsshom}
 \xymatrix{ \bigl[G_{W_0} (F_0)_\rs\bigr] 
 \bigsqcup\,\bigl[ G_{W_1} (F_0)_\rs\bigr]  \ar[r]^-\sim& \bigr[G'(F_0)_\rs\bigr]} .
\end{equation}

\begin{remark}\label{rem orb}
In our previous work, the norm of the special vector is assumed to be one. In that case we characterize the bijection of orbits in terms of  invariant theory. If the norm of the special vector  is not equal to one, we need to scale the Hermitian form so that the norm of the special vector becomes one and then we can apply the invariant-theoretical characterization of the orbit comparison.
\end{remark}

In the inhomogeneous setting, recall the symmetric space 
\begin{equation}
S_{n+1} = \{ \gamma \in \Res_{F/F_0}\GL_{n+1} \mid \gamma \ov \gamma = 1_{n+1}\} .
\end{equation}
Here $\gamma\mapsto \ov\gamma$ is the natural involution on $\Res_{F/F_0}\GL_{n+1}$ induced by the Galois conjugation relative to $F/F_0$. There is the  map 
 \begin{equation}\label{def:r}
 \fkr: \Res_{F/F_0}\GL_{n+1}\to  S_{n+1}, \quad \gamma\mapsto \gamma\ov \gamma^{-1} ,
 \end{equation}  which induces an isomorphism $(\Res_{F/F_0}\GL_{n+1})/\GL_{n+1}\simeq S_{n+1}.$ We have the notion of a \emph{regular semi-simple element}, for $\gamma\in  S_{n+1}(F_0)$ and for $g\in \U(W)(F_0)$ and the notion of \emph{matching} $\gamma\leftrightarrow g$.  Here $W$ is any hermitian space of dimension $n+1$ and the  notions of regular semi-simple elements are with respect to the conjugation actions of $H':=\GL_{n}$ on $S$, resp., of $H:=\U(W^\flat)$ on $\U(W)$. The sets of regular semi-simple elements are denoted by $ S_{n+1}(F_0)_\rs$ and $\U(W)(F_0)_\rs$ respectively. The inhomogeneous version of the bijection \eqref{decrsshom} is
\begin{equation}
 \xymatrix{ \bigl[ \U(W_0) (F_0)_\rs\bigr]
 \bigsqcup\, \bigl[ \U(W_1) (F_0)_\rs\bigr]  \ar[r]^-\sim& \bigr[S_{n+1}(F_0)_\rs\bigr]} .
\end{equation}

We use the same notation for the map  
$\fkr: \GL_{n+1}(F)\to  S_{n+1}(F_0)$ introduced in \eqref{def:r}  and the map 
\begin{equation}
\fkr: G'(F_0)\to  S_{n+1}(F_0) ,
\end{equation}
 obtained by precomposing $\fkr$ with $G'(F_0)\to \GL_{n+1}(F)$ given by $(\gamma_1,\gamma_2)\mapsto \gamma_1^{-1}\gamma_2$. Then $G'(F_0)_\rs=\fkr^{-1}( S_{n+1}(F_0)_\rs)$, and $\gamma\in G'(F_0)_\rs$ matches $g=(g_1, g_2)\in G_{W_i}(F_0)_\rs$ if and only if $\fkr(\gamma)\in  S_{n+1}(F_0)_\rs$ matches $g_1^{-1}g_2\in\U(W_i)(F_0)_\rs$. 

 We will also need the semi-Lie version of this set-up in the inhomogeneous version. Set 
 \begin{equation}\label{defW}
 W'=W'_{n+1}=F_0^{n+1}\times (F_0^{n+1})^* .
 \end{equation}
 There is the notion of a regular semi-simple element, for $(\gamma, v)\in  (S_{n+1}\times W')(F_0)$ and for $(g, u)\in (\U(W)\times W)(F_0)$, cf. \cite[\S 2]{Zha21}. These notions are with respect to  group actions of $\GL_{n}$, resp. $\U(W^\flat)$. Again, there is a notion of matching between elements of $( S_{n+1}\times W')(F_0)_{\rm rs}$ and $(\U(W)\times W)(F_0)_{\rm rs}$.

\subsection{Orbital integrals}\label{s:orb}

  We recall the orbital integrals in both homogeneous and inhomogeneous settings, following \cite[\S5]{RSZ1}. 
  Now we assume that $F/F_0$ is an {\em unramified} extension of non-archimedean local fields. We then have a unique extension of $\eta$ to an unramified quadratic character $\wt\eta:F^\times\to \{\pm1\}$. We also define the twist 
 $$
 \wt\eta_s(z)=\wt\eta(z)|z|_F^{s/2},\quad z\in F^\times,
 $$
 for a complex parameter $s\in \mathbb{C}$.

We first introduce a transfer factor on $G'$ and $ S_{n+1}$ respectively. For $\gamma\in  S_{n+1}(F_0)_{\rs}$, we define
\begin{equation}\label{def del+}
\Delta^+(\gamma)=\det( (\gamma^i e_{n+1})_{i=0}^{n} ) ,
\end{equation}
where 
\begin{equation}\label{def e}
e_{n+1}=\,^t(0,\cdots,0,1)\in M_{(n+1)\times 1}(F_0).
\end{equation}
For $s\in \mathbb{C}$, we define 
the transfer factor on the symmetric space $ S_{n+1}$ as\footnote{Here our transfer
factor, when specialized to $s=0$, coincides with \cite[\S2.4]{RSZ2} and \cite[\S2.3]{Zha21} where $\wt\eta(\det(\gamma))=1$ by our choice of $\wt\eta$. } 
\begin{equation}\label{def trans S}
\omega_{ S,s}(\gamma)=\wt\eta_{-s}(\Delta^+(\gamma)),\quad \gamma\in  S_{n+1}(F_0)_{\rs}.
\end{equation}
It satisfies 
$$
\omega_{ S,s}(h^{-1} \gamma h)=\eta_s(h) \omega_{ S,s}(\gamma), \quad h\in \GL_{n}(F_0).
$$
Here we write $\eta_s(h)$ for $\eta_s(\det h)$. 

In the homogeneous case, we define for $\gamma=(\gamma_{n},\gamma_{n+1})\in G'(F_0)_\rs$ and  $s\in \mathbb{C}$,
the transfer factor on the group $G'$ as
\begin{equation}\label{def trans G'}
\omega_{G',s}(\gamma)=\wt \eta^{n}(\gamma_{n}^{-1}\gamma_{n+1})  |\gamma_{n}|_F^{-s}\,\omega_{S,2s}(\fkr(\gamma)).
\end{equation}
Then it is easy to verify
$$
\omega_{G',s}(h_1^{-1}\gamma h_2)= |\det h_1|_F^{s} \eta(h_2)\, \omega_{G',s}(\gamma), \quad (h_1,h_2)\in  (H_1'\times H_2')(F_0).
$$
Here we denote
\[
   \eta(h_2) := \eta(\det h_2')^{n-1} \eta(\det h_2'')^{n}
	\quad\text{for}\quad
	h_2 = (h_2', h_2'')\in H_2'({F_0}) = \GL_{n}(F_0) \times \GL_{n+1}(F_0).
\]

Next we  introduce some \emph{weighted orbital integrals}. In the homogeneous setting, for an element $\gamma\in G'(F_0)_\rs$, for a function $\fp'\in C^\infty_0(G')$ and for a complex parameter $s\in \mathbb{C}$, we define\footnote{Note that, contrary to e.g. \cite{RSZ2} and \cite{Zha21}, we incorporate the transfer factor in the definition of the weighted orbital integral.}
\begin{equation}\label{def Orb s}
   \Orb(\gamma, \fp', s) := \omega_{G',s}(\gamma)
	   \int_{H_{1,2}'(F_0)} \fp'(h_1^{-1}\gamma h_2) \lv\det h_1\rv_F^s \eta(h_2)\, \rd h_1\, \rd h_2,
\end{equation}
where we use fixed Haar measures on $H_1'(F_0)$ and $H_2'(F_0)$ and the product Haar measure on $H_{1,2}'(F_0) = H_1'(F_0) \times H_2'(F_0)$. 
We further define the value and derivative at $s=0$,
\begin{equation}
   \Orb(\gamma, \fp') := \Orb(\gamma, \fp', 0)
	\quad\text{and}\quad
	\del(\gamma, \fp') := \frac{\rd}{\rd s} \Big|_{s=0} \Orb({\gamma},  \fp',s) . 
\end{equation}
The integral defining $\Orb(\gamma,\fp',s)$ is absolutely convergent, and $\Orb(\gamma, \fp')$ and $\del(\gamma, \fp')$ depend only on the orbit of $\gamma$.

Now we turn to the inhomogeneous setting.  For  $\gamma\in  S_{n+1}(F_0)_\rs$,  for a function $\phi'\in C_c^\infty( S_{n+1})$, and for a complex parameter $s\in \BC$, we introduce the \emph{weighted orbital integral}
\begin{align}\label{eqn def inhom}
   \Orb(\gamma,\phi', s) :=\omega_{S,s}(\gamma) \int_{H'(F_0)}\phi'(h^{-1}\gamma h)\lvert \det h \rvert^s \eta(h) \, dh ,
\end{align}
as well as the value and derivative at $s=0$,
\begin{equation}
   \Orb(\gamma,\phi') := \Orb(\gamma,\phi', 0)
   \quad\text{and}\quad
   \del(\gamma,\phi') : = \frac \rd{\rd s} \Big|_{s=0} \Orb(\gamma, \phi',s) . 
\end{equation}
   As in the homogeneous setting, the integral defining $\Orb(\gamma,\phi', s)$ is absolutely convergent, and $ \Orb(\gamma,\phi')$ and $  \del(\gamma,\phi')$ depend only on the orbit of $\gamma$.

Now  let us define orbital integrals on the unitary side. In the homogeneous setting, for $W$ the $(n+1)$-dimensional hermitian space as above, for an element $g \in G_W(F_0)_\rs$, and for a function $f \in C_c^\infty(G_W)$, we  define the orbital integral
\begin{equation}\label{def Orb U}
   \Orb(g, f) := \int_{H_W(F_0)} f(h_1^{-1} g h_2)\, dh_1\, dh_2 .
\end{equation}
Here the Haar measure on $H_W(F_0) = H(F_0) \times H(F_0)$ is the product of two identical Haar measures on $H(F_0)$.

In the inhomogeneous setting, let $G=\U(W)$. We define for $g \in G(F_0)_\rs$ and $f \in C_c^\infty(G)$, 
\begin{equation}
   \Orb(g,f) := \int_{H(F_0)} f(h^{-1}gh) \, dh,
\end{equation}
where we use the same fixed Haar measure on $H(F_0)$.

We also have use for the semi-Lie versions of these orbital integrals. For $(\gamma, w')\in ( S_{n+1}\times W')(F_0)_\rs$ and $\Phi'\in C_c^\infty( S_{n+1}\times W')$ and $s\in\BC$, we set
\begin{equation}\label{def Orb S W'}
   \Orb((\gamma, w'), \Phi', s):= \omega_{S\times W',s}(\gamma, w')\int_{\GL_n(F_0)} \Phi'(h\cdot(\gamma, w'))|\det(h)|^s\eta(h) \, dh,
\end{equation}
 and obtain $\Orb((\gamma, w'), \Phi')$ and $\del((\gamma, w'), \Phi')$ as before. Here the transfer factor $ \omega_{S\times W',s}(\gamma, w')$ is defined
similarly to \eqref{def trans S},
\begin{equation}\label{def trans S W'}
\omega_{S\times W',s}((\gamma,w'),s)=\wt\eta_{-s}(\Delta^+(\gamma,w')),\quad (\gamma,w')\in ( S_{n+1}\times W')(F_0)_{\rs} ,
\end{equation}where $\Delta^+(\gamma,w')$ is defined similarly to
 \eqref{def del+}: for $w'=(e,e^*)\in W'$,
\begin{equation}\label{def del+ W}
\Delta^+(\gamma,w')=\det( (\gamma^i e)_{i=0}^{n} ) .
\end{equation} 
The value of $\omega_{S\times W',s}$ at $s=0$ recovers \cite[eq. (2.17)]{Zha21}. On the unitary side, we have 
 \begin{equation}\label{semiunit}
 \Orb((g, v), \Phi)= \int_{H(F_0)} \Phi(h\cdot (g, v)) \, dh,
, \quad\Phi\in C_c^\infty(\U(W)\times W). 
\end{equation}

Having fixed  the transfer factors, we have the notion of \emph{transfer} between functions $\fp'\in C_c^\infty(G')$ and pairs of functions $(f_0,f_1)\in C_c^\infty(G_{W_0})\times C_c^\infty(G_{W_1})$ (\cite[Def. 2.2]{RSZ2}), and between functions $\phi'\in C_c^\infty( S_{n+1})$ and pairs of functions $(f_0,f_1)\in C_c^\infty(\U(W_0))\times C_c^\infty(\U(W_1))$ (\cite[Def. 2.4]{RSZ2}) and between functions $\Phi'\in C_c^\infty( S_{n+1}\times W')$ and pairs of functions $(\Phi_0, \Phi_1)\in C_c^\infty(\U(W_0)\times W_0)\times C_c^\infty(\U(W_1)\times W_1)$. We recall here the notion of the transfer only in  the homogeneous group setting: $\fp'\in C_c^\infty(G')$ and $(f_0,f_1)\in C_c^\infty(G_{W_0})\times C_c^\infty(G_{W_1})$ are called  transfers of each other if, for all regular semisimple $\gamma \in G'(F_0)$,  we have
\begin{align}\label{eq:trans}
\Orb(\gamma, \fp')= \begin{cases}\Orb(g, f_0), \quad\text{if $\gamma$ matches an element $g\in G_{W_0}(F_0)$}, 
\\
\Orb(g, f_1), \quad \text{if $\gamma$ matches an element $g\in G_{W_1}(F_0)$}.
\end{cases}
\end{align}
The other cases are defined similarly.

\begin{definition}\label{defequ}
  We say two functions $\fp'_1, \fp'_2\in C_c^\infty(G')$ (resp. $\phi'_1, \phi'_2\in C_c^\infty( S_{n+1})$) are \emph{equivalent}, denoted by $\fp'_1\sim\fp'_2$ (resp. $\phi'_1\sim \phi'_2$), if $$\Orb(\gamma, \fp'_1)=\Orb(\gamma, \fp'_2), \quad \text{(resp. $\Orb(\gamma, \phi'_1)=\Orb(\gamma, \phi'_2)$)}$$ for all $\gamma\in G'(F_0)_\rs$ (resp. $\gamma\in  S_{n+1}(F_0)_\rs$).
\end{definition}

    Let us relate functions on $G'(F_0)$ and on $ S_{n+1}(F_0)$. For $\fp'\in C^\infty_c(G')$, define the function $\fp'^\natural$ on $ S_{n+1}(F_0)$ by
    \begin{equation}\label{eqn def wt f}
   \fp'^\natural(\fkr(\gamma)) :=\wt\eta^{n}(\gamma) \int_{\GL_{n}(F)\times \GL_{n+1}(F_0)}\fp'(h^{-1},h^{-1} \gamma h_2)\eta^{n}(h_2)\, \rd h\, \rd h_2,\quad \gamma\in\GL_{n+1}(F).
\end{equation}
  It is easy to see that the right hand side depends only on $\fkr(\gamma)$ rather than on $\gamma$, hence indeed defines a function in $C_c^\infty( S_{n+1})$.
We also define a family version: for $s\in\BC$, 
\begin{equation}\label{eqn def wt f s}
   \fp'^\natural_s(\fkr(\gamma)) :=\wt\eta^{n}(\gamma) \int_{\GL_{n}(F)\times \GL_{n+1}(F_0)}\fp'(h^{-1},h^{-1} \gamma h_2)|\det h|_F^s\eta^{n}(h_2)\, \rd h\, \rd h_2.
\end{equation}
Then
$$
   \fp'^\natural_s(\gamma)|_{s=0}=   \fp'^\natural(\gamma),\quad\gamma\in  S_{n+1}(F_0).
$$

\begin{lemma}\label{lem: Orb G2S}
Let $\fp'\in C^\infty_c(G')$ and $\gamma \in G'(F_0)_\rs$. Then 
$$
\Orb(\gamma,\varphi',s)=\Orb(\fkr(\gamma),\varphi'^\natural_s,2s) .
$$
\end{lemma}

\begin{proof}

By definition \eqref{def Orb s},  for $\gamma=(\gamma_1,\gamma_2)\in G'(F_0)_\rs$, we have
\[
   \Orb\bigl((\gamma_1,\gamma_2), \fp', s\bigr) 
	   = \omega_{G',s}(\gamma)\int_{H_{1, 2}'(F_0)} \fp'(h^{-1}\gamma_1 h_1,h^{-1}\gamma_2 h_2) \lvert\det h\rvert_F^s \eta(h_1)\eta^{n-1}(h_1h_2) \,\rd h\,\rd h_1\,\rd h_2,
\]
where $h\in H_1'(F_0) = \GL_{n-1}(F)$ and $(h_1,h_2)\in H_2'(F_0)=\GL_{n-1}(F_0)\times\GL_n(F_0)$. Replacing $h$ by $\gamma_1h_1h$, we have 
\[
   \Orb\bigl((\gamma_1,\gamma_2),\fp', s\bigr)
	   =\omega_{G',s}(\gamma) |\gamma_1|_F^s\int_{H_{1, 2}'(F_0)} \fp'\bigl(h^{-1},h^{-1}h_1^{-1}(\gamma_1^{-1}\gamma_2) h_2\bigr)   \lvert\det h h_1\rvert_F^s \eta(h_1)\eta^{n-1}(h_1h_2) \,\rd h\,\rd h_1\,\rd h_2.
\]
Comparing with the definition \eqref{eqn def wt f s} of $\fp'^\natural$, we have
\[
   \Orb\bigl((\gamma_1,\gamma_2),\fp', s\bigr)
	   =\omega_{G',s}(\gamma) |\gamma_1|_F^s \wt \eta^{-(n-1)}(\gamma_1^{-1}\gamma_2 ) \int_{\GL_{n-1}(F_)}\fp'^\natural_s( h_1^{-1}\fkr(\gamma) h_1 )   \lvert\det  h_1\rvert_F^s \eta(h_1)\, \rd h_1 .
\]
By the definition of the transfer factor on $G'$ \eqref{def trans G'} and $  \lvert\det  h_1\rvert_F=  \lvert\det  h_1\rvert^2_{F_0}$, the last equation is equal to
\[
\omega_{S,2s}(\fkr( \gamma) ) \int_{\GL_{n-1}(F_)}\fp'^\natural_s( h_1^{-1}\fkr(\gamma) h_1 )   \lvert\det  h_1\rvert^{2s} \eta(h_1)\, \rd h_1,
\]
which is equal to $\Orb(\fkr(\gamma),\fp'^\natural_s,2s)$ by \eqref{eqn def inhom}.
\end{proof}

\begin{remark}
This lemma is a  more general version of \cite[Lem. 3.7.2]{LRZ} (note that in {\it loc. cit} the orbital integral is not normalized by the transfer factor). \end{remark}

\begin{corollary}\label{compOr}
We have
\[
\Orb(\gamma,\varphi')=\Orb(\fkr(\gamma),\varphi'^\natural), \quad \del(\gamma, \fp')=2\del(\fkr(\gamma), \fp'^\natural),\quad  \gamma\in G'(F_0)_\rs.
\]\qed
\end{corollary}
On the unitary side, for $f\in C_c^\infty(G_W)$, we define a function $f^\natural\in C_c^\infty(\U(W))$,
\begin{align}\label{def:wt varphi}
f^\natural(g)=\int_{U(W^\flat)}f(h,hg)\, dh, \quad g\in\U(W)(F_0).
\end{align}
Then it is easy to see that 
$$
\Orb((g_1,g_2), f)=\Orb(g_1^{-1}g_2, f^\natural),  \quad (g_1, g_2)\in\U(W^\flat)(F_0)\times\U(W)(F_0).
$$
From Corollary \ref{compOr} we therefore deduce the following statement.
\begin{corollary}\label{transnat}
 The function $\fp'\in C_c^\infty(G')$ is a transfer of  $(f_0,f_1)\in C_c^\infty(G_{W_0})\times C_c^\infty(G_{W_1})$ if and only if $\fp'^\natural\in C_c^\infty( S_{n+1})$ is a transfer of  $(f^\natural_0,f^\natural_1)\in C_c^\infty(\U(W_0))\times C_c^\infty(\U(W_1))$.\qed
\end{corollary} 
\begin{proof}
This follows from the definition of the transfer in the two settings \cite[Def. 2.2 and Def. 2.4]{RSZ2} (cf. the definition \eqref{eq:trans} for the homogeneous group setting).
\end{proof}

\section{FL and AFL}\label{s:FLplusAFL}

From now on and for the rest of the paper, we let $W_0$, resp. $W_1$, be the split, resp. non-split, hermitian space of dimension $n+1$. In this section, we recall the FL and the AFL. We do this to facilitate the comparison with the quasi-canonical FL/AFL proved below.

\subsection{The FL} \label{ss:FL}
 To state the FL, resp. the AFL, we assume that the special vectors $u_0\in W_0$, resp. $u_1\in W_1$, have unit length.  Let $K_0\subset \U(W_0)(F_0)$ be the hyperspecial maximal compact subgroup which is the stabilizer of a selfdual lattice $\Lambda_0$ containing $u_0$ and, similarly, let $K^\flat_0\subset \U(W^\flat_0)(F_0)$ be the hyperspecial maximal compact subgroup which is the stabilizer of the selfdual lattice $\Lambda^\flat_0=\Lambda_0\cap W_0^\flat$, i.e.,  $K_0^\flat=K_0\cap \U(W_0^\flat)(F_0)$.
 Throughout the paper, the Haar measures on $\GL_n(F_0)$ and  $\GL_{n+1}(F_0)$ are normalized such that $\GL_n(O_{F_0})$ and  $\GL_{n+1}(O_{F_0})$ have measure one. For the FL, we furthermore   choose the Haar measures on $\GL_n(F)$ and on $\U(W_0)$ such that $\GL_n(O_F)$ and $K_0^\flat$ have volume one (this is independent of the choice of $\Lambda_0$). The Fundamental Lemma of Jacquet--Rallis is the following theorem.

\begin{theorem}\label{FLconj} Let $p>2$. 
\hfill
\begin{altenumerate}
\renewcommand{\theenumi}{\alph{enumi}}
\item\label{FLconj gp}
\textup{(Inhomogeneous group version)} The characteristic function $\mathbf{1}_{S_{n+1}(O_{F_0})} \in C_c^\infty(S_n)$ transfers to the pair of functions $(\mathbf{1}_{K_0},0) \in C_c^\infty(\U(W_0))\times C_c^\infty(\U(W_1))$.

\item\label{FLconj gp hom}
\textup{(Homogeneous group version)} The characteristic function $\mathbf{1}_{\GL_n(O_F)\times  \GL_{n+1}(O_{F})}\in  C_c^\infty(G')$ transfers to the pair of functions $(\mathbf{1}_{K^\flat_0\times K_0},0) \in C_c^\infty(G_{W_0})\times C_c^\infty(G_{W_1})$.

\item\label{FLconj smilie}
\textup{(Semi-Lie algebra version)} The characteristic function $\mathbf{1}_{(S_{n+1}\times  W'_{n+1})(O_{F_0})}\in  C_c^\infty(S_{n+1} \times W'_{n+1})$ transfers to the pair of functions $(\mathbf{1}_{K_0\times \Lambda_0},0) \in C_c^\infty(\U(W_0)\times W_0)\times C_c^\infty(\U(W_1)\times W_1)$.

\end{altenumerate}
\end{theorem}
The equal characteristic analog of the Jacquet--Rallis FL conjecture was proved by Z.~Yun for $p> n$, cf. \cite{Yun}; J.~Gordon deduced the $p$-adic case for $p$ large, but unspecified, cf. \cite[App.]{Yun}. The general statement above has two proofs: a purely local proof is given by Beuzart-Plessis \cite{BP},  another one essentially of a global nature  is given by the third author \cite{Zha21} when $p>n$ and in general by Z. Zhang \cite{ZZha}.  There is also a Lie algebra version of FL, which is also known for any residue characteristic.

\subsection{The AFL}Next we recall the AFL. We identify $W_1$ with $\mathbb{V}_{n+1}$ defined in \S\ref{sec:herm-space-mathbbv} and choose the special vector  $u_1\in W_1$ to be $\x\in \mathbb{V}_{n+1}$ defined in \S\ref{sec:relat-with-spec}, assumed to have length $1$.   Then we may identify the hermitian space  $W_1^\flat$ with $\mathbb{V}_{n}$. Recall that $\Zx\subseteq \n$ is the special divisor on $\n$ associated to $u$. It may be identified with the inclusion of RZ-spaces  $\CN_n\subseteq \CN_{n+1}$, cf. \cite[\S 5]{Kudla2012}. Then $G_{W_1}(F_0)$ acts on $\Zx\times\n$  via this identification and hence the arithmetic intersection number $\left\langle \Zx, g\Zx\right\rangle_{\Zx\times\n}$  defined in \S\ref{sec:qcAFL-arithm-inters} makes sense for $g\in G_{W_1}(F_0)$.

\begin{theorem}\label{AFLconj}
Let $p>2$.
\hfill
\begin{altenumerate}
\renewcommand{\theenumi}{\alph{enumi}}
\item
\textup{(Inhomogeneous group version)}\label{AFL gp} 
Suppose that $\gamma\in S_{n+1}({F_0})_\rs$ matches an element $g\in  \U(W_1)(F_0)_\rs$. Then 
\[
  \bigl\la \CZ(u), (1 \times g)\CZ(u)\bigr\ra_{\CZ(u)\times\CN_{n+1}}\cdot\log q=-\del\bigl(\gamma, \mathbf{1}_{S_{n+1}(O_{F_0})}\bigr) . 
\]
\item
\textup{(Homogeneous group version)}\label{AFL gp hom} 
Suppose that $\gamma\in G'({F_0})_\rs$ matches an element $g\in  G_{W_1}(F_0)_\rs$. Then 
\[
  \bigl\la \CZ(u), g\CZ(u)\bigr\ra_{\CZ(u)\times\CN_{n+1}}\cdot\log q=-\frac{1}{2}\del\bigl(\gamma, \mathbf{1}_{G'(O_{F_0})}\bigr) .
 \]
\item
\textup{(Semi-Lie algebra version)}\label{AFL semilie}
Suppose that $(\gamma,u')\in (S_{n+1}\times W_{n+1}')({F_0})_\rs$ matches an element $(g,u)\in ( \U(W_1)\times  W_1)(F_0)_\rs$, where $u=u_1\in W_1$ is the fixed special vector. Then 
\[
 \bigl\la \CZ(u), (1\times g)\CZ(u)\bigr\ra_{\CN_{n+1}\times\CN_{n+1}}\cdot\log q=-\del\bigl((\gamma,u'), \mathbf{1}_{(S_{n+1}\times W_{n+1}')(O_{F_0})}\bigr) .
 \]
 \em{Here $ \bigl\la \CZ(u), (1\times g)\CZ(u)\bigr\ra_{\CN_{n+1}\times\CN_{n+1}}=\chi(\N\times \N, \CZ(u)\cap^\BL (1\times g)\CZ(u))$.}
\end{altenumerate}
\end{theorem}
This is proved by the third author \cite{Zha21} when $F_0=\BQ_p$ and $p>n$,  for a general $F_0$ with residue cardinality $q>n$  by 
Mihatsch and the third author  in \cite{MZ}, and finally in full generality by Z. Zhang \cite{ZZha}. All of these works are essentially of global nature, relying on the modularity of generating series of special divisors.

\section{Transfer theorems}\label{sec:transfer}

In this section, we will formulate and prove some transfer theorems which arise in connection with the quasi-canonical fundamental lemma and some variants of it.
We continue to denote by $W_0$ (resp. $W_1$) a split (resp. non-split) Hermitian space of dimension $n+1$.
\subsection{Open compact subgroups}\label{ss: cpt open}
  Let $n\geq 1$ and $0\leq r\leq n$. We first define open compact subgroups, on the unitary group side and on the general linear group side, in a generality that is greater than is required for the quasi-canonical FL, but will be needed later. 

Recall the parity $\varepsilon=\varepsilon(r)\in\{0,1\}$ of $r$, defined in (\ref{eq:epsilon}). Fix a special vector $u_0$ of norm $\varpi^\varepsilon$  in $W_0$. We also fix a lattice  $\Lambda_0\in \Ver^0(W_0)$ with $u_0\in \Lambda_0$. Let $\Lambda^\flat_{0}$ be the lattice $W^\flat_0\cap \Lambda_0$. It is a vertex lattice of type $\varepsilon\in\{0,1\}$ (this can be seen using \cite[Lem. 7.2.2]{LZ22}), i.e., selfdual when $\varepsilon=0$, resp. almost selfdual when $\varepsilon=1$. We also fix a lattice   $\Lambda^\flat\in \Ver^r(W_0^\flat)$ such  that   $\Lambda^\flat\subset\Lambda^\flat_{0}$ or, equivalently, 
\begin{equation}\label{def Lambda06}
\Lambda^\flat \obot \langle u_0\rangle\subseteq \Lambda_0 .
\end{equation} 

  The following lemma shows the independence of our results of the choices made. 
  \begin{lemma}\label{lem transitive 1}
The group $\U(W_0^\flat)$ acts transitively on the set of $(\Lambda^\flat, \Lambda_0)$ satisfying \eqref{def Lambda06}.
\end{lemma} 

\begin{proof}We  can assume $\Lambda^\flat$ is the standard vertex lattice of type $r$ (they are all conjugate under  $\U(W_0^\flat)$). Consider the $\BF_{q^2}/\BF_q$-hermitian space $\ov W=(\Lambda^\flat \obot \langle u_0\rangle)^\vee/(\Lambda^\flat\obot \langle u_0\rangle)$ of dimension $r+\varepsilon$. Then the set of self-dual $\Lambda_0$ containing $\Lambda^\flat\obot \langle u_0\rangle$ corresponds to the set of lagrangian subspaces of $\ov W$. Hence we need to see that the subgroup  in the finite unitary group $\U(\ov W)$ fixing the elements  of the subspace   $\langle\ov u_0\rangle^\vee/\langle\ov u_0\rangle\subset \ov W$ acts transitively on the set of all lagrangian subspaces of $\ov W$. There are two cases. Either the subspace is zero, in which case the subgroup is the whole unitary group $\U(\ov W)$. Then it is well known by Witt's theorem that $\U(\ov W)$ acts transitively on the set of all lagrangian subspaces of $\ov W$. In the alternative case the subspace  is a non-isotropic line in $\ov W$. In this case this again follows from Witt's theorem which implies that $\U(\ov W)$ acts transitively on the set of vectors $(w_1,\ldots, w_{\frac{r+\varepsilon}{2}}, u)$ of $\ov W$ such that $w_1,\ldots, w_{\frac{r+\varepsilon}{2}}$ are a basis of a lagrangian subspace and $u$ is a non-isotropic vector of given length.
\end{proof}
  
Define open compact subgroups 
\begin{equation}\label{defKrsimple}
\kN=K_{n+1}^{[0]}:=\U(\Lambda_0)\subseteq \U(W_0)(F_0),\quad \knr:=\U(\Lambda^\flat)\subseteq \U(W_0^\flat)(F_0).
\end{equation} We define a finite index subgroup of $\knr$ by
\begin{equation}\label{defKtilde}
\tknr:=\knr\cap \U(\Lambda_0)= \knr\cap\kN.
\end{equation} 
Note that when $r$ is odd, $\tknr$ is not a parahoric subgroup when $n>1$.

\begin{remark}\label{explK}
\noindent (i) When $r=0$, then $K_{n}^{[0]}=\wit{K}_{n}^{[0]}$ is a  hyperspecial subgroup.

\smallskip

\noindent (ii) When $r$ is even, then $\Lambda_0= \Lambda^\flat_0\obot \langle u_0\rangle$. Then $\wit{K}_{n}^{[r]}=\U(\Lambda^\flat)\cap\U(\Lambda^\flat_0)=K_n^{[r, 0]}$ is a parahoric subgroup, the joint stabilizer of a vertex lattice of type $0$ and  a vertex lattice of type $r$ contained in it. 

\smallskip

\noindent (iii) When $r=1$, we have $\wit{K}_{n}^{[1]}=\ker({K}_{n}^{[1]}\to \BF^1_{q^2})$, where  the homomorphism is given by
\[
{K}_{n}^{[1]}\to \U((\Lambda^\flat_0)^\vee/\Lambda^\flat_0)=\ker(\Nm\colon \BF_{q^2}^\times\to\BF_{q}^\times)=:\BF^1_{q^2} .
\]
In particular,  $\wit{K}_{n}^{[1]}$ is  a normal subgroup of  ${K}_{n}^{[1]}$ (but $\wit{K}_{n}^{[r]}$ is  not a normal subgroup of  ${K}_{n}^{[r]}$ for general $r$). To see one inclusion, note that an element in the kernel of this homomorphism induces the identity automorphism on the two-dimensional $\BF_{q^2}/\BF_q$-hermitian space  $(\Lambda^\flat)^\vee/\Lambda^\flat)\obot \langle u_0\rangle^\vee/\langle u_0\rangle$. It therefore fixes the line  in this space defined by $\Lambda_0$ and hence lies in $\wit{K}_{n}^{[1]}$. Conversely,  an element of  $\wit{K}_{n}^{[1]}$ respects the decomposition into the  two hermitian subspaces and induces the identity on the second summand. Since it also  fixes an isotropic line in $(\Lambda^\flat)^\vee/\Lambda^\flat)\obot \langle u_0\rangle^\vee/\langle u_0\rangle$, it  induces the identity automorphism, and thus lies in the kernel.

\end{remark}
\begin{remark}
Recall that the members of the RZ-tower of rigid-analytic spaces corresponding to the local Shimura datum $(\U(W^\flat_0), (1, 0,\ldots,0), b_{\rm basic})$ are parametrized by open compact subgroups $K\subset \U(W_0^\flat)(F_0)$, cf. \cite[\S 5]{RV}. For instance, the generic fiber of $\CN_n^{[r]}$ corresponds to ${K}_{n}^{[r]}$, cf. \cite[\S 4.25]{RV}. The generic fiber of 
$\wt\CN_n^{[r]}$ corresponds to $\wit{K}_{n}^{[r]}$.  

The case $n=2$ and $r=1$ is of special interest.   Then $K_2^{[1]}=\U(\Lambda_0^\flat)$. It is equipped with a homomorphism
\begin{equation}\label{Ktame}
K_2^{[1]}\to (\BF_{q^2}^1)^2,
\end{equation}
given by letting $K_2^{[1]}$ act on the two one-dimensional hermitian spaces $(\Lambda^\flat_0)^\vee/ \Lambda^\flat_0$ and $\varpi^{-1} \Lambda^\flat_0/(\Lambda^\flat_0)^\vee$. This homomorphism is surjective, and the kernel is the pro-unipotent radical of $K_2^{[1]}$. The map induced in the generic fiber of $\wt\CN_2^{[1]}\to \CN_2^{[1]}$ corresponds to the Galois cover given by the kernel of the first component of \eqref{Ktame}. Note that there is an involution $\theta$ on $ \CN_2^{[1]}$, given by sending a point $(X, \iota, \lambda, \rho)$ to  $(X^\vee, \bar\iota^\vee, \lambda^\vee, (\rho^\vee)^{-1})$. Here $\bar\iota^\vee\colon O_F\to \End(X^\vee)$ sends $a$ to $\iota^\vee(\bar a)$, and  the polarization $\lambda^\vee$ is the unique map $\lambda^\vee: X^\vee\to X$ such that $\lambda^\vee\circ\lambda=\varpi\cdot\id$. Then $\theta$  induces an  involution in the generic fiber of  $ \CN_2^{[1]}$ such that the induced action on $K_2^{[1]}$ interchanges the two components of \eqref{Ktame}.
\end{remark}

Now we define open compact subgroups on the general linear group side. We  fix an orthogonal basis $\{e_1,\cdots,e_n\}$  of $W_0^\flat$ and extend it by $u_0$ (of length $\varpi^\varepsilon$) to an orthogonal basis of $W_0$; from now on we will use this basis to identify $\GL_{n,F}\simeq \GL_F(W^\flat_0)$ and  $\GL_{n+1,F}\simeq \GL_F(W_0)$. We may assume that 
$$(e_i,e_i)=\begin{cases}1, & 1\leq i\leq n-1,\\ -\varpi^{\varepsilon}, & i=n.
\end{cases}
$$
We may take $\Lambda^\flat_0$ to be the standard lattice with respect to the basis $e_1,\ldots, e_n$. We define 
\begin{equation}\label{defKn'}
K_n'=\GL_{O_F}(\Lambda^\flat_{0 })={\rm Stab}_{\GL_n(F)}(\Lambda^\flat_{0 }) .
\end{equation}
 We also introduce $K_n^{\prime [r]}$, the joint stabilizer of $\Lambda^\flat$ and its dual lattice $(\Lambda^\flat)^\vee$,
 \begin{equation}
 K_n^{\prime [r]}={\rm Stab}_{\GL_n(F)}(\Lambda^\flat) \cap {\rm Stab}_{\GL_n(F)}((\Lambda^\flat)^\vee) .
 \end{equation} Then $K_n^{\prime [r]}$ is a parahoric subgroup and  $K_n^{\prime }$ is a maximal parahoric subgroup.

Recall the self-dual lattice $\Lambda_0$ chosen earlier at \eqref{def Lambda06}; it is related to $\Lambda^\flat_{0 }$ through the following sequence of inclusions,
  \begin{equation}\label{eq Lambda0 odd}
\begin{aligned}
\begin{cases}\Lambda^\flat_{0 }\oplus\pair{ u_0}\subset^1\Lambda_0\subset^1(\Lambda^\flat_{0 })^\vee\oplus\pair{ u_0}^\vee, &\text{  $r$  odd}\\
\Lambda^\flat_{0 }\oplus\pair{ u_0}=\Lambda_0=(\Lambda^\flat_{0 })^\vee\oplus\pair{ u_0}^\vee, &\text  {$r$ even.}
\end{cases}
\end{aligned}
\end{equation}

We define 
\begin{equation}\label{eq K' K''}
K'_{n+1}:=\GL_{O_F}(\Lambda_0),\quad \KM= \GL_{O_F}(\Lambda^\flat_{0 } \oplus \pair{u_0}) .
\end{equation} 
When $r$ is even, then $K'_{n+1}=\KM$.  Also, we have an inclusion $K_n'\subset \KM$ (but $K_n'\not\subset K'_{n+1}$ if $r$ is odd).
In the sequel, we consider the integral form of $G'$ defined by the lattice $(\Lambda^\flat_{0 })\oplus (\Lambda^\flat_{0 } \oplus \pair{u_0})$ of $W_0^\flat\oplus W_0$, i.e., $G'(O_{F_0})=K'_n\times\KM$.

We also define
the finite index subgroup $\widetilde{K}_n^{\prime [r]}$ of $K_n^{\prime [r]}$ by 
\begin{equation}\label{eq wit K'}
\widetilde{K}_n^{\prime [r]}:=K_n^{\prime [r]}\cap K'_{n+1}= {\rm Stab}_{\GL_n(F)}(\Lambda^\flat) \cap {\rm Stab}_{\GL_n(F)}((\Lambda^\flat)^\vee) \cap {\rm Stab}_{\GL_{n+1}(F)}(\Lambda_0).
\end{equation} 
More explicitly, the lattice $\Lambda_0$ defines a lagrangian subspace inside the $\BF_{q^2}/\BF_q$-hermitian space $((\Lambda^\flat)^\vee\obot \langle u_0\rangle^\vee)/(\Lambda^\flat \obot \langle u_0\rangle)$ of even dimension $r+\varepsilon$. Then $K_n^{\prime [r]}$ acts on this hermitian space and $\widetilde{K}_n^{\prime [r]}$ is the stabilizer of this lagrangian subspace.  In analogy to the unitary side (i.e., to  $\widetilde{K}_n^{ [r]}$), we expect that  $\widetilde{K}_n^{\prime [r]}$ is not a parahoric subgroup when $r$ is odd.

 When defining orbital integrals, we need to choose Haar measures. For $\GL_n(F_0)$ and $\GL_{n+1}(F_0)$ we normalize the Haar measures by giving a maximal compact subgroup volume one. For  $\GL_{n}(F)$ and $\U(W_0^\flat)(F_0)$, we postulate
\begin{equation}\label{normmea}
\vol(K_n^{\prime})=1,\quad
\vol(\wit{K}_{n}^{[\varepsilon]} )=1.
\end{equation}
In other words, on the $\GL$-side, we normalize the Haar measures by giving the maximal compact subgroups volume one. On the $\U$-side, when 
$\varepsilon=0$, we normalize the Haar measures by giving the hyperspecial maximal compact subgroups volume one, whereas when $\varepsilon=1$, we normalize the Haar measures by giving the compact subgroup $\wit{K}_{n}^{[1]}$  volume one.
We denote 
 \begin{equation}\label{defc}
\begin{aligned}
c_r&:=\vol(\wit{K}_{n}^{[r]} )^{-1}=[\wit{K}_{n}^{[\varepsilon]} : \tknr], \\
c'_r&:=\vol(\widetilde{K}_n^{\prime [r]})^{-1}=[ K'_{n}:\widetilde{K}_n^{\prime [r]}].
\end{aligned}
\end{equation}

For $r=1$, we have $c_1=1$ and $c_1'=$ the index in $\GL_2(\BF_{q^2})$  of the mirabolic subgroup (the subgroup of lower triangular matrices with $1$ in the right lower corner). Hence $c_1'=(q^2+1)(q^2-1)$.

\subsection{The case  when $r$ is odd} Consider the following  element in  $\GL_{n+1}(F_0)$
\begin{align}\label{eq:def u}
{\bf u}=\begin{bmatrix}
	1 _{n} &  \varpi^{-1}\,e_{n}  \\ 
      0&  1
   \end{bmatrix},
\end{align}
where $e_{n}=\,^t(0,\cdots, 0,1)\in M_{n\times 1}(F_0)$.  It acts on $C_c^\infty( S_{n+1})$ by
$$
({\bf u}\ast \Fp')(\gamma)=\Fp'({\bf u}^{-1}\gamma{\bf u}).
$$
Set $h_0=\begin{bmatrix}
	  \varpi\cdot 1 _{n} &  \\ 
      0&  1
   \end{bmatrix}$ and ${\bf u'}=h_0 {\bf u}$. Then we have  by  \eqref{eq Lambda0 odd}   and the definitions \eqref{eq K' K''} of $K'_{n+1}$ and $\KM$,
\begin{equation}\label{Kuconj}
\KM={\bf u'}^{-1}K'_{n+1}{\bf u'} .
\end{equation}

Define the function
\begin{align}\label{eq:tran Fp}
\Fp'_s=(q^{2(n+1)}-1)\, {\bf u}\ast {\bf 1}_{ S_{n+1}(O_{F_0})}+ ((-1)^{n+1}q^{-(n+1)s}+1){\bf 1}_{ S_{n+1}(O_{F_0})}\in C_c^\infty( S_{n+1})
\end{align}
and\begin{align}\label{eq:tran Fp}
\Fp'=\Fp'_{|{s=0}}.
\end{align}

\begin{theorem}\label{prop:an-explicit-transfer} Let $r$ be odd.  

\smallskip

\noindent (i)\emph{(Inhomogeneous version)} The function $\phi'$ is a transfer of $({\bf 1}_{\kN},0)$.

\smallskip

\noindent (ii)\emph{(Homogeneous version)}  Let $\fp'\in C_c^\infty(G')$ be any function such that  $\fp'^\natural=c_r^{-1} \Fp'\in C_c^\infty( S_{n+1})$. Then $\fp'$ is a transfer of $({\bf 1}_{\tknr\times \kN},0)$. 

\end{theorem}

Before giving the proof, we exhibit some functions as in the statement of Theorem \ref{prop:an-explicit-transfer}.

\begin{lemma}\label{lem hom2inhom}
Let $r$ be odd.
\smallskip

\noindent (i) 
Let $\varphi'=  \mathbf{1}_{ K'_{n}\times\KM}= \mathbf{1}_{G'(O_{F_0})}\in C_c^\infty(G')$. Then 
$$\varphi'^\natural={\bf 1}_{ S_{n+1}(O_{F_0})}.$$

\noindent (ii) Let $\fp'= \mathbf{1}_{\widetilde{K}_n^{\prime [r]}\times K'_{n+1} }\in C_c^\infty(G')$. Then
$$
\fp'^\natural=c_r'^{-1}(-1)^n{\bf u'}\ast {\bf 1}_{ S_{n+1}(O_{F_0})}.
$$

\end{lemma}
\begin{proof}
 \noindent (i) Recall from \eqref{eqn def wt f} the definition of $\fp'^\natural$ and note that $K'_{n}\subset\KM$.  The integral over $\GL_n(F)$ gives the function $\vol(K'_n) {\bf 1}_{\KM}= {\bf 1}_{\KM}$ on $\GL_{n+1}(F)$. Then the integral over $\GL_{n+1}(F_0)$ gives the function ${\bf 1}_{ S_{n+1}(O_{F_0})}$.

\noindent (ii) Similarly to (i), by $\widetilde{K}_n^{\prime [r]}\subset K'_{n+1} $,  the integral over $\GL_n(F)$ gives the function  $\vol(\widetilde{K}_n^{\prime [r]}) {\bf 1}_{K'_{n+1}}$ on $\GL_{n+1}(F)$. To compute the second integral, 
we note that  $K'_{n+1}= {\bf u'} \KM{\bf u'}^{-1}$ by \eqref{Kuconj}.  Now note that $ {\bf u'}\in \GL_{n+1}(F_0)$.  By \eqref{eqn def wt f}, we may replace the function $ {\bf 1}_{K'_{n+1}}$ on $\GL_{n+1}(F)$  by  $\wit\eta^n({\bf u'} )  {\bf 1}_{{\bf u'}  \KM}=(-1)^{n^2}  {\bf 1}_{{\bf u'}  \KM}=(-1)^{n}  {\bf 1}_{{\bf u'}  \KM}$  and the assertion follows from the proof of part (i).
\end{proof}
Applying Theorem  \ref{prop:an-explicit-transfer}, we obtain the following corollary.  
\begin{corollary}\label{defifct} Let $r$ be odd. 
The following transfer statements hold.
\smallskip

\noindent (i) \emph{(Inhomogeneous version)} The function 
\begin{equation}\label{varphi nat}
\phi'_r=(q^{2(n+1)}-1)\, {\bf u}\ast {\bf 1}_{ S_{n+1}(O_{F_0})}+ ((-1)^{n+1}+1){\bf 1}_{ S_{n+1}(O_{F_0})}\in C_c^\infty( S_{n+1})
\end{equation}
of \eqref{eq:tran Fp} is a transfer of $({\bf 1}_{K_{n+1}}, 0)$. 

\smallskip

\noindent (ii)  \emph{(Homogeneous version)} The function 
\begin{equation}\label{varphi sharp}
\varphi'_r=c_r \big(c_r' (q^{2(n+1)}-1) \mathbf{1}_{\widetilde{K}_n^{\prime [r]}\times K'_{n+1}}+ ((-1)^{n+1}+1)\mathbf{1}_{G'(O_{F_0})}\big)\in C_c^\infty(G')
\end{equation}

 is a transfer of $(c_r^{2}\,{\bf 1}_{\tknr\times \kN},0)$.

 \smallskip
 
\end{corollary}
\begin{proof}
Part (i) is repeating part (i) of Theorem  \ref{prop:an-explicit-transfer}. To show part (ii), 
we first claim that $ {\bf u}\ast {\bf 1}_{ S_{n+1}(O_{F_0})}$ and $(-1)^n{\bf u'}\ast {\bf 1}_{ S_{n+1}(O_{F_0})}$ have the same orbital integrals. In fact,  by \eqref{eqn def inhom} we have
\begin{equation*}
   \Orb(\gamma,{\bf u'}\ast {\bf 1}_{ S_{n+1}(O_{F_0})}, s)=\omega_{S,s}(\gamma) \int_{H'(F_0)}{\bf u}\ast {\bf 1}_{ S_{n+1}(O_{F_0})} ( h^{-1}_0 h^{-1}\gamma hh_0)\lvert \det h \rvert^s \eta(h) \, dh .
\end{equation*}
Since $h_0\in \GL_n(F_0)$,
we may substitute $h$ by $hh_0^{-1}$: 
   \begin{align}\label{eq u'2u}
   \Orb(\gamma,{\bf u'}\ast {\bf 1}_{ S_{n+1}(O_{F_0})}, s)= (-1)^nq^{ns}\Orb(\gamma,{\bf u}\ast {\bf 1}_{ S_{n+1}(O_{F_0})}, s).
   \end{align}Setting $s=0$ proves the claim.
Now part (ii) follows from Lemma \ref{lem hom2inhom} and part (i).

\end{proof}

\begin{remark}
(i) In particular, when $n$ is  even, we see  that  the more natural looking  function $\mathbf{1}_{\wt{K}_n^{\prime [r]}\times K'_{n+1}}$ is up to a scalar  a transfer of $({\bf 1}_{\tknr\times \kN },0)$. But when $n$ is odd, this natural looking function  does not seem to give the desired transfer. 

 (ii) The statement of Corollary \ref{defifct} is the analogue of the FL which states in its homogeneous version that, when $u$ has unit length, the function $\mathbf{1}_{{K}_n^{\prime }\times K'_{n+1}}\in C_c^\infty(G')$ is a transfer of $({\bf 1}_{K_n\times K_{n+1}},0)\in C_c^\infty(G_{W_0})\times C_c^\infty(G_{W_1})$. Its inhomogeneous version states that $\mathbf{1}_{ S_{n+1}(O_{F_0})}$ is a transfer of $(\mathbf{1}_{K_{n+1}}, 0)\in C_c^\infty(\U({W_0}))\times C_c^\infty(\U({W_1}))$.

(iii) Let $r=1$. Then $c_1=1$, and the expression \eqref{varphi sharp} of $\varphi'_1$ simplifies slightly. Also the function $\phi'_1$ of \eqref{varphi nat} transfers to $({\bf 1}_{K_{n+1}}, 0)$. The case $r=1$ of Corollary \ref{defifct} is referred to in the sequel as the \emph{quasi-canonical FL}. We restate it in Theorem \ref{thmqcFL}.
\end{remark}
\begin{proof} (of Theorem \ref{prop:an-explicit-transfer})
Part (ii) follows from part (i): we convert the homogeneous version for $f={\bf 1}_{\tknr\times \kN}\in C_c^\infty(\U(W_0^\flat)\times\U(W_0)) $ into the inhomogeneous version, cf. \eqref{def:wt varphi}. Since $\tknr\subset \kN$, it is easy to see that the resulting function is
 $$f^\natural=\vol( \tknr) {\bf 1}_{ \kN}\in C_c^\infty(\U(W_0)).
 $$
 Therefore part (ii) follows from part (i) by Corollary \ref{transnat} (cf. the end of  \S \ref{s:orb}).  
 
We now prove part (i). We   claim that we have an interpretation as lattice counting,  
\begin{align}\label{eq Orb to lat}
\Orb(g,   {\bf 1}_{ \kN})=\#\{\Lambda\in \Ver^0(W_0)\mid u_0\in \Lambda, g\Lambda=\Lambda\}.
\end{align}
Here $u_0\in W_0$ is the special vector with valuation one. To show the claim, we note that the lattice $\Lambda_0^\flat:=\Lambda_0\cap W^\flat_0$ is a vertex lattice of type $1$ and  $\U(W^\flat_0)(F_0)\cap \kN=\wit{K}_{n}^{[1]}$. By definition we have
\begin{align*}
\Orb(g,{\bf 1}_{ \kN})&=\int_{\U(W^\flat_0)} {\bf 1}_{ \kN}(h^{-1}gh)\,dh\\
&=\vol(\wit{K}_{n}^{[1]} ) \sum_{\U(W^\flat_0)/ \wit{K}_n^{[1]}}  {\bf 1}_{ \kN}(h^{-1}gh).
\end{align*}
The condition $h^{-1}gh\in \kN$ is equivalent to 
$$
gh\Lambda_0=h\Lambda_0.
$$
Let $\Xi$ denote the set of lattices $\Lambda$ of the form $h\Lambda_0$ for $h\in \U(W^\flat_0)(F_0)$ such that $g\Lambda=\Lambda$, and let $\Xi'$ denote the set of lattices $\Lambda\in \Ver^0(W_0)$ such that $ u_0\in \Lambda, g\Lambda=\Lambda$. Then clearly  $\Xi\subset\Xi'$. We now show the reverse inclusion. For $\Lambda\in \Xi'$, we may write $\Lambda=h' \Lambda_0 $ for some $h'\in \U(W_0)(F_0)$.  Since $u_0\in \Lambda$ we have $h'^{-1}u_0\in\Lambda_0$. Note that $u_0$ and $h'^{-1}u_0$ are both in $\Lambda_0$ and both of length $\varpi$. It follows that there exists $k\in K_{n+1}$ such that $h'^{-1}u_0=k^{-1}u_0$ (here we are using the fact that the compact open subgroup $K_{n+1}=\U(\Lambda_0)$ acts transitively on the set of length-$\varpi$ vectors in $\Lambda_0$). Therefore $h:=h'k^{-1}\in \U(W_0^\flat)$ and $\Lambda=h' \Lambda_0=h' k^{-1}k\Lambda_0=h \Lambda_0\in \Xi$. We have thus proved $\Xi=\Xi'$.
The claim now follows from $\sum_{\U(W^\flat_0)/ \wit{K}_n^{[1]}}  {\bf 1}_{ \kN}(h^{-1}gh)=\#\Xi=\#\Xi'$, and $\vol(\wit{K}_{n}^{[1]} )=1$ by our normalization of measures.

Now we may relate \eqref{eq Orb to lat} to the orbital integral in the semi-Lie algebra version, cf. \eqref{semiunit},
$$
\Orb(g, {\bf 1}_{ \kN})=\Orb((g,u_0), {\bf 1}_{\kN\times\Lambda_0}), \quad g\in \U(W_0)(F_0)_\rs .
$$
Here on the right hand side the measure on $\U(W_0)(F_0)$ is chosen such that $\vol(\kN)=1$. Here we are implicitly using the relation between transfer factors,  comp. \eqref{eq tran=} below. 

As in \eqref{defW}, let $W'=F_0^{n+1}\times (F_0^{n+1})^\ast$, and let $\Lambda'$ be the standard lattice in $ W'(F_0)$. Also let  $K'= S_{n+1}(O_{F_0})$. Then by the semi-Lie version of the Jacquet--Rallis FL, see Theorem \ref{FLconj} part \ref{FLconj smilie},  ${\bf 1}_{K'\times\Lambda'}$ is a transfer of $({\bf 1}_{K_{n+1}\times \Lambda_0}, 0)$. Therefore we know that for regular semisimple $(\gamma,w)\in  (S_{n+1}\times W')(F_0)_\rs$  matching  $(g,u)\in (\U(W)\times W)(F_0)_\rs$,
\begin{equation}\label{Fleq}
\begin{aligned}
\Orb((\gamma,w), {\bf 1}_{K'\times\Lambda'})=\begin{cases}\Orb((g,u), {\bf 1}_{\kN\times\Lambda_0}),&  W=W_0  \text{ split}, \\
0, & W=W_1 \text{ non-split}.
\end{cases}
\end{aligned}
\end{equation}

It suffices to relate the left hand side of \eqref{Fleq}  to the orbital integral (relative to $\GL_{n}(F_0)$) of the function $\Fp'$ defined by \eqref{eq:tran Fp}.
   We may assume that $(g,u_0)$ matches $(\gamma,w_0)$ where $w_0$ is the special vector 
\begin{equation}
w_0=(\varpi\, e_{n+1}, \, ^te_{n+1})\in W'.
\end{equation}
Then the assertion follows from the following lemma.
\end{proof}

\begin{lemma}
\label{lem Orb red}
For all regular semisimple $\gamma\in S_{n+1}(F_0)$,
\begin{equation}\label{orbB}
 \Orb((\gamma,w_0), {\bf 1}_{K'\times\Lambda'},s)=\Orb(\gamma, \phi'_s,s).
 \end{equation}
Here  the RHS is defined by the formula \eqref{eqn def inhom}, in which the function $\phi'$ has to be replaced by $\phi'_s$. 
 \end{lemma}
 \begin{proof}
We first note that the transfer factors match 
\begin{equation}\label{eq tran=}
\omega_{S\times W',s}(\gamma,w_0)=\omega_{S,s}(\gamma)
\end{equation}
(cf. \eqref{def trans S} and \eqref{def trans S W'}).
Next we compare the integrals in \eqref{eqn def inhom} and \eqref{def Orb S W'}.

By the Iwasawa decomposition, we may write 
$$
\GL_{n+1}(F_0)\simeq Z\times \GL_{n}(F_0)\times N(F_0)\times \GL_{n+1}(O_{F_0}) ,
 $$
 where $Z\simeq F_0^\times$ is the center, $N$ (resp. $ \GL_{n}$) is the unipotent radical (resp. the Levi) of the mirabolic subgroup. The Haar measure on $\GL_{n+1}(F_0)$ can be taken as
 the product measure of the Haar measures on the factors, normalized such that the natural maximal compact open subgroups all have volume one. Write an element in $\GL_{n+1}(F_0)$ as a product $zhuk$ according to the decomposition. Note that  
 ${\bf 1}_{K'\times\Lambda'}$ is invariant under $\GL_{n+1}(O_{F_0})$ and hence the integral over $\GL_{n+1}(O_{F_0})$ can be dropped.
Then the integral  in $\Orb((\gamma,w_0), {\bf 1}_{K'\times\Lambda'},s)$ (cf. \eqref{def Orb S W'}) decomposes into
\begin{align*}
 \int_{F_0^\times\times \GL_{n}(F_0)\times N }{\bf 1}_{K'}( u^{-1}h^{-1}\gamma h u) {\bf 1}_{\Lambda'}( z^{-1}h^{-1}\varpi\, e_{n+1} ,\,^t e_{n+1} h u  )\eta(h)\eta(z)^{n+1}|z|^{(n+1)s}|\det(h)|^sdz\, du \, dh.
 \end{align*}
 Note that $h\in \GL_{n}(F_0)$ acts trivially on the special vector $w_0=(\varpi\, e_{n+1}, \,^te_{n+1})$, and $^te_{n+1}  u=\,^te_n$. Hence the condition
 $(u^{-1} h^{-1}z^{-1}\varpi\, e_{n+1}, \,^te_{n+1} z h u )\in \Lambda'$ is equivalent to 
  $(u^{-1}z^{-1} \varpi\, e_{n+1}, \,^te_{n+1} z  )\in \Lambda'$.  There are two cases.
  
  \begin{altenumerate}
  \item ${\rm val}(z)=0$.  Then the integrality of $u^{-1}z^{-1} \varpi\, e_{n+1}$ is equivalent to that of $\varpi u$.  
The contribution to the orbital integral is the same as 
\begin{align}\label{eq:orb int}
\int_{\GL_{n}(F_0)\times N(\varpi^{-1}O_{F_0}) }{\bf 1}_{K'}( u^{-1}h^{-1}\gamma h u)\eta(h)|\det(h)|^s \, dh \, du.
\end{align}  
 Note that the integrand is $\GL_{n}(O_F)$-invariant. Therefore we have for any $k\in \GL_{n}(O_F)$,
$$
\int_{\GL_{n}(F_0) }{\bf 1}_{K'}(k u^{-1}k^{-1}h^{-1}\gamma h k uk^{-1})\eta(h)|\det(h)|^s \, dh =\int_{\GL_{n}(F_0) }{\bf 1}_{K'}(u^{-1} h^{-1}\gamma h u )\eta(h) |\det(h)|^s\, dh.
$$ 
Note that we may identify $N$ with $F_0^{n}$ and  $\GL_{n}(O_{F_0})$ acts on it in the standard way. The above invariance shows that the inner integral on $h\in \GL_{n}(F_0)$ in \eqref{eq:orb int}, viewed as a function in $u\in N(F_0)$, depends only on the  $\GL_{n}(O_{F_0})$-orbit of $u$. There are precisely two $\GL_{n}(O_{F_0})$-orbits in $N(\varpi^{-1}O_{F_0})$, represented by $1$ and the special element ${\bf u}$ defined in \eqref{eq:def u}. 
It is now easy to see that the integral \eqref{eq:orb int} is equal to
\begin{align}\label{eq:orb int1}
\int_{\GL_{n}(F_0) }{\bf 1}_{K'}(h^{-1}\gamma h )\eta(h) |\det(h)|^s\, dh +(q^{2(n+1)}-1) \int_{\GL_{n}(F_0) }{\bf 1}_{K'}( {\bf u}^{-1}h^{-1}\gamma h{\bf u} )\eta(h)|\det(h)|^s \, dh. 
\end{align}  
\item ${\rm val}(z)=1$.  Then ${\rm val} (z^{-1} \varpi)=0$. Hence by the integrality of $u^{-1}z^{-1} \varpi\, e_{n+1}$, we  have $u\in \GL_{n+1}(O_{F_0})\cap N$. By the invariance of $ {\bf 1}_{K'}$ under $\GL_{n+1}(O_{F_0})$, this contribution to the orbital integral is the same as 
\begin{align}\label{eq:orb int2}
(-1)^{n+1}q^{-(n+1)s}\int_{\GL_{n}(F_0)}{\bf 1}_{K'}( h^{-1}\gamma h ) \eta(h)|\det(h)|^s\,dh,
\end{align} 
where the first factor is due to $\eta(z)^{n+1}|z|^{(n+1)s}=(-1)^{n+1}q^{-(n+1)s}$.  
\end{altenumerate}
Combining \eqref{eq:orb int1} and \eqref{eq:orb int2}, and using the equality \eqref{eq tran=} we obtain the required identity \eqref{orbB}. 
The proof is complete.
\end{proof}
\subsection{The case when $r$ is even} 
In the previous subsection, we considered the case when $r$ is odd. We state here the results in the case of even $r$, which  is simpler. Now $\varepsilon=0$ and $\langle u_0\rangle$ is a direct summand: $
\Lambda^\flat_{0} \obot \langle u_0\rangle= \Lambda_0 
$,
where  $\Lambda^\flat_{0}$
 is a self-dual lattice in $W_0^\flat$.  In particular, we have 
 $\tknr:=\knr\cap \U(\Lambda_0)= \knr\cap \U(\Lambda^\flat_0) =K_{n}^{[0,r]}$,
 cf. \eqref{defKtilde}.

On the general linear group side, the two compact opens in \eqref{eq K' K''} coincide and  in \eqref{eq wit K'}  we have $\widetilde{K}_n^{\prime [r]}=K_n^{\prime [r]}$.

\begin{proposition}\label{prop:an-explicit-transfer even r} Let $r$ be even. The following transfer statements hold.
\smallskip

\noindent (i) \emph{(Inhomogeneous version)} The function 
$
\phi_r'={\bf 1}_{ S_{n+1}(O_{F_0})}\in C_c^\infty( S_{n+1})
$
 is a transfer of $( {\bf 1}_{K_{n+1}}, 0)$. 

 \smallskip
 
\noindent (ii) \emph{(Homogeneous version)}    The function 
$
\varphi'_r=\frac{c'_r}{c_r} \mathbf{1}_{\widetilde{K}_n^{\prime [r]}\times K_{n+1}'}\in C_c^\infty(G')
$
 is a transfer of $({\bf 1}_{\tknr\times \kN},0)$.

\end{proposition}
\begin{proof}The proof of
part (i) is similar to that of Theorem \ref{prop:an-explicit-transfer} and we omit it. For part (ii), we have an analog of Lemma \ref{lem hom2inhom}: $$(\mathbf{1}_{\widetilde{K}_n^{\prime [r]}\times K_{n+1}'})^{\nat}=\vol(\widetilde{K}_n^{\prime [r]}){\bf 1}_{ S_{n+1}(O_{F_0})}$$ and $$({\bf 1}_{\tknr\times \kN})^\natural=\vol(\tknr){\bf 1}_{ \kN}.$$ Then the assertion follows from part (i).
\end{proof}

\section{The quasi-canonical FL and AFL}\label{s:qcAFL}

To state the quasi-canonical FL, resp. the quasi-canonical AFL, we assume that the special vectors $u_0\in W_0$, resp. $u_1\in W_1$, have  length $\varpi$.  The Haar measures on $\GL_n(F)$, on $\GL_n(F_0)$, on $\GL_{n+1}(F_0)$, and on $\U(W_0^\flat)(F_0)$ are chosen as in \eqref{normmea} for $r=1$. Note that now we have $\vol(\wit{K}_{n}^{[1]} )=1$.

\subsection{The quasi-canonical FL} The following theorem is just a restatement of Corollary \ref{defifct} in the case $r=1$.

\begin{theorem}({\rm Quasi-canonical FL})\label{thmqcFL}
Let $p>2$.

\smallskip

\noindent (i) \emph{(Inhomogeneous version)} The function 
\begin{equation*}
\phi'=(q^{2(n+1)}-1)\, {\bf u}\ast {\bf 1}_{ S_{n+1}(O_{F_0})}+ ((-1)^{n+1}+1){\bf 1}_{ S_{n+1}(O_{F_0})}\in C_c^\infty( S_{n+1})
\end{equation*}
 is a transfer of $({\bf 1}_{K_{n+1}}, 0)\in C_c^\infty(\U(W_0))\times C_c^\infty(\U(W_1))$. 

\smallskip

\noindent (ii) \emph{(Homogeneous version)} Recall that $c_1=1,c_1'=(q^2+1)(q^2-1)$. The function 
\begin{equation*}
\varphi'=c_1' (q^{2(n+1)}-1) \mathbf{1}_{\widetilde{K}_n^{\prime [1]}\times K'_{n+1}}+ ((-1)^{n+1}+1)\mathbf{1}_{G'(O_{F_0})}\in C_c^\infty(G')
\end{equation*}
 is a transfer of $({\bf 1}_{\tilde {K}_n^{[1]}\times \kN},0)\in C_c^\infty(G_{W_0})\times C_c^\infty(G_{W_1})$.\qed

 \smallskip
 
\end{theorem}
\subsection{The quasi-canonical AFL}
We next turn to the quasi-canonical AFL. 
 We take up the setup in \S\ref{sec:analytic-side}, with  $n\geq 1$.  We identify $W_1$ with $\mathbb{V}_{n+1}$ defined in \S\ref{sec:herm-space-mathbbv} in such a way that the special vector  $u_1\in W_1$ equals $\x\in \mathbb{V}_{n+1}$ defined in \S\ref{sec:relat-with-spec} (assumed to have length $\varpi$).   Then we may identify the hermitian space  $W_1^\flat$ with $\mathbb{W}_{n}^{[1]}$. Recall that $\Zx\subseteq \n$ is the special divisor on $\n$ associated to $u$. Then $G_{W_1}(F_0)$ acts on $\Zx\times\n$  via this identification and hence the arithmetic intersection number $\left\langle \Zx, g\Zx\right\rangle_{\Zx\times\n}$  defined in \S\ref{sec:qcAFL-arithm-inters} makes sense for $g\in G_{W_1}(F_0)$.

\begin{theorem}({\rm Quasi-canonical AFL})\label{thmzAFL}

\smallskip

\noindent (i)  \emph{(Inhomogeneous version)} Let $\phi'\in C_c^\infty( S_{n+1})$ be as in (i) of Theorem \ref{thmqcFL}. If $\gamma\in  S_{n+1}(F_0)_\rs$ is matched with  $g\in \U(W_1)(F_0)_\rs$, then  
\begin{equation*}
    \left\langle \Zx, (1,g)\Zx\right\rangle_{\Zx\times\n}\cdot\log q= -\del\big(\gamma,  \phi' \big)- \Orb\big(\gamma,  \phi'_\corr \big) ,
\end{equation*}
where 
$$
\phi'_\corr=\begin{cases} (n+1) {\bf 1}_{S_{n+1}(O_{F_0})}\cdot\log q
, & $n$ \text{ is even, }\\
0, & $n$ \text{ is odd.}
\end{cases}
$$

\smallskip

\noindent (ii)  \emph{(Homogeneous version)}
Let $\varphi'\in C_c^\infty(G')$ as in (ii) of Theorem \ref{thmqcFL}. If $\gamma\in G'(F_0)_\rs$ is matched with  $g\in G_{W_1}(F_0)_\rs$,   then
\begin{equation*}
   \left\langle \Zx, g\Zx\right\rangle_{\Zx\times\n}\cdot\log q=- \frac{1}{2}\del\big(\gamma,  \varphi' \big)- \Orb\big(\gamma,  \fp'_\corr \big) ,
\end{equation*}
where 
$$
\fp'_\corr=\begin{cases} \frac{1}{c'_1}(n+1) {\bf 1}_{G'(O_{F_0})}\cdot\log q
, & $n$ \text{ is even, }\\
0, & $n$ \text{ is odd.}
\end{cases}
$$

\end{theorem}
 The argument in  \cite[Conj. 5.3, resp. Conj. 5.6, resp. Conj. 5.10]{RSZ1} of the implication of $a)\implies b)$ in loc.~cit. implies the following corollary. 
 
\begin{corollary}\label{item:conjz3}
Assume the density conjecture \cite[Conj. 5.16]{RSZ1}.
 \smallskip

\noindent (i)  \emph{(Inhomogeneous version)} Let   $\phi'\in C_c^\infty(S_{n+1})$ be any transfer   of  $({\bf 1}_{K_{n+1}}, 0)\in  C_c^\infty(\U(W_0))\times  C_c^\infty(\U(W_1))$. Then there exists 
a function $\phi'_\corr\in C_c^\infty(G')$  such that, if $\gamma\in  S_{n+1}(F_0)_\rs$ is matched with  $g\in \U(W_1)(F_0)_\rs$, then
\begin{equation*}
   \left\langle \Zx, (1\times g)\Zx\right\rangle_{\Zx\times\n}= -\del\big(\gamma,  \Fp'\big) -  \Orb\big(\gamma,  \Fp'_\corr \big).
\end{equation*}

\smallskip

\noindent (ii)  \emph{(Homogeneous version)} Let   $\varphi'\in C_c^\infty(G')$ be any transfer   of $({\bf 1}_{\tknr\times \kN},0)\in  C_c^\infty(G_{W_0})\times  C_c^\infty(G_{W_1})$.  Then there exists 
a function $\fp'_\corr\in C_c^\infty(G')$  such that, if $\gamma\in G'(F_0)_\rs$ is matched with  $g\in G_{W_1}(F_0)_\rs$, then
\begin{equation*}
   \left\langle \Zx, g\Zx\right\rangle_{\Zx\times\n}= -\frac{1}{2}\del\big(\gamma,  \fp'\big) -  \Orb\big(\gamma,  \fp'_\corr \big).
\end{equation*}
\end{corollary}
\begin{proof} We only give a sketch the proof as it is similar to that of \cite[Lemma 5.18]{RSZ1} assuming the density conjecture \cite[Conj. 5.16]{RSZ1}. Recall that the density conjecture of {\it loc. cit.} implies that the difference of two functions $ \phi'_1,\phi_2'\in C_c^\infty(S_{n+1})$ with the same orbital integrals (for all regular semisimple orbits) must be of the form $\,^{\eta(h)h}\phi-\phi$ for some $h\in\GL_n(F_0)$ and $\phi\in C_c^\infty(S_{n+1})$, where $^{\eta(h)h}\phi\in C_c^\infty(S_{n+1})$ is the function $(^{\eta(h)h}\phi)(\gamma):=\eta(h)\phi(h^{-1}\gamma h)$. It is then straightforward to check that (cf. \cite[Lemma 5.12]{RSZ1})
$$
\del\big(\gamma,   \phi'_1-\phi_2'\big)=\del\big(\gamma, \,^{\eta(h)h}\phi-\phi)=c    \Orb\big(\gamma, \phi)
$$
for some constant $c$ (independent of $\gamma$).
\end{proof}

\begin{proof} (of Theorem \ref{thmzAFL}) 
Let us first prove the inhomogeneous version. We have a cartesian diagram
\begin{equation*}
  \label{eq:projform}
  \begin{aligned}
  \xymatrix{\CZ(u) \ar@{^(->}[r] \ar[d] \ar@{}[rd]|\square & \CZ(u)\times\n \ar[d]^{\iota\times \id}\\ \n \ar@{^(->}[r] & \n\times\n.}
\end{aligned}
\end{equation*} By the projection formula for the morphism $\Zx\times\n\hookrightarrow \n\times\n$, we can relate the intersection numbers,
$$
\left\langle \Zx, (1\times g)\Zx\right\rangle_{\Zx\times\n}=\left\langle \Delta_{\n}, (1\times g)\Zx\right\rangle_{\n\times\n} .
$$
By the semi-Lie algebra version of AFL, Theorem \ref{AFLconj} part \ref{AFL semilie}, we have
$$
\left\langle \Delta_{\n}, (1\times g)\Zx\right\rangle_{\n\times\n}\log q=-\del((\gamma,w_0), {\bf 1}_{K'\times\Lambda'}).$$
Now we apply Lemma \ref{lem Orb red} and take the first derivative. By Leibniz's rule, we have
$$
 \frac \rd{\rd s} \Big|_{s=0}\Orb(\gamma, \phi'_s,s)= \frac \rd{\rd s} \Big|_{s=0}\Orb(\gamma, \phi'_{s=0},s)+(-1)^{n+1}  \Orb(\gamma,{\bf 1}_{ S_{n+1}(O_{F_0})}) \frac \rd{\rd s}\Big|_{s=0} q^{-(n+1)s}.
$$

Recall that, when we make the bijection of the orbits in \eqref{decrsshom}, we need to rescale the Hermitian form, cf. Remark \ref{rem orb}.  Here, after we scale the Hermitian space $W_0$ by a factor $\varpi$, the new Hermitian space has Hasse invariant $(-1)^{n+1}$. Now we distinguish two cases according to the parity of $n$. If $n$ is odd, then for any $\gamma\in G'(F_0)_\rs$  matched with  $g\in G_{W_1}(F_0)_\rs$, we have 
$  \Orb(\gamma,{\bf 1}_{ S_{n+1}(O_{F_0})})=0$ (by the ``easy" part of the Jacquet--Rallis fundamental lemma), hence the second summand vanishes. If $n$ is even, then the second summand does not identically vanish. We have
$$(-1)^{n+1}  \Orb(\gamma, {\bf 1}_{ S_{n+1}(O_{F_0})}) \frac \rd{\rd s}\Big|_{s=0} q^{-(n+1)s}=
(n+1)\log q  \Orb(\gamma,{\bf 1}_{ S_{n+1}(O_{F_0})}) .
$$
This proves the inhomogeneous version.

The homogeneous version follows. Indeed, we obtain  by  Lemma \ref{lem hom2inhom} that
 $$
 \varphi'^\natural-\phi'=(q^{2(n+1)}-1) ((-1)^n{\bf u'}\ast {\bf 1}_{ S_{n+1}(O_{F_0})}- \, {\bf u}\ast {\bf 1}_{ S_{n+1}(O_{F_0})})
 $$
  and $(\varphi'_\corr)^\natural=\phi'_\corr$.
Note that by \eqref{eq u'2u} we have for $\gamma$ matching  $g\in G_{W_1}(F_0)_\rs$,
\begin{align}\label{eq u'u}
   \del(\gamma,{\bf u'}\ast {\bf 1}_{ S_{n+1}(O_{F_0})})= (-1)^n\del(\gamma,{\bf u}\ast {\bf 1}_{ S_{n+1}(O_{F_0})}).
   \end{align}
 Therefore, we get from Corollary \ref{compOr} 
\[
\Orb(\gamma,\varphi'_\corr)=\Orb(\fkr(\gamma),\phi'_\corr), \quad \del(\gamma, \fp')=2\del(\fkr(\gamma), \phi'),\quad  \gamma\in G'(F_0)_\rs.
\]
On the other hand, we have, for $(g_1, g_2)\in(\U(W_1^\flat)\times\U(W_1))(F_0)$,
\[
\left\langle \Zx, (g_1, g_2)\Zx\right\rangle_{\Zx\times\n}=\left\langle \Zx, (1,g_1^{-1}g_2 )\Zx\right\rangle_{\Zx\times\n} .
\]
The result follows because, if $\gamma\in G'(F_0)_\rs$ matches $(g_1, g_2)\in(\U(W_1^\flat)\times\U(W_1))(F_0)_\rs$, then $\fkr(\gamma)\in S_{n+1}(F_0)_\rs$ matches $g_1^{-1}g_2\in\U(W_1)(F_0)_\rs$, comp. \S \ref{ss:gps}.
\end{proof}

\section{The graph version of the AT conjecture of type $(r, 0)$}\label{s:graphv}
Let $n\geq 1$ and let $0\leq r\leq n$. Assuming  that $\tNr$ is regular, we may consider the intersection number $\left\langle \tnr, g\tnr\right\rangle_{\tnr\times\n}$ of \S \ref{sec:vari-arithm-intersGV}. Recall that  Conjecture \ref{conjreg} implies the required regularity; it holds if $r=1$ (Theorem \ref{conj:KRSZ}) or $r$ is even (Proposition \ref{regreven}).

The Haar measures on $\GL_n(F)$ and on $\U(W_0^\flat)(F_0)$ are chosen as in \eqref{normmea} for the given $r$.

\subsection{Reduction to the (quasi-canonical) AFL} 
 
\begin{theorem}\label{conjt}
Assume that $\tnr$ is regular. 

\smallskip

\noindent (i)  \emph{(Inhomogeneous version)} When $r$ is odd, let $\Fp_r'\in C_c^\infty( S_{n+1})$ as in \eqref{varphi nat}. When $r$ is even, let  $\Fp_r'={\bf 1}_{ S_{n+1}(O_{F_0})}\in C_c^\infty( S_{n+1})$.  If $\gamma\in  S_{n+1}(F_0)_\rs$ is matched with  $g\in \U(W_1)(F_0)_\rs$, then  
\begin{equation*}
    \left\langle \tnr, (1,g)\tnr\right\rangle_{\tnr\times\n}\log q=- \del\big(\gamma,  c_r\Fp_r' \big)- \Orb\big(\gamma,  c_r\phi'_{r, \corr} \big)
\end{equation*}
where 
$$
\phi'_{r, \corr}=\begin{cases} (n+1)\,{\bf 1}_{S_{n+1}(O_{F_0})}\cdot \log q
, & $n$ \text{ is even, and $r$ is odd},\\
0, & $n$ \text{ is odd, or  $r$ is even}.
\end{cases}
$$

\smallskip

\noindent (ii)  \emph{(Homogeneous version)}
When $r$ is odd, let $\varphi'_r\in C_c^\infty(G')$ as in \eqref{varphi sharp}. When $r$ is even, let  $\varphi'_r=c_rc'_{r}\mathbf{1}_{\widetilde{K}_n^{\prime [r]}\times K_{n+1}'}\in  C_c^\infty(G')$ (then $\varphi'_r$ is a transfer of $(c_r^2\cdot {\bf 1}_{\tknr\times \kN},0)$). If $\gamma\in G'(F_0)_\rs$ is matched with  $g\in G_{W_1}(F_0)_\rs$, then
\begin{equation*}
  \left\langle \tnr, g\tnr\right\rangle_{\tnr\times\n}\cdot\log q= -\frac{1}{2}\del\big(\gamma,  \varphi'_r \big)- \Orb\big(\gamma,  \fp'_{r, \corr} \big)
\end{equation*}
where 
$$
\fp'_{r,\corr}=\begin{cases} c_r \cdot (n+1)\, {\bf 1}_{G'(O_{F_0})}\cdot \log q
, & $n$ \text{ is even, and $r$ is odd},\\
0, & $n$ \text{ is odd, or  $r$ is even}.
\end{cases}
$$

\end{theorem}

Again, as was the case for the quasi-canonical AFL, we obtain the following corollary. Recall that the choice of Haar measures depends on $r$, which is why we indicate $r$ in the statement of the corollary. 

\begin{corollary}\label{item:conjt3}
Assume the density conjecture \cite[Conj. 5.16]{RSZ1}, and
assume that $\tnr$ is regular.

\,

\noindent (i)  \emph{(Inhomogeneous version)}  Let $\Fp'_r\in C_c^\infty( S_{n+1})$ be any transfer  of  $({\bf 1}_{K_{n+1}}, 0)$. Then there exists 
a function $\phi'_{r, \corr}\in C_c^\infty(G')$  such that, if $\gamma\in  S_{n+1}(F_0)_\rs$ is matched with  $g\in \U(W_1)(F_0)_\rs$, then 
\begin{equation*}
    \left\langle \tnr, (1,g)\tnr\right\rangle_{\tnr\times\n}\cdot\log q= -\del\big(\gamma, c_r \phi'_r \big)- \Orb\big(\gamma,  \phi'_{r, \corr} \big) .
    \end{equation*}
    
    \smallskip

 \noindent (ii)  \emph{(Homogeneous version)}    Let   $\varphi'_r\in C_c^\infty(G')$ be any transfer   of $(c_r^2\cdot {\bf 1}_{\tknr\times \kN},0)$.  Then there exists  
a function $\fp'_{r, \corr}\in C_c^\infty(G')$  such that, if $\gamma\in G'(F_0)_\rs$ is matched with  $g\in G_{W_1}(F_0)_\rs$, then
\begin{equation*}
    \left\langle \tnr, g\tnr\right\rangle_{\tnr\times\n}\cdot\log q=-\frac{1}{2}\del\big(\gamma,  \varphi'_r \big)- \Orb\big(\gamma,  \fp'_{r, \corr} \big) .
\end{equation*}\qed
\end{corollary}
\begin{proof} (of Theorem \ref{conjt})
We first note that part (ii) follows from part (i). In fact, we have $\varphi'^\natural_r=c_r\phi_r'$, as follows from Lemma \ref{lem hom2inhom} when $r$ is odd, and from  the similar identity   $(\mathbf{1}_{\widetilde{K}_n^{\prime [r]}\times K_{n+1}'} )^\natural=\vol( \widetilde{K}_n^{\prime [r]}) {\bf 1}_{S_{n+1}(O_{F_0})}=\frac{1}{c'_{r}} {\bf 1}_{S_{n+1}(O_{F_0})}$,  when $r$ is even.

We now show part (i).
Recall that the group $\GW(F_0)$ naturally acts on $\tnr\times\n$ and $\Zx\times\n$, where $\CZ(u)$ denotes the special divisor for the vector $u$ of length $\varpi^\varepsilon$. We have a cartesian diagram
\begin{equation}
  \label{eq:pushpull}
  \begin{aligned}
  \xymatrix{\tnr \ar@{^(->}[r] \ar[d]^{\pi_2} \ar@{}[rd]|\square & \tnr\times\n \ar[d]^{\pi_2\times \id}\\ \Zx \ar@{^(->}[r] & \Zx\times\n.}
\end{aligned}
\end{equation} The  morphism $\pi_2\times\id: \tnr\times\n\rightarrow \Zx\times\n$ is $\GW(F_0)$-equivariant.  

In view of the AFL (for $r$ even, see Theorem \ref{AFLconj}), resp. the quasi-canonical AFL (for $r$ odd, see Theorem \ref{thmzAFL}), it suffices to prove the identity 
\[
\left\langle \tnr, g\tnr\right\rangle_{\tnr\times\n}=c_r \left\langle \CZ(u), g\CZ(u)\right\rangle_{\CZ(u)\times\n}, \quad g\in {G_{W_1}(F_0)}_\rs .
\]
This identity follows from the projection formula for the proper morphism $\pi_2\times\id:\tnr\times\n\rightarrow \Zx\times\n$ and the following Lemma.
\end{proof}
\begin{lemma} \label{prop:pushpull} Let $g\in\GW(F_0)$.
  \begin{altenumerate}
  \item The identity $(\pi_2\times\id)^*(g\Zx)=g\wt\CN_n^{[r]}$ holds in $K_0'(\wt\CN_n^{[r]}\times \n)$.
  \item The identity $(\pi_{2}\times\id)_*(g\wt\CN_n^{[r]})=c_rg\Zx$ holds in $\Gr^{n}K_0^{g\Zx}(\Zx\times \n)$.
  \end{altenumerate}
\end{lemma}

\begin{proof} By the $\GW(F_0)$-equivariance of $\pi_2\times\id$, it suffices to consider the case $g=1$.
  \begin{altenumerate}
  \item  Let $\tilde z$ be a point of $\tnr$ and let  $z=\pi_2(\tilde z)$ be the image point in $\Zx$. Let $R=\mathcal{O}_{\Zx \times \n,z}$ and $S=\mathcal{O}_{\tnr\times\n,\tilde z}$. Since $\Zx$ and $\Zx\times\n$ are both regular, we know that $\Zx$ is locally defined by a regular sequence $f_1,\ldots,f_{n}$ in $R$. Thus the $R$-module $O_{\Zx,z}$ has a free resolution given by the Koszul complex $K(f_1,\ldots,f_{n})$. Let $\tilde f_1,\ldots,\tilde f_{n}$ be the images  of $f_1,\ldots,f_{n}$ under the morphism $R\rightarrow S$ induced by $\pi_2\times \id$. Since the diagram (\ref{eq:pushpull}) is cartesian, and both $\tnr$ and $\tnr\times\n$ are regular by our assumption on $\tnr$ so that $\tnr\hookrightarrow \tnr\times\n$ is a regular immersion of pure codimension $n$, we know that $\tnr$ is locally defined by the regular sequence $\tilde f_1,\ldots,\tilde f_{n}$ in $S$. Thus the $S$-module $O_{\tnr, \tilde z}$ has a free resolution given by the Koszul complex $K(\tilde f_1,\ldots,\tilde f_{n})$. We have an isomorphism of  complexes of $S$-modules $$K(f_1,\ldots f_{n}) \otimes_RS \simeq K(\tilde f_1,\ldots\tilde f_{n}),$$ which gives the desired identity $(\pi_2\times\id)^*(\Zx)=\tnr$ at $\tilde z$ by the definition of $(\pi_2\times\id)^*$.
  \item  The formal scheme  $\CZ(u)$ is regular of dimension $n$ and the morphism $\tnr\to \CZ(u)$  is finite in its generic fiber. Let $m_r$ be its degree. Then the  coherent sheaf $\pi_{2, *}(\CO_\tnr)$ coincides with  a free $\CO_{\CZ(u)}$-sheaf of rank $m_r$ up to coherent sheaves with support of dimension strictly smaller than $n$.  It  follows  that the difference $m_r\Zx-(\pi_{2}\times\id)_*(\tnr)$ of elements of ${\rm Fil}^nK_0^{\Zx}(\Zx\times \n)$ has zero image  in $\Gr^{n}K_0^{\Zx}(\Zx\times \n)$. The result follows because $m_r=[\tilde{K}_n^{[\varepsilon]}:\tilde {K}_n^{[r]}]=c_r$. 
  
 \qedhere
  \end{altenumerate}
\end{proof}

\section{AT conjecture of type $(r, 0)$: the case  $r$ even}\label{s:ATCeven}

In the next  sections, we will be concerned with the AT conjectures, i.e., the arithmetic intersection number of  \S \ref{sec:arithm-inters-numbAT}, resp. of \S \ref{sec:arithm-inters-numbAT(r,0)}. In this section we consider the AT conjecture of type $(r, 0)$ in the case when   $r$ is even. We will reduce the problem to the FL and AFL for certain  (non-unit)  elements in the spherical Hecke algebra.

\subsection{An explicit transfer: an application of FL for the whole Hecke algebra}

 Since $r$ is even, i.e., $\varepsilon=0$, we have a direct sum decomposition $\Lambda_0= \Lambda^\flat_0\obot \langle u_0\rangle$ where $\Lambda_0^\flat\in \Ver^0(W^\flat_0)$, which necessarily satisfies $\Lambda^\flat\subset \Lambda^\flat_0$, cf. \eqref{def Lambda06}.
Recall from \eqref{eq K' K''} and \eqref{defKn'}
\begin{align}\label{eq:cap1}
K_n' =\kN'\cap \GL_n(F)=   \GL_{O_F}(\Lambda^\flat_0).
\end{align}
We also define $K_{n+1}=\U(\Lambda_0)$ and 
\begin{align}\label{eq:cap2}
K_n:=\kN\cap  \U(W_0^\flat)(F_0)= \U(\Lambda_0^\flat ) ,
\end{align}
(in terms of \eqref{defKrsimple}, we have $K_n=K_n^{[0]}$).
We continue with the  choice of the Haar measures such that 
$$
\vol(K_n)=1,\quad \vol(K_n')=1,
$$
(this normalization is consistent with \eqref{normmea} since $K_{n}^{[0]}=\wit{K}_{n}^{[0]}$, cf. Remark \ref{explK} ~(i)).
Recall from \cite[\S4]{LRZ} the atomic  Hecke function in the spherical Hecke algebra $\CH_{K_n}$, defined as the  convolution,
\begin{equation}\label{def phir9}
\varphi_r:=\frac{1}{\vol(\knr)} {\bf 1}_{K_n\knr}\ast {\bf 1}_{\knr K_n}.
\end{equation}
We also recall from \cite[\S3.6]{LRZ}  the base change homomorphisms between spherical Hecke algebras,
$$ \Bc_{n}: \CH_{K^{\prime}_n}\rightarrow \CH_{K_{n}}, \quad \Bc_{n+1}: \CH_{K'_{n+1}}\rightarrow\CH_{K_{n+1}} ,
$$ and $$\Bc=\Bc_{n}\otimes\Bc_{n+1}: \CH_{K^{\prime}_n}\otimes_{\BQ}\CH_{K^{\prime}_{n+1}}\rightarrow \CH_{K_{n}} \otimes_{\BQ}\CH_{K_{n+1}}.$$
Note that all of them are surjective. Recall from Remark \ref{explK}, (ii) that $\tilde K_n^{[r]}=K_n\cap\knr=K_n^{[r,0]}$.

\begin{proposition}\label{FL r} Recall that $r$ is even. Let $\varphi'$ be any element in $\CH_{K^{\prime}_n}\otimes_{\BQ}\CH_{K^{\prime}_{n+1}}$ such that 
$$\Bc(\varphi')= \varphi_r\otimes {\bf 1}_{ K_{n+1}}\in \CH_{K_n}\otimes_\BQ\CH_{K_{n+1}}.$$
Then the function $\varphi'$ is a transfer of  $(c_r^2 \,{\bf 1}_{\knr\times \kN},0)\in C_c^\infty(G_{W_0})\times C_c^\infty(G_{W_1})$.
\end{proposition}

\begin{proof}
Taking into account that $c_r=\vol(\wt K_n^{[r]})^{-1}=\vol(K_n^{[r,0]})^{-1},$ this follows from Lemma \ref{lem:Orb r even} below and the Jacquet--Rallis fundamental lemma for the full Hecke algebra due to Leslie \cite{Les}, cf. \cite[Thm. 3.7.1]{LRZ}.
\end{proof}

\begin{lemma}\label{lem:Orb r even}
For every $g\in G_{W_0}(F_0)_\rs$,
    \begin{align*}
  \Orb(g, \frac{1}{\vol(K_n^{[r,0]})^2}{\bf 1}_{\knr\times K_{n+1}} )= \Orb(g, \varphi_r\otimes {\bf 1}_{ K_{n+1}} ).
   \end{align*}

\end{lemma}

\begin{proof}
 We recall the definition \eqref{def Orb U} of the orbital integral,
\begin{equation*}
   \Orb(g, f) = \int_{ H(F_0) \times H(F_0)} f(h_1^{-1} g h_2)\, dh_1\, dh_2 . 
\end{equation*}
It follows that,  for any  $\phi_1,\phi_2\in C_c^{\infty}(H)$, 
  $$
   \Orb(g,\phi_1\ast f \ast\phi_2)= c(\phi_1)c(\phi_2)  \Orb(g, f),
   $$
   where $$c(\phi):= \int_{H(F_0)}\phi(h)\,dh.$$
Here the convolution is defined in the usual way induced by the two actions of $H(F_0)$ on $G(F_0)$: for any $\phi\in C_c^{\infty}(H)$ and  $f\in C_c^{\infty}(G)$,
   $$
(   \phi\ast f)(g):=\int_{H(F_0)} \phi(h)f(h^{-1}g)\,dh,
   $$
   and
      $$
(   f\ast \phi)(g):=\int_{H(F_0)} \phi(h)f(gh)\,dh.
   $$In particular, we have
    \begin{equation}\label{eq Orb}
   \Orb(g, f)=    \Orb(g,{\bf 1}_{\Delta(K_n)}\ast f \ast {\bf 1}_{\Delta(K_n)}) ,
   \end{equation}
   where we have used the fact that $\vol(K_n)=1$. Here $\Delta(K_n)$ is the image of $K_n$ under the inclusion $ \U(W^\flat_0)\simeq H\subset G=\U(W^\flat_0)\times \U(W_0)$.
 
   We apply the consideration to $f={\bf 1}_{\knr\times \kN}$. Using \eqref{eq:cap2}, we see that  $K_{n+1}$ is bi-$K_n$-invariant. It follows from a substitution in the integral defining the convolution that we have 
   \begin{align}\label{eq:conv}
   {\bf 1}_{\Delta(K_n)}\ast
   {\bf 1}_{\knr\times \kN}\ast {\bf 1}_{\Delta(K_n)}=({\bf 1}_{K_n}\ast
   {\bf 1}_{\knr}\ast {\bf 1}_{K_n} )\otimes {\bf 1}_{\kN}.
   \end{align}
 Here the convolution is the usual one defined for any two functions in $C_c^{\infty}(\U(W^\flat_0))$. 
   To compute the triple convolution on the RHS, we note
   $$
   {\bf 1}_{K_n}\ast
   {\bf 1}_{\knr}=\vol(K_n^{[r,0]})  {\bf 1}_{K_n\knr}, \quad    {\bf 1}_{\knr} \ast {\bf 1}_{K_n}=\vol(K_n^{[r,0]})  {\bf 1}_{\knr K_n}.
   $$
In fact, the  first convolution is left-$K_n$-invariant and right-$\knr$-invariant and has support in $K_n\knr$. Therefore it suffices to compare the values of both sides at $g=1$: the left hand side gives the volume of $K_n\cap\knr=K_n^{[r,0]}$, which verifies the first identity. The argument works for the second identity  as well.
   Therefore, we have
   \begin{align*}
   {\bf 1}_{K_n}\ast
   {\bf 1}_{\knr}\ast {\bf 1}_{K_n}& =\vol(\knr)^{-1}{\bf 1}_{K_n}\ast
   {\bf 1}_{\knr}\ast {\bf 1}_{\knr} \ast {\bf 1}_{K_n} \\\
   &=\vol(K_n^{[r,0]})^{2}\vol(\knr)^{-1}{\bf 1}_{K_n\knr}\ast {\bf 1}_{\knr K_n}\\
   &=\vol(K_n^{[r,0]})^{2}\varphi_r ,
   \end{align*}
   where the last equality follows from the definition of  $\varphi_r$, see \eqref{def phir9}. The lemma follows.
  
\end{proof}

We present an alternative proof by the ``lattice counting interpretation" of orbital integrals, for the reason that the latter gives the heuristics on how to formulate the arithmetic transfer conjectures and will appear repeated in later sections. Following the notation of \cite[\S4.1]{LRZ}, we have
a diagram (analogous to the one for the RZ spaces in the next section), which is an analog of \eqref{eq Nrs},
\begin{equation}\label{Hk cor r}
\begin{aligned}
\xymatrix{&\BN^{[r,0]}_n \ar[rd]  \ar[ld] & \\ \BN^{[r]}_n  & &\BN^{[0]}_{n+1} . } 
\end{aligned}
\end{equation}
Here $\BN^{[0]}_{n+1}$, resp. $\BN^{[r]}_n$, denotes the set of vertex lattices of type $0$ in $W_0$, resp. of type $r$ in $W_0^\flat$, and $\BN^{[r,0]}_n$ consists of pairs $(\Lambda^\flat,\Lambda^\flat_{0})\in \BN_n^{[r]}\times \BN_n^{[0]}$ such that $ \Lambda^\flat\subset \Lambda^\flat_{0}$, and the two maps record $\Lambda^\flat$ and $\Lambda^\flat_0\oplus\pair{u_0}$ respectively. Combining the two maps above, we obtain an injective map,
$$
\BN^{[r,0]}_n\to \BN^{[r]}_n\times \BN^{[0]}_{n+1}.
$$

We also have the Hecke correspondence $\BT_{\leq r}$ which   consists of the triples  $(\Lambda^\flat,\Lambda^\flat_{0},\Lambda'^\flat_{0}  )\in \BN_n^{[r]}\times \BN_n^{[0]}\times \BN_n^{[0]}$ such that  $ \Lambda^\flat\subset \Lambda^\flat_{0}\cap  \Lambda'^\flat_{0}$. In other words,  $\BT_{\leq r}$  is the composition of 
  the obvious correspondence with its transpose (comp.  \cite[(4.1.5)]{LRZ}),
  \begin{equation}\label{Hklat}
\begin{aligned}
\xymatrix{&&\BT_{\leq r} \ar[rd]  \ar[ld] & \\ &\BN_n^{[0,r]} \ar[rd]  \ar[ld] & &\BN_n^{[r,0]} \ar[rd]  \ar[ld] &\\ \BN_n^{[0]} &&  \BN_n^{[r]}&&  \BN_n^{[0]}.}
\end{aligned}
\end{equation}
We form the cartesian product $\BN_{n}^{[r]}(g)$,
\begin{equation}\label{cartdi}
\begin{aligned}\xymatrix{ \BN_{n}^{[r]}(g)\ar[d]\ar[r] & \BT_{\leq r}\times \Delta_{  \BN^{[0]}_{n+1}} \ar[d]\\
\BN^{[0]}_n \times \BN^{[0]}_n\,\,  \ar[r]^-{(\id,g)}& \,\,(\BN^{[0]}_n\times \BN^{[0]}_{n+1})\times (\BN^{[0]}_n\times \BN^{[0]}_{n+1}) .
}
\end{aligned}
\end{equation}
The lower horizontal map in \eqref{cartdi} maps $(\Lambda^\flat_{0}, \Lambda^{\flat \prime}_{0})$ to  $\big((\Lambda^\flat_{0},  \Lambda^\flat_{0}\oplus \langle u_0\rangle), g(\Lambda^{\flat \prime}_{0}, \Lambda^{\flat \prime}_{0}\oplus \langle u_0\rangle)\big)$. The right vertical map maps $\big((\Lambda^\flat,\Lambda^\flat_{0},\Lambda^{\flat \prime}_{0}  ), \Lambda_0\big)$ to $\big((\Lambda^\flat_{0},\Lambda_0),(\Lambda^{\flat \prime}_{0}  , \Lambda_0)\big)$. 

\begin{lemma}
\label{lem orb=lat}
\begin{altenumerate}
Let $g\in G_{W_0}(F_0)$ be regular semisimple.
\item
We have
\begin{equation}\label{sharporb}
\Orb(g, \vol(K_n^{[r,0]})^{-2}\, {\bf 1}_{\knr\times K_{n+1}} )= \#(\BN^{[r,0]}_n\cap g\BN^{[r,0]}_n)= \#\BN_{n}^{[r]}(g) ,
\end{equation}
where the second term is the cardinality of the intersection of  two subsets  of $\BN^{[r]}_n\times \BN^{[0]}_{n+1}$.
\item
We have
\begin{equation}\label{sharporb hk}
 \Orb(g, \varphi_r\otimes {\bf 1}_{ K_{n+1}} )= \#\BN_{n}^{[r]}(g).
\end{equation}
\end{altenumerate}

\end{lemma}

\begin{proof}It obviously suffices to consider elements $g$ of the form $g=(1, g^\sharp)$, with $g^\sharp\in \U_{W_0}(F_0)$. 

In part (i), the first identity  is an easy exercise, by unfolding the orbital integral.     
Unpacking the definitions, the intersection $\BN^{[r,0]}_n\cap g\BN^{[r,0]}_n$ is in bijection with   the set of triples $(\Lambda^\flat,\Lambda^\flat_{0}, \Lambda^{\flat \prime}_{0})\in \BN_n^{[r]}\times \BN_n^{[0]} \times \BN_n^{[0]}$ such that $ \Lambda^\flat\subset \Lambda^\flat_{0}\cap  \Lambda^{\flat\prime}_{0}$  and 
$g^\sharp ( \Lambda^{\flat\prime}_{0} \oplus\pair{u_0})= \Lambda^\flat_{0} \oplus\pair{u_0}$. On the other hand, this last set of triples is  easily seen to be in bijection with the cartesian product $\BN_{n}^{[r]}(g)$ in \eqref{cartdi}, and this proves part (i).

For part (ii), we note that for any spherical Hecke function of the form $\varphi =\varphi_r\otimes \varphi_{r'}$ on $G_{W_0}(F_0)$,
 we have $$
  \Orb((1, g^\sharp), \varphi)=\sum_{h_1, h_2\in \U(W_0^\flat)/K_n}\varphi(h_1^{-1}(1, g^\sharp) h_2).
  $$ 
We may naturally identify   $\U(W_0^\flat)/K_n$  with the set $\Ver^0(W_0^\flat)$. Then the index set in the above sum is in bijection with  the set of pairs $(\Lambda_0^\flat,\Lambda_0^{'\flat})\in \Ver^0(W_0^\flat)\times \Ver^0(W_0^\flat)$ such that 
the relative position of $\Lambda_0^\flat$ and  $\Lambda_0^{'\flat}$ (resp.,  of $\Lambda_0^\flat\oplus \langle u_0\rangle$ and  $g^\sharp(\Lambda_0^{'\flat}\oplus \langle u_0\rangle)$) is stipulated by $\varphi_r$  (resp. by  $\varphi_{r'}$). Taking now $ \varphi_{r'}= {\bf 1}_{ K_{n+1}}$, it is easy to see that the set of such pairs is bijective to the image of the left vertical map in \eqref{cartdi}. Moreover, the weight factor $\varphi(h_1^{-1}(1, g^\sharp) h_2)$ is exactly the size of the fiber of this map.  This concludes the proof of part (ii).

  \end{proof}

\begin{remark}The natural-looking  candidate ${\bf 1}_{ {K}_n^{\prime [r]}\times K_{n+1}'} $ appears to fail to be a transfer  (up to a constant multiple) of $ {\bf 1}_{{K}_n^{ [r]}\times K_{n+1}}$. We can show this at least when $n=r=2$ and the method below should work in general. Indeed, by the method of the proof of the lemma, we have the identity for $G'$, 
  $$
   \Orb(\gamma, {\bf 1}_{{K}_n^{\prime [r]}\times K_{n+1}'})= \vol(K_n^{\prime [r,0] })   \Orb(\gamma ,  {\bf 1}_{K'_n {K}_n^{\prime [r]}\times K_{n+1}'} ) ,
   $$
   where $K_n^{\prime [r,0] }=K_n' \cap K_n^{\prime [r]}$.  (Note that we only use the action of $H_1'$ from the left.)
Note that the function ${\bf 1}_{K'_n {K}_n^{\prime [r]}\times K'_{n+1}} $ is not in the spherical Hecke algebra but that the image $r^{\eta^{n-1}}({\bf 1}_{K'_n K_n^{\prime [r]}} )$ lies in $\CH_{S_n}$ (in \cite[\S3.7]{LRZ},  the definition of the map $r^{\eta^{n-1}}$ makes sense for any function in $C_c^\infty(\U(W_0^\flat))$; in particular we can apply it to the function  ${\bf 1}_{K'_n {K}_n^{\prime [r]}}$). Hence, 
using the isomorphism $ \Bc^\eta_{S_n}$ in the commutative diagram from \cite{LRZ},
 \begin{equation}\label{eq:diag1}
 \begin{aligned}
\xymatrix{ \CH_{K'_n} \ar[r]^-{r^{\eta^{n-1}}_\ast} \ar[dr]_-{\Bc_n }& \CH_{K'_{S_n}}\ar[d]^-{ \Bc^\eta_{S_n}}\\
& \CH_K} 
\end{aligned}
\end{equation}
it suffices to compare
$
 \Bc^\eta_{S_n}( r^{\eta^{n-1}}({\bf 1}_{K'_n K_n^{\prime [r]}} )) $ and $ \varphi_r$
(up to a constant multiple). One can show that the function ${\bf 1}_{K'_n K_n^{\prime [r]}} $ is equivalent   to the spherical function ${\bf 1}_{K'_n \varpi^{(1^{r/2},0^{n-r/2})}  K'_n}$ (i.e., has  the same regular semi-simple orbital integrals, cf. Definition \ref{defequ}). The question is now to compare 
$\Bc({\bf 1}_{K'_n \varpi^{(1^{r/2},0^{n-r/2})}  K'_n} )$ with $\varphi_r$ (up to a constant multiple). 
Let us consider the special case $n=r=2$. Using notation and results from \cite[\S7]{LRZ}, the Satake transform of ${\bf 1}_{K'_2 \varpi^{(1,0)}  K'_2} $  is $q(X+X^{-1})$, and the Satake transform of $\varphi_2$ is 
$$(q+1)+\phi_1=(q+1)+(q(X+1+X^{-1})-1)=q(X+2+X^{-1}).
$$ They do not match!
\end{remark}

\subsection{The AT conjecture} We continue to assume that $r$ is even. We identify $W_1$ with $\mathbb{V}_{n+1}$ defined in \S\ref{sec:herm-space-mathbbv} and choose the special vector in $W_1$ to be $\x\in \mathbb{V}_{n+1}$ defined in \S\ref{sec:relat-with-spec}.  Then we may identify the hermitian space $\U(W_1)$, resp. $\U(W_1^\flat)$, with $\U(\mathbb{V}_{n+1})$, resp. $\U(\mathbb{W}^{[r]}_{n})$. Then $G_{W_1}(F_0)$ acts on $\nr\times\n$  via this identification and hence the arithmetic intersection number $\langle \tnr, g\tnr\rangle_{\nr\times\n}$  defined in \S\ref{sec:arithm-inters-numbAT} makes sense for $g\in G_{W_1}(F_0)$.

\begin{conjecture}\label{conjreven}
Recall that $r$ is even. Let $\varphi'$ be any element in $\CH_{K^{\prime}_n}\otimes_{\BQ}\CH_{K^{\prime}_{n+1}}$ such that 
$$\Bc(\varphi')=  \varphi_r\otimes {\bf 1}_{ K_{n+1}}\in \CH_{K_n}\otimes_\BQ\CH_{K_{n+1}}.$$ 
(Then $\varphi'$ is a transfer  of $(c_r^2\,{\bf 1}_{K_n^{[r]}\times K_{n+1}},0)\in C_c^\infty(G_{W_0})\times C_c^\infty(G_{W_1})$.)
\begin{altenumerate}
\item\label{item:conjr1} If $\gamma\in G'(F_0)_\rs$ is matched with  $g\in G_{W_1}(F_0)_\rs$, then \begin{equation*}
    \left\langle \tnr, g\tnr\right\rangle_{\nr\times\n} \cdot\log q=-\frac{1}{2}\del\big(\gamma,  \varphi' \big).
\end{equation*}
\item\label{item:conjr2} For any $\tilde\fp'\sim \varphi'$ (i.e., $\tilde\fp'$ and $\varphi'$ have identical  regular semi-simple orbital integrals, cf. Definition \ref{defequ}), there exists $\fp'_\corr\in C_c^\infty(G')$  such that if $\gamma\in G'(F_0)_\rs$ is matched with  $g\in G_{W_1}(F_0)_\rs$, then
\begin{equation*}
 \left\langle \tnr, g\tnr\right\rangle_{\nr\times\n}\cdot\log q= -\frac{1}{2}\del\big(\gamma,  \tilde\fp' \big) -  \Orb\big(\gamma,  \fp'_\corr \big).
\end{equation*}
\end{altenumerate}
\end{conjecture}

\begin{remark}
\noindent (i) When $r=0$,  Conjecture \ref{conjreven} (i) recovers the (homogeneous version of) arithmetic fundamental lemma. It seems hard to formulate an inhomogeneous version of  Conjecture \ref{conjreven} beyond the case $r=0$.

\smallskip

\noindent (ii) Part (ii) of Conjecture \ref{conjreven} follows from part (i) and  the density conjecture \cite[Conj. 5.16]{RSZ1}.

\end{remark}

It turns out that
Conjecture \ref{conjreven} is a consequence of the AFL conjecture for the Hecke correspondence formulated in \cite{LRZ}. We recall the statement of the latter.
\begin{conjecture}\label{AFL Hk}
(AFL for the spherical Hecke algebra, homogeneous version \cite[Conj. 6.1.4]{LRZ}.)
Let $\fp'\in \CH_{K^{\prime}_n}\otimes_{\BQ}\CH_{K^{\prime}_{n+1}}$, and let $\fp={\rm BC}(\fp')\in \CH_{K_n}\otimes_{\BQ}\CH_{K_{n+1}}$. Then
  \begin{equation*}
   \left\langle \BT_\varphi(\Delta_{\CN_n^{[0]}}), g\Delta_{\CN_n^{[0]}}\right\rangle_{\CN_n^{[0]}\times \CN_{n+1}^{[0]}} \cdot\log q= -\frac{1}{2}\del\big(\gamma,  \fp'\big),
\end{equation*}
whenever $\gamma\in G'(F_0)_\rs$ is matched with  $g\in G_{W_1}(F_0)_\rs$. 
\end{conjecture}
Here $\BT_\varphi$ is the Hecke operator on K-theory defined in \cite[\S5]{LRZ}. The full definition of $\BT_\varphi$  is delicate; for our purpose, the case of $\varphi_r\otimes{\bf 1}_{K_{n+1}}$ is sufficient.
We
recall from \cite[\S5.5]{LRZ} that the Hecke correspondence on $\CN_n$ associated to $\varphi_r$ is defined as
$$
\CT_n^{\leq r}= \CN_n^{[0,r]}\circ \CN_n^{[r,0]}.
$$
More precisely we have the following diagram (analogous to \eqref{Hklat}) (in which, as explained in \cite[\S 5]{LRZ}, the cartesian square has to be interpreted in the framework  of \emph{derived algebraic geometry}), 
\begin{equation}\label{diagCN}
\begin{aligned}
\xymatrix{&&\CT_n^{\leq r} \ar[rd]  \ar[ld] & \\ &\CN_n^{[0,r]} \ar[rd]  \ar[ld] & &\CN_n^{[r,0]} \ar[rd]  \ar[ld] &\\ \CN^{[0]}_n &&  \CN_n^{[r]}&&  \CN^{[0]}_n.}
\end{aligned}
\end{equation}

By abuse of notation, we also denote by $\CT_n^{\leq r}$ the Hecke correspondence $\CT_n^{\leq r}\times \Delta_{\CN^{[0]}_{n+1}}$ on $\CN^{[0]}_n\times \CN^{[0]}_{n+1}$. Then $\BT_{\varphi_r\otimes {\bf 1}_{K_{n+1}}}$ is the Hecke operator induced by the Hecke correspondence $\CT_n^{\leq r}$ in K-theory, cf. \cite[\S 9]{LRZ}. 

\begin{lemma}\label{lem Int r even}
For every $g\in G_{W_1}(F_0)_\rs$,
$$
\left\langle \tnr, g\tnr\right\rangle_{\nr\times \CN_{n+1}^{[0]}}=\left\langle \BT_{\varphi_r\otimes {\bf 1}_{K_{n+1}}}( \Delta_{\CN_n^{[0]}}),g\Delta_{\CN_n^{[0]}}  \right\rangle_{\CN^{[0]}_n\times \CN^{[0]}_{n+1}}.
$$
\end{lemma}
\begin{proof}This is modeled on the  ``alternative proof" of  Lemma \ref{lem:Orb r even}, but we replace the sets $\BN^{[r]}_n, \BN^{[0]}_n,$ etc. by the respective RZ spaces.  Recall that    $\tNr=\CN^{[r,0]}_n$, comp. the proof of Proposition \ref{regreven}. Consider the following correspondence
\begin{equation}
\label{cor r 0}
\begin{aligned}
\xymatrix{&\CN^{[r,0]}_n\times\Delta_{ \CN^{[0]}_{n+1}} \ar[rd]^{\wt\pi_2}  \ar[ld]_{\wt\pi_1} & \\ \CN^{[0]}_n\times \CN^{[0]}_{n+1}  & &\CN^{[r]}_{n}\times  \CN^{[0]}_{n+1}.} 
\end{aligned}
\end{equation}
Then for $\CN_n^{[0]}$ viewed as a class $\Delta_{\CN_n^{[0]}}$ in $K_0^{\CN_n^{[0]}}(\CN^{[0]}_n\times \CN^{[0]}_{n+1})$, we have 
\begin{equation}\label{eq cor N0 Nr}
(\wt\pi_2)_\ast\wt\pi_1^\ast(\Delta_{\CN_n^{[0]}})=\CN_n^{[r,0]} ,
\end{equation}
as classes in $\Gr^{n}K_0^{\CN_n^{[r,0]}}(\CN^{[r]}_n\times \CN^{[0]}_{n+1})$.
It therefore follows from the projection formula that
\begin{align*}
\left\langle \tnr, g\tnr\right\rangle_{\nr\times \CN_{n+1}^{[0]}}=&\left\langle \tnr, (\wt\pi_2)_\ast\wt\pi_1^\ast(g\Delta_{\CN_n^{[0]}})\right\rangle_{\nr\times \CN_{n+1}^{[0]}} \\
=&\left\langle (\wt\pi_1)_\ast\wt\pi_2^\ast (\tnr), g\Delta_{\CN_n^{[0]}}\right\rangle_{\CN_n^{[0]}\times \CN_{n+1}^{[0]}} ,
\end{align*}
where we note that the action of $g$ commutes with the action of the correspondences. 
Using \eqref{eq cor N0 Nr} again, we obtain
$$
 (\wt\pi_1)_\ast\wt\pi_2^\ast (\tnr)= (\wt\pi_1)_\ast\wt\pi_2^\ast ((\wt\pi_2)_\ast\wt\pi_1^\ast(\Delta_{\CN_n^{[0]}}))=(\CT_n^{\leq r})_*(\Delta_{\CN_n^{[0]}})
$$
by the definition of the Hecke correspondence $\CT_n^{\leq r}$. Therefore we arrive at the desired assertion
\begin{align*}
\left\langle \tnr, g\tnr\right\rangle_{\nr\times \CN_{n+1}^{[0]}}=\left\langle \BT_{\varphi_r\otimes {\bf 1}_{K_{n+1}}}(\Delta_{\CN_n^{[0]}}), g\Delta_{\CN_n^{[0]}}\right\rangle_{\CN_n^{[0]}\times \CN_{n+1}^{[0]}} .
\end{align*}

  \end{proof}
  
  \begin{corollary}\label{CorAFLwh}
  Recall that $r$ is even. Conjecture 
\ref{AFL Hk} (for $\varphi=\varphi_r\otimes {\bf 1}_{K_{n+1}}$) implies Conjecture  \ref{conjreven} (i).
  \end{corollary}
Unfortunately, we do not know of any instance in which Conjecture 
\ref{AFL Hk} is known, if $n\geq 2$. 

\section{AT conjecture of type $(0,r)$: the case $r$ even}\label{s:AT(0,r)}

We continue to denote by $W_0$ (resp. $W_1$) a split (resp. non-split) Hermitian space of dimension $n+1$.

\subsection{Open compact subgroups} \label{ss: cpt open 0 r}
 Let $0\le r\le n+1$. In this subsection $r$ can be even or odd. We first define open compact subgroups on the unitary group side. 

Recall the parity $\varepsilon=\varepsilon(r)\in\{0,1\}$ of $r$, defined in (\ref{eq:epsilon}). Fix a special vector $u_0$ of norm $\varpi^\varepsilon$  in $W_\varepsilon$. Then $W^\flat:=\pair{u_0}^\perp$ is a split hermitian space. We  fix a self-dual lattice  $\Lambda^\flat_0\in \Ver^0(W^\flat)$. Let $\Lambda$ be a lattice of type $r$ in $W_\varepsilon$. We impose the following condition
\begin{equation}\label{def Lambda r}
\Lambda_0^\flat \obot \langle u_0\rangle\supseteq \Lambda .
\end{equation} 
We let $ \Lambda^+$ be such a lattice  satisfying the additional condition
\begin{equation}\label{def Lambda r +}
 u_0\in  \Lambda ,
\end{equation} 
(if it exists) and  let $ \Lambda^-$ be such a lattice  (if it exists) satisfying
\begin{equation}\label{def Lambda r +}
 u_0\notin  \Lambda.
\end{equation} 
We denote by $K_n=K_n^{[0]}\subset \U(W^\flat)(F_0)$ the stabilizer of the self-dual lattice $\Lambda_0^\flat$, resp. $K_{n+1}^{[\varepsilon]}\subset \U(W_\varepsilon)(F_0)$ the stabilizer of the vertex lattice $\Lambda_0^\flat \obot \langle u_0\rangle$ of type $\varepsilon$, resp. $K_{n+1}^{[r],+}\subset \U(W_\varepsilon)(F_0)$ and $K_{n+1}^{[r],-}\subset \U(W_\varepsilon)(F_0)$ the stabilizers of the   type $r$ vertex lattices  
$ \Lambda^+$ and $ \Lambda^-$, respectively. 
\begin{lemma}
\label{lem 2 orb}
Consider the action of the group $\U(W^\flat)$ on the set of pairs $(\Lambda_0^\flat, \Lambda)$ satisfying \eqref{def Lambda r}.
\begin{altenumerate}
\item When $r$ is odd or $r=0$, there is exactly one orbit with the representative given by $(\Lambda_0^\flat, \Lambda^+)$.
\item When $r$ is even and $r=n+1$, there is exactly one orbit with the representative given by $(\Lambda_0^\flat, \Lambda^-)$.
\item  When $r$ is even and $1\le r\le n$,
there are exactly two orbits with representatives given by $(\Lambda_0^\flat, \Lambda^+)$ and $(\Lambda_0^\flat, \Lambda^-)$.
\end{altenumerate}
\end{lemma}
\begin{proof}The proof is similar to that of Lemma \ref{lem transitive 1}. We can assume $\Lambda_0^\flat$ is the standard vertex lattice of type $0$. Let $\ov W=\varpi^{-1}(\Lambda_0^\flat \obot \langle u_0\rangle)/ (\Lambda_0^\flat \obot \langle u_0\rangle)^\vee$ be the  $\BF_{q^2}/\BF_q$-hermitian space of dimension $n+1-\varepsilon$ with the induced (non-degenerate) hermitian form. Then the set of type $r$ lattices $\Lambda$ contained in $\Lambda_0^\flat\obot \langle u_0\rangle$  corresponds to the set of  isotropic subspaces of $\ov W$ of dimension $(r-\varepsilon)/2$ (sending $\Lambda$ to $\Lambda^\vee/(\Lambda_0^\flat \obot \langle u_0\rangle)^\vee$). 
We need to consider the action of $K_n$ on the set of such $\Lambda$. 
We have a subspace  $\varpi^{-1}\langle\ov u_0\rangle/\langle\ov u_0\rangle^\vee\subset \ov W$ generated by  the class $\ov u_0$ of $ \varpi^{-1}u_0$ in $\ov W$.  Note that this subspace is zero in the first two cases and non-degenerate in the third case. We let $\ov W^\flat$ be the orthogonal complement of $\pair{\ov u_0}$.  Then the reduction of $K_{n}$ is the subgroup $\U(\ov W^\flat)$ of the finite unitary group $\U(\ov W)$. We are reduced to considering the action of $\U(\ov W^\flat)$  on the set of  isotropic subspaces $N$ of $\ov W$ of dimension $(r-\varepsilon)/2$. 

If $\varepsilon=1$, then $\ov W^\flat=\ov W^\flat$ and Witt's theorem implies that $\U(\ov W)$ acts transitively. Now assume  $\varepsilon=0$ so that $\pair{\ov u_0}$ is an anisotropic line. Then Witt's theorem implies that there are at most two orbits, depending on whether $\pair{\ov u_0}\perp N$ or not. Only one orbit  exists when $r=0$ (then we necessarily have $\pair{\ov u_0}\perp N$ ) or when $r=n+1\equiv 0\mod 2$ (then   $\pair{\ov u_0}\perp N$  cannot happen since $N$ is maximal isotropic in an even dimensional hermitian space $\ov W$ while $\pair{\ov u_0}$ is anisotropic). In the remaining cases, both can happen and we obtain exactly two orbits. Now the desired assertion follows. 
\end{proof}
\begin{remark}
A more direct argument for part (i) in the case $r$ odd is that  we have then  $\Lambda_0\supset \varpi(\Lambda_0^\flat \obot \langle u_0\rangle)^\vee\ni u_0$ and hence $\pair{u_0}$ is an orthogonal direct summand of $\Lambda_0$. 
\end{remark}

\begin{corollary}
\label{cosets}
\begin{altenumerate}
\item 
 When $r$ is odd or $r=0$, then
 $$K_{n+1}^{[\varepsilon]}K_{n+1}^{[r],+}= K_nK_{n+1}^{[r],+},\quad K_{n+1}^{[r],+}K_{n+1}^{[\varepsilon]}= K_{n+1}^{[r],+}K_n. 
 $$
 \item  When $r$ is even and $r=n+1$, then 
 $$K_{n+1}^{[0]}K_{n+1}^{[r],-}= K_nK_{n+1}^{[r],-},\quad K_{n+1}^{[r],-}  K_{n+1}^{[0]}= K_{n+1}^{[r],-}K_n. 
 $$
 \item  Let $r$ be  even and $1\le r\le n$. In this case, both $\Lambda^{+}$ and $\Lambda^{-}$ exist. There is an element $h\in K_{n+1}^{[0]}$ such that
\begin{equation}\label{eq conjugation}
K_{n+1}^{[r],-}=hK_{n+1}^{[r],+} h^{-1}.
\end{equation} 
 There are disjoint sum decompositions, 
 \begin{equation*}
 \begin{aligned}
 K_{n+1}^{[0]}K_{n+1}^{[r],+}&=K_nK_{n+1}^{[r],+} \sqcup K_n h K_{n+1}^{[r],+} \\
  K_{n+1}^{[0]}K_{n+1}^{[r],-}&=K_nK_{n+1}^{[r],-} \sqcup K_n  h^{-1} K_{n+1}^{[r],-} .
 \end{aligned}
 \end{equation*}
 
 {\rm There are also similar decompositions for $K_{n+1}^{[r],+}K_{n+1}^{[0]}$ and $K_{n+1}^{[r],-}K_{n+1}^{[0]}$, etc., by taking the inverses of the two sides of the equations.}
 \end{altenumerate}
\end{corollary}
\begin{proof} We only prove part (iii), as the others can be proved similarly.
Following the notation in the proof of Lemma \ref{lem 2 orb}, we have bijections
$$
K_{n+1}^{[\varepsilon]}K_{n+1}^{[r],+}/K_{n+1}^{[r],+}\simeq K_{n+1}^{[\varepsilon]}/(K_{n+1}^{[\varepsilon]}\cap K_{n+1}^{[r], +}) \simeq \U(\ov W)/\ov P
$$
where $\ov P$ is the parabolic stabilizing the isotropic subspace corresponding to $\Lambda^+$.
The action of $K_n$ on $K_{n+1}^{[\varepsilon]}K_{n+1}^{[r],+}/K_{n+1}^{[r], +}$ corresponds to the action of $\U(\ov W^\flat)$ on $ \U(\ov W)/\ov P$, where, we recall, $\ov W^\flat$ is the orthogonal complement of the anisotropic line generated by  $\ov u_0\in \ov W$.
By Lemma \ref{lem 2 orb}, there are exactly two $\U(\ov W^\flat)$ orbits, corresponding to the two  isotropic subspaces corresponding to $\Lambda^+$ and $\Lambda^-$ respectively. The base point in $ \U(\ov W)/\ov P$ corresponds to  $\Lambda^+$. Pick any representative of the other orbit, say $\ov h\in \U(\ov W)$, and lift it to $h\in K_{n+1}^{[\varepsilon]}$. Then $h$ satisfies \eqref{eq conjugation} and the decomposition $K_{n+1}^{[\varepsilon]}K_{n+1}^{[r],+}= K_nK_{n+1}^{[r], +} \sqcup K_n h K_{n+1}^{[r],+}$ holds.  This proves the existence of $h$ and the first equation in (iii). The second equation is proved in a similar way.
\end{proof}

\begin{remark}  Note that, for any two $h,h'$ satisfying  \eqref{eq conjugation}, we have  $h'h^{-1}\in K_{n+1}^{[r],+}$ so that the truth of the assertion in (iii) is independent of the choice of $h$ (a maximal parahoric is its own normalizer).  The existence of $h$ implies 
$$
K_{n+1}^{[0]}K_{n+1}^{[r],+}K_{n+1}^{[0]}=K_{n+1}^{[0]}K_{n+1}^{[r],-}K_{n+1}^{[0]}.$$ 
\end{remark}

\subsection{Orbital integrals}
Let us consider the analogous diagram to \eqref{eq Mnr}
\begin{equation}
  \begin{aligned}
  \xymatrix{\wt\BM_{n}^{[r]} \ar@{^(->}[r] \ar[d] \ar@{}[rd]|{\square}  &  \BN_{n+1}^{[r,\varepsilon]} \ar[d] \\ 
  \BN_n^{[0]} \ar@{^(->}[r] &   \BN_{n+1}^{[\varepsilon]}.}
  \end{aligned}
\end{equation}
Explicitly the set $\wt\BM_{n}^{[r]}$ consists of $(\Lambda^\flat, \Lambda)\in \Ver^0(W^\flat_0)\times  \Ver^r(W_{\varepsilon})$ such that $\Lambda^\flat \obot \langle u_0\rangle\supseteq \Lambda
$ holds.  According to whether $u_0\in \Lambda$ or not, the set $\wt\BM_{n}^{[r]}$ is partitioned into a disjoint union of  two subsets $\wt\BM_{n}^{[r],+}$ and $\wt\BM_{n}^{[r],-}$. Note that $\wt\BM_{n}^{[r],-}$ is empty if $r$ is odd or $r=0$, and that $\wt\BM_{n}^{[r],+}$ is empty if $r$ is even and $r=n+1$, cf. Lemma \ref{lem 2 orb}. Then   $\wt\BM_{n}^{[r],+}$ (if non-empty) is naturally bijective to $\BN^{[r,0]}$. 
We have 
\begin{equation}
 \begin{aligned}
 \xymatrix{ \wt\BM_{n}^{[r]} =\wt\BM_{n}^{[r],+}\sqcup \wt\BM_{n}^{[r],-}\ar[r] &\BN_n^{[0]} \times \BN^{[r]}_{n+1}.}
 \end{aligned}
 \end{equation}
Define the analogous function of $\varphi_r$ in \eqref{def phir9},
\begin{equation}\label{def phir}
\varphi_r^{[\varepsilon]}:=\frac{1}{\vol(K_{n+1}^{[r]})}{\bf 1}_{K^{[\varepsilon]}_{n+1}K_{n+1}^{[r]}}\ast {\bf 1}_{K_{n+1}^{[r]} K^{[\varepsilon]}_{n+1}}\in \CH_{\U(W_\varepsilon)}.
\end{equation}
Here we normalize the Haar measure so that $\vol(K^{[\varepsilon]}_{n+1})=1$.

We have an analog of Lemma \ref{lem orb=lat}, interpreting the orbital integral as a suitable lattice counting.
\begin{lemma}\label{lem orb lat}
For any regular semi-simple $g\in (\U(W^\flat)\times\U(W_{\varepsilon}))(F_0)$, we have 
\begin{equation*}
\begin{aligned}
\Orb(g, {\bf 1}_{K_{n}}\otimes  \varphi^{[\varepsilon]}_r )&=\#( \wt\BM_{n}^{[r]}\cap g \wt\BM_{n}^{[r]})\\
\Orb(g, {\bf 1}_{K_n\times K^{[r], +}_{n+1}} )&=\vol(K_n\cap K^{[r], +}_{n+1})^2\#(\wt\BM_{n}^{[r],+}\cap g\wt\BM_{n}^{[r],+}),\\
\Orb(g, {\bf 1}_{K_n\times K^{[r], -}_{n+1}} )&=\vol(K_n\cap K^{[r], -}_{n+1})^2\#( \wt\BM_{n}^{[r],-}\cap g \wt\BM_{n}^{[r],-}),\\
\Orb(g, {\bf 1}_{K_n\times h K^{+}_{n+1}} )&=\vol(K_n\cap K^{[r], +}_{n+1})\vol(K_n\cap K^{[r], -}_{n+1})\#( \wt\BM_{n}^{[r],+}\cap g \wt\BM_{n}^{[r],-}).
\end{aligned}
\end{equation*}
Here the intersection is taken inside $\BN_n^{[0]} \times \BN^{[r]}_{n+1}$ and  in the last three identities $r$ is even with $2\leq r\leq n$.
\end{lemma}
\begin{proof}
The proof is similar to that  of Lemma \ref{lem orb=lat},  and we omit the details. For the first equality, we will see an analog  in the context of RZ spaces in the proof of Lemma \ref{lem Int (0,r) even} below. 
\end{proof}

\begin{remark}We do not know the explicit transfer of the  functions appearing on the LHS above and therefore the analogue of Proposition \ref{FL r}  is still missing.
 \end{remark}
 
\subsection{The AT  conjecture} We assume that $r$ is even, and hence $\varepsilon=0$.
 We now consider the intersection number defined by \eqref{eq Int (0,r)}
 \begin{equation}
\left\langle \Mnr, g\Mnr\right\rangle_{\CN_n^{[0]}\times\CN_{n+1}^{[r]}}:=\chi(\CN_n^{[0]}\times\CN_{n+1}^{[r]}, \Mnr\cap^\BL g\Mnr ) , \quad g\in G_{W_1}(F_0)_\rs .
\end{equation}

\begin{conjecture}\label{conjreven 0 r}
Recall that $r$ is even. Let $\varphi'$ be any element in $\CH_{K^{\prime}_n}\otimes_{\BQ}\CH_{K^{\prime}_{n+1}}$ such that
$$\Bc(\varphi')= {\bf 1}_{ K_{n}}\otimes  \varphi_r \in \CH_{K_n}\otimes_\BQ\CH_{K_{n+1}}.$$ 
If $\gamma\in G'(F_0)_\rs$ is matched with  $g\in G_{W_1}(F_0)_\rs$, then \begin{equation*}
  \left\langle \Mnr, g\Mnr\right\rangle_{\CN_n^{[0]}\times\CN_{n+1}^{[r]}} \cdot\log q=-\frac{1}{2}\del\big(\gamma,  \varphi' \big).
\end{equation*}
\end{conjecture}
\begin{remark}
In contrast to   Conjecture \ref{conjreven}, we could formulate an inhomogeneous version of Conjecture \ref{conjreven 0 r}.
\end{remark}

Similar to the AT conjecture of type $(r, 0)$ with even $r$, Conjecture \ref{conjreven 0 r} is a consequence of the AFL conjecture \ref{AFL Hk}. We have the following analog of Lemma \ref{lem Int r even}.
\begin{lemma}\label{lem Int (0,r) even}
For every $g\in G_{W_1}(F_0)_\rs$, we have
$$
\left\langle \Mnr, g\Mnr\right\rangle_{\CN_n^{[0]}\times\CN_{n+1}^{[r]}}=\left\langle \BT_{ {\bf 1}_{K_{n}}\otimes  \varphi_r }( \Delta_{\CN_n^{[0]}}),g\Delta_{\CN_n^{[0]}}  \right\rangle_{\CN^{[0]}_n\times \CN^{[0]}_{n+1}}.
$$
\end{lemma}
\begin{proof}The proof is similar to that of Lemma \ref{lem Int r even} and we only indicate the differences.
Consider the correspondence  analogous to \eqref{cor r 0},
\begin{equation}
\label{cor 0 r}
\begin{aligned}
\xymatrix{& \Delta_{ \CN^{[0]}_{n}} \times \CN^{[0,r]}_{n+1}\ar[rd]^{\wt\pi_2}  \ar[ld]_{\wt\pi_1} & \\ \CN^{[0]}_n\times \CN^{[0]}_{n+1}  & &\CN^{[0]}_{n}\times  \CN^{[r]}_{n+1}.} 
\end{aligned}
\end{equation}
Then, unpacking the definition of $\Mnr$, we have 
\begin{equation}\label{eq cor r 0}
(\wt\pi_2)_\ast\wt\pi_1^\ast(\Delta_{\CN_n^{[0]}})=\Mnr ,
\end{equation}
as classes in $\Gr^{n}K_0^{\Mnr}(\CN^{[0]}_n\times \CN^{[r]}_{n+1})$.
The rest of the proof is similar to that of Lemma  \ref{lem Int r even}, applying the projection formula.

\end{proof}

  \begin{corollary}\label{CorAFLwh 0 r}
  Recall that $r$ is even. Conjecture 
\ref{AFL Hk} (for $\varphi={\bf 1}_{K_{n}}\otimes \varphi_r$) implies Conjecture  \ref{conjreven 0 r}.
  \end{corollary}

  \begin{remark}There is no analogous relation to Conjecture 
\ref{AFL Hk} for the variants for $\wt\CM_{n}^{[r],+},\wt\CM_{n}^{[r],-}$ and the mixed case, cf. \S\ref{ss type 0r}, Conjecture \ref{conj 0 r mixed}.
\end{remark}
  
\section{AT conjectures of type $(n, 0)$ and $(0, n+1)$ with $n$  odd.}\label{s:ATCodd}
In this section we  consider the AT conjecture of type $(r, 0)$ when $r=n$ is odd. For this, we use an idea  similar to the one in the last section:  reduce the problem to the FL and AFL for certain  (non-unit)  elements in the spherical Hecke algebra. We then use the exceptional isomorphism  \eqref{relMN} to deduce the AT conjecture of type $(0, n+1)$, where $n$ is odd.   

\subsection{An explicit transfer: an application of the FL for the whole Hecke algebra} 

We first return to the situation in \S\ref{ss: cpt open}.
We  pick a basis $e_1,e_2,\cdots, e_n$ of the $n$-dimensional Hermitian space $W_0^\flat$ such that $(e_i,e_i)=\varpi$ for all $i$, and add the special vector $u_0$ with $(u_0,u_0)=\varpi$. Then, because $n=r$ is odd,  $W_0=W_0^\flat\oplus\pair{u_0}_F$ is a split Hermitian space. Set $ \Lambda^\flat=\pair{e_1,\ldots, e_n}\in  \Ver^n(W_0^\flat)$ and $\Lambda= \Lambda^\flat\oplus \pair{u_0}\in  \Ver^{n+1}(W_0)$. Then $\Lambda$ is selfdual up to a scalar. Let $\Lambda_0\in \Ver^0(W_0)$ be such that $\Lambda_0\supset \Lambda$. Let $K_{n+1}=K_{n+1}^{[0]}=\U(\Lambda_0)$ and $K_{n+1}^{[n+1]}=\U(\Lambda)$ be the stabilizer of $\Lambda_0$ and $\Lambda$ respectively. Both are hyperspecial compact open subgroups of $\U(W_0)$. Let $K_n^{[r]}=\U(\Lambda^\flat)$.

Let $W_0'$ be the same space as $W_0$ but with a rescaling of the hermitian form by a factor $\varpi^{-1}$. Then there exists an isometry between $W_0$ and $W_0'$ that induces a bijection between the vertex lattices of type $(n+1)$ in $W_0$ and  the vertex lattices of type $0$  in $W_0'$. This isometry also induces an isomorphism between the unitary groups and hence also  isomorphisms of the Hecke algebras  $\CH_{K_n^{[n]}}$ and  $\CH_{K_{n+1}^{[n+1]}}$ with the spherical Hecke algebras, so that the base change homomorphisms $\Bc: \CH_{K^{\prime}_n}\to \CH_{K_n^{[n]}}$ and $\Bc:\CH_{K^\prime_{n+1}}\to \CH_{K_{n+1}^{[n+1]}}$ make sense.

We continue with  the  Haar measures fixed in \eqref{normmea}, so that 
$$
\vol(K_n^{[n]})=1,\quad \vol(K_n')=1,
$$
(note that  $\vol(K_n^{[n]})=\vol(K_n^{[0]})=\vol(\wt K_n^{[0]})$).
 We also define the atomic function analogous to \eqref{def phir9} but for the current hyperspecial compact open $K_{n+1}^{[n+1]}$,
\begin{equation}\label{def phi 0 n+1}
\varphi_{0}^{[n+1]}:= {\bf 1}_{K_{n+1}^{[n+1]}K_{n+1}}\ast {\bf 1}_{K_{n+1} K_{n+1}^{[n+1]}}\in \CH_{K_{n+1}^{[n+1]}}.
\end{equation}
Note that $\vol (K_{n+1}^{[n+1]})=\vol(K_{n+1})$ which is taken to be one to normalize our choice of measure. Here we use the notation $\varphi_{0}^{[n+1]}$ to be consistent with $\varphi_r^{[\varepsilon]}$ in \eqref{def phir}. In particular, the function $\varphi_r$ in  \eqref{def phir9} could be renamed as $\varphi_{r}^{[0]}$.

We have the following result similar to Proposition  \ref{FL r}.

\begin{proposition}\label{FL r odd}
Let $r=n$ be odd. Let $\varphi'$ be any element in $\CH_{K^{\prime}_n}\otimes_{\BQ}\CH_{K^{\prime}_{n+1}}$ such that 
$$
\Bc(\varphi')=  {\bf 1}_{K_n^{[n]}}\otimes\varphi_{0}^{[n+1]}\in \CH_{K_n^{[n]}}\otimes_\BQ\CH_{K_{n+1}^{[n+1]}}.
$$
 Then $ \varphi'$ is a transfer of  $(c_n ^{2}\,{\bf 1}_{K_n^{[n]}\times K_{n+1}},0)\in C_c^\infty(G_{W_0})\times C_c^\infty(G_{W_1})$.  
\end{proposition}

\begin{proof}
This follows from Lemma \ref{lem: Orb n=r odd} below and the Jacquet--Rallis fundamental lemma for the full Hecke algebra due to Leslie \cite{Les}, cf. \cite[Thm. 3.7.1]{LRZ}. Note that, even though the theorem of  Leslie \cite{Les} was formulated for the hyperspecial compact open associated to the self-dual lattices, it can be easily translated into the version for the hyperspecial compact open in the current set-up.
\end{proof}

\begin{lemma}\label{lem: Orb n=r odd}Assume that $n=r$ is odd. Then we have for $g\in G_{W_0}(F_0)_\rs$,
 \begin{equation*}
 \begin{aligned}
   \Orb(g, {\bf 1}_{K_n^{[n]}\times K_{n+1}} )&=  \vol(K_n^{[n]}\cap K_{n+1}) ^2\,  \Orb(g,{\bf 1}_{ K^{[n]}_{n}} \otimes  \varphi_{0}^{[n+1]}  ),\\
   \Orb(g, {\bf 1}_{K_n^{[0]}\times K_{n+1}^{[n+1]}} )&=  \vol(K_n^{[0]}\cap K_{n+1}^{[n+1]}) ^2\,  \Orb(g,{\bf 1}_{ K^{[n]}_{n}} \otimes  \varphi_{0}^{[n+1]}  ).
   \end{aligned}
   \end{equation*}
\end{lemma}

\begin{proof}
Let us prove the first identity and  the second identity is proved the same way. 
 
The proof is similar to that of Lemma \ref{lem:Orb r even}. 
We only indicate the changes needed. As in  \eqref{eq Orb}, using    $\vol(K_n^{[n]})=1$, we have 
  \begin{equation}\label{eq Orb1}
   \Orb(g, f)=    \Orb(g,{\bf 1}_{\Delta(K_n^{[n]})}\ast f \ast {\bf 1}_{\Delta(K_n^{[n]})}) .
   \end{equation}
Similar to \eqref{eq:conv}, we obtain
  $$
  {\bf 1}_{\Delta(K_n^{[n]})}\ast {\bf 1}_{K_n^{[n]}\times K_{n+1}}  \ast {\bf 1}_{\Delta(K_n^{[n]})}={\bf 1}_{ K^{[n]}_{n}} \otimes  ({\bf 1}_{K_n^{[n]}}\ast{\bf 1}_{K_{n+1}}\ast  {\bf 1}_{K_n^{[n]}}).
  $$ 
  Now it suffices to note that, by Lemma \ref{lem KfK} below,
 $$
   {\bf 1}_{K_n^{[n]}}\ast
   {\bf 1}_{K_{n+1}}=\vol(K_n^{[n]}\cap K_{n+1})  {\bf 1}_{K_{n+1}^{[n+1]}K_{n+1}}, \quad    
   {\bf 1}_{K_{n+1}}\ast {\bf 1}_{K_n^{[n]}}=\vol(K_n^{[n]}\cap K_{n+1})  {\bf 1}_{K_{n+1}K_{n+1}^{[n+1]}}.
   $$

\end{proof}

The following lemma, used in the previous proof, essentially follows from Corollary  \ref{cosets}. 
\begin{lemma}
\label{lem KfK}
We have  an equality of subsets of $\U(W_0)(F_0)$,
$$
K_n^{[n]}  K_{n+1}^{[0]}=K_{n+1}^{[n+1]}K_{n+1}^{[0]}, \quad   K_{n+1}^{[0]}K_n^{[n]} =K_{n+1}^{[0]}K_{n+1}^{[n+1]} .$$
\end{lemma}

\begin{proof}
By rescaling the hermitian form on both hermitian spaces $W_0, W_0^\flat$ by a factor $\varpi^{-1}$, the two equations become those in Case (ii) of Corollary  \ref{cosets}, noting that there the group $K_{n+1}^{[r],-}$ is equal to $K_{n+1}^{[n+1]}$ when $r=n+1$.

\begin{remark}
Similar to Lemma \ref{lem:Orb r even}, we have an alternative proof of Lemma \ref{lem: Orb n=r odd} using the lattice counting interpretation of orbital integrals in Lemma \ref{lem orb lat}. Now, we have $ \wt\BM_{n}^{[r]}= \wt\BM_{n}^{[r],-}\simeq \BN^{[n,0]}$. Hence the right hand sides of the first and the third equations  in Lemma \ref{lem orb lat} are equal (up to the desired constants).  It follows that their left hand sides are also equal, as desired. 
\end{remark}

\end{proof}

\subsection{AT conjecture of type $(r, 0)$ with  $r=n$ odd}
The following conjecture is analogous to Conjecture \ref{conjreven}.

\begin{conjecture}\label{conjrodd}
Let $n=r$ be odd. Let $\varphi'$ be any element in $\CH_{K^{\prime}_n}\otimes_{\BQ}\CH_{K^{\prime}_{n+1}}$ such that 
$$\Bc(\varphi')=  {\bf 1}_{K_n^{[n]}}\otimes\varphi_0^{[n+1]}\in \CH_{K_n^{[n]}}\otimes_\BQ\CH_{K_{n+1}^{[n+1]}}.$$ 
\label{item:conjr1} If $\gamma\in G'(F_0)_\rs$ is matched with  $g\in G_{W_1}(F_0)_\rs$, then
 \begin{equation*}
    \left\langle \tnr, g\tnr\right\rangle_{\nr\times\n} \cdot\log q=-\frac{1}{2} \del\big(\gamma,  \varphi' \big).
\end{equation*}
\end{conjecture}

Similarly to the even $r$ case,
Conjecture \ref{conjrodd} is a consequence of Conjecture \ref{AFL Hk} (AFL for the full Hecke algebra). This is based on  the following analog of Lemma \ref{lem Int r even}.

  Let $ \CT_{n+1}^{\leq (n+1)}$ be the Hecke correspondence associated to the spherical Hecke function $ {\bf 1}_{K_n^{[n]}}\otimes\varphi_0^{[r+1]}\in \CH_{K_n^{[n]}}\otimes\CH_{K_{n+1}^{[n+1]}}$, and let $\BT_{ {\bf 1}_{K_n^{[n]}}\otimes\varphi_0^{[n+1]}}$ be the corresponding Hecke operator in K-theory, cf. \cite[\S 9]{LRZ}.
\begin{lemma}\label{lem Int r odd}Let $n=r$ be odd. 
For every $g\in G_{W_1}(F_0)_\rs$,
$$
\left\langle \tnr, g\tnr\right\rangle_{\nr\times\n}=\left\langle \BT_{ {\bf 1}_{K_n^{[n]}}\otimes\varphi_0^{[n+1]}}( \Delta_{\CN_n^{[n]}}),g\Delta_{\CN_n^{[n]}}  \right\rangle_{\CN^{[n]}_n\times \CN^{[n+1]}_{n+1}}.
$$
\end{lemma}
The proof is analogous to that of Lemma \ref{lem Int r even}, and we leave the details to the reader.

  \begin{corollary}\label{cor odd}
  Conjecture 
\ref{AFL Hk} (for $\varphi= {\bf 1}_{K_n}\otimes\varphi_{n+1}$) implies Conjecture  \ref{conjrodd}.
  \end{corollary}
  \begin{proof}
  Let us rescale the hermitian spaces $W_0$ and $W_0^\flat$ by the factor $\varpi^{-1}$ to have the natural isomorphism $\CH_{K_n^{[n]}}\otimes\CH_{K_{n+1}^{[n+1]}}\simeq \CH_{K_n^{[0]}}\otimes\CH_{K_{n+1}^{[0]}}$. Under this isomorphism, the function $\varphi_{0}^{[n+1]}\in \CH_{K_{n+1}^{[n+1]}}$ (resp. ${\bf 1}_{K_n^{[n]}}\in \CH_{K_{n}^{[n]}}$) corresponds to the function $\varphi_{n+1}\in\CH_{K_{n+1}}$ (resp. ${\bf 1}_{K_n}\in \CH_{K_{n}}$). Then the isomorphism $\CN^{[n]}_n\times \CN^{[n+1]}_{n+1}\simeq \CN^{[0]}_n\times \CN^{[0]}_{n+1}$ also induces an isomorphism of correspondences so that we have an equality of intersection numbers
 $$ \left\langle \BT_{ {\bf 1}_{K_n^{[n]}}\otimes\varphi_0^{[n+1]}}( \Delta_{\CN_n^{[n]}}),g\Delta_{\CN_n^{[n]}}  \right\rangle_{\CN^{[n]}_n\times \CN^{[n+1]}_{n+1}}=
 \left\langle \BT_{ {\bf 1}_{K_n^{[0]}}\otimes\varphi_{n+1}}( \Delta_{\CN_n^{[0]}}),g\Delta_{\CN_n^{[0]}}  \right\rangle_{\CN^{[0]}_n\times \CN^{[0]}_{n+1}}.
  $$
  The corollary nows follows from Lemma \ref{lem Int r odd}.
  
  Alternatively, this corollary also follows from Lemma \ref{lem MN} below together with Corollary \ref{CorAFLwh 0 r}.
  \end{proof}
Since  Conjecture 
\ref{AFL Hk} is known  in the case $n=1$  \cite[Thm. 7.5.1]{LRZ}, we deduce the following statement.   
\begin{theorem} \label{thm n=r=1}
The AT conjecture of type $(1, 0)$ holds for $n=1$, i.e., Conjecture \ref{conjrodd} holds when $n=r=1$. \qed
\end{theorem}

\begin{remark}
We point out that in the special case $n=r=1$ the natural looking function $\mathbf{1}_{K^{\prime [1]}_{1}\times K_{2}'}$ is not a transfer of  ${\bf 1}_{K_1^{[1]}}\otimes\varphi_0^{[2]}$ (up to a constant multiple). Using the rescaling isomorphism in  the proof of Corollary \eqref{cor odd}, we need to consider  ${\bf 1}_{K_1}\otimes \varphi_2$. But we have
$$
\varphi_{2}=(q+1){\bf 1}_{K_2}+ {\bf 1}_{K_2\varpi^{(1,-1)}K_2},
$$
cf. \cite[(7.1.5)]{LRZ}. 
The function  $\mathbf{1}_{K^{\prime [1]}_{1}\times K_{2}'}$
can be shown to be equivalent to (up to a constant multiple) $${\bf 1}_{\GL_{1}(O_F)}\otimes ({\bf 1}_{K_{2}' \varpi^{(1,0)}K_{2}'} -{\bf 1}_{K_{2}'})\in \CH_{K^{\prime}_1}\otimes_{\BQ}\CH_{K^{\prime}_{2}} . 
$$
In particular, using the explicit calculation in \cite[\S7]{LRZ}, we see that it does not transfer to ${\bf 1}_{K_1}\otimes \varphi_2$ on the unitary side. 
\end{remark}

\subsection{AT conjecture of type $(0, r+1)$ with  $r=n$ odd}
The following lemma shows the equivalence of the AT conjectures  of type $(n, 0)$ and $(0, n+1)$ with  $n$ odd.

\begin{lemma}\label{lem MN}
Let $n$ be odd. For every $g\in G_{W_1}(F_0)_\rs$,
we have
$$
\left\langle \wt\CN_n^{[n]}, g\wt\CN_n^{[n]}\right\rangle_{\CN_n^{[n]}\times\CN_{n+1}^{[0]}}=\left\langle \wt\CM_n^{[n+1]}, g\wt\CM_n^{[n+1]}\right\rangle_{\CN_n^{[0]}\times\CN_{n+1}^{[n+1]}}.
$$
\end{lemma} 
\begin{proof}
The rescaling isomorphism $\CN_{n}^{[n]}\times \CN_{n+1}^{[0]}\simeq \CN_{n}^{[0]}\times \CN_{n+1}^{[n+1]}$ is $G_{W_1}(F_0)$-equivariant, and induces an isomorphism $\wt\CN_n^{[n]}\simeq \wt\CM_n^{[n+1]}$, cf. \eqref{relMN}.
\end{proof}
We therefore may formulate the following conjecture. 
\begin{conjecture}\label{conjrodd+}
Let $n$ be odd. Let $\varphi'$ be any element in $\CH_{K^{\prime}_n}\otimes_{\BQ}\CH_{K^{\prime}_{n+1}}$ such that  
$$\Bc(\varphi')= {\bf 1}_{K_{n}^{[n]}}\otimes\varphi_0^{[n+1]}\in \CH_{K_n^{[n]}}\otimes_\BQ\CH_{K_{n+1}^{[n+1]}}.$$
 If $\gamma\in G'(F_0)_\rs$ is matched with  $g\in G_{W_1}(F_0)_\rs$, then
 \begin{equation*}
    \left\langle \wt\CM_n^{[n+1]}, g\wt\CM_n^{[n+1]}\right\rangle_{\CN_n\times\CN_{n+1}^{[n+1]}} \cdot\log q=-\frac{1}{2} \del\big(\gamma,  \varphi' \big).
\end{equation*}
\end{conjecture}

 We again have the following corollary, whose proof is similar to that of  Corollary \ref{cor odd} and will be omitted.

  \begin{corollary}\label{cor odd+}
  Let $n$ be odd. Conjecture 
\ref{AFL Hk} (for $\varphi=  {\bf 1}_{K_{n}}\otimes \varphi_{n+1}$) implies Conjecture  \ref{conjrodd+}.
  \end{corollary}

Again, as in Theorem \ref{thm n=r=1}, we deduce the AT conjecture of type $(0,2)$ for $n=1$.

\section{AT conjecture of type $(0,1)$}
\label{s: case (01)}
In this section we  consider the AT conjecture of type $( 0,r)$ when $r=1$. In fact, this case was considered in \cite[\S10, \S14]{RSZ2}, where we reduced the AT conjecture to the (now known) AFL conjecture (for the unit element), at least in the artinian case. It was then revisited independently by Z. Zhang \cite{ZZha}, who gave a direct proof (as a special case of a more general result). 

We are now in the set-up of \S\ref{ss: cpt open 0 r}, specialized to the case $r=1$. We have a special vector $u_0$ of norm $\varpi$ in $W_1$, and $W^\flat=\pair{u_0}^\perp$ is split. (Note that in \cite{RSZ2}, the special vector is denoted by $u_1$.) We have $ \Lambda_0^\flat\in \Ver^0(W^\flat)$ and $ \Lambda=\Lambda_0^\flat \obot \langle u_0\rangle\in \Ver^{1}(W_1)$, with corresponding parahoric subgroups $K_n$ and $K_{n+1}^{[1]}$. Let $K_n'\subset
\GL(W^\flat)$ be the stabilizer of $ \Lambda_0^\flat$ and let $K'_{n+1}(\varpi)\subset
\GL(W_1)$ be the joint stabilizer of $ \Lambda$ and $\Lambda^\vee$ (the latter group was denoted by $ K_0(\varpi)$ in \cite[\S10]{RSZ2}). Let $K_S(\varpi)\subset S_{n+1}(F_0)$ denote the intersection $K'_{n+1}(\varpi)\cap S_{n+1}(O_{F_0})$.

The following explicit transfer theorem was \cite[Conj. 10.3]{RSZ2}, and is now a theorem.
\begin{theorem}
\label{thm type 01}
\begin{altenumerate}
\item\textup{(Homogeneous version)} The function $(-1)^{n-1}\frac{q^n-1}{q-1}\mathbf{1}_{K_n'\times K'_{n+1}} \in C_c^\infty(G')$ transfers to the pair of functions $(0,\mathbf{1}_{K_n\times K_{n+1}^{[1]}}) \in C_c^\infty(G_{W_0}) \times C_c^\infty(G_{W_1})$. 
\item
\textup{(Inhomogeneous version)} The function $(-1)^{n-1}\mathbf{1}_{K_S(\varpi)} \in C_c^\infty(S)$ transfers to the pair of functions $(0,\mathbf{1}_{K_{n+1}^{[1]}}) \in C_c^\infty(\U(W_0))\times C_c^\infty(\U(W_1))$.
\end{altenumerate}

\end{theorem}
\begin{proof}Recall that $q$ is odd.
In \cite[Thm. 14.1]{RSZ2}, the above statement was reduced to the FL for the Lie algebra, at least when $q\geq n$. The FL for Lie algebra  is now known for any  (odd) residue characteristic, see \S\ref{ss:FL}. Therefore the theorem follows, at least when $q\geq n$. 

On the other hand, Z. Zhang \cite[Thm. 1.1]{ZZha}  also provided a direct proof of the explicit transfer for any maximal parahoric subgroup, which implies the claim  for all odd $q$.  Note that there is a sign difference between \cite{RSZ2} and \cite{ZZha}.
\end{proof}

The following AT theorem was \cite[Conj. 10.4]{RSZ2}, and is now also a theorem. For the sake of brevity we only consider the inhomogeneous version.
\begin{theorem}
\label{thm AT type 01}
Suppose that $\gamma\in S_{n+1}({F_0})_\rs$ matches an element $g\in  \U(W_0)(F_0)_\rs$. Then 
\[
  \bigl\la \CN_n, (1 \times g)\CN_n\bigr\ra_{\CN_n\times\CN^{[1]}_{n+1}}\cdot\log q=-\del\bigl(\gamma, (-1)^{n-1}\mathbf{1}_{K_S(\varpi)} \bigr) . 
\]

\end{theorem}
\begin{proof}
In \cite[Thm. 14.9, and Thm. 14.10]{RSZ2}, the above statement was reduced to the AFL conjecture (for the unit element), at least when $q\geq n$ and when the intersection is artinian. The AFL is now known by \cite{Zha21} (for $q\geq n$) and \cite{ZZha} (for all odd $q$). Therefore the theorem follows, at least when $q\geq n$ and the intersection is artinian.

 On the other hand, Z. Zhang \cite[Thm. 1.4]{ZZha}  also provided a direct proof of the arithmetic transfer conjecture for any maximal parahoric subgroup (comp.  Introduction), which implies our assertion for all odd $q$.  
 \end{proof}
\begin{remark}
The reduction argument  in \cite[Thm. 14.9, and Thm. 14.10]{RSZ2} was done only when the intersection is artinian. However, the same idea should apply in general, so that Theorem \ref{thm AT type 01} can be reduced to the AFL (for the unit element), at least when $q\geq n$. We leave this enhancement of \cite{RSZ2} to interested readers.
\end{remark}

\section{AT conjectures: the remaining cases}\label{s:ATCgen}
In the last sections
we have stated several cases of AT conjectures in which we have an explicit test function (at least with the help of the base change homomorphism). In the remaining cases, we do not have an explicit test function. Instead, similar to \cite[Conj. 5.3]{RSZ1}, we can only formulate an AT conjecture where we postulate the existence of a test function with an explicit transfer. Assuming the density conjecture \cite[Conj. 5.16]{RSZ1}, this conjecture implies that then any test function with the given transfer yields an AT identity, but  with a correction function, analogous to  part (ii) in Conjecture \ref{conjreven}.

\subsection{Type $(r, 0)$  and $r$ odd}

Let $r$ be an odd integer such that $0\leq r\leq n$, not necessarily equal to $n$ (the latter case is considered in \S \ref{s:ATCodd}).

Recall from \eqref{eq:Int1} that we have defined the arithmetic intersection number
$  \left\langle \tNr, g\tNr\right\rangle_{\Nr\times\N}$.

\begin{conjecture}\label{conjrodd all}
Let $r$ be odd such that $0\leq r\leq n$. 
\begin{altenumerate}
\item\label{item:conjrodd1} There exists $\varphi'\in C_c^\infty(G')$ with transfer $(c_r^2\cdot {\bf 1}_{K_n^{[r]}\times K_{n+1}},0)\in C_c^\infty(G_{W_0})\times C_c^\infty(G_{W_1})$ such that, if $\gamma\in G'(F_0)_\rs$ is matched with  $g\in G_{W_1}(F_0)_\rs$, then
 \begin{equation*}
    \left\langle \tnr, g\tnr\right\rangle_{\nr\times\n} \cdot\log q=-\frac{1}{2}\del\big(\gamma,  \varphi' \big).
\end{equation*}
\item\label{item:conjrodd2} For any $\varphi'\in C_c^\infty(G')$ transferring to  $(c_r^2\cdot {\bf 1}_{K_n^{[r]}\times K_{n+1}},0)\in C_c^\infty(G_{W_0})\times C_c^\infty(G_{W_1})$, there exists $\fp'_\corr\in C_c^\infty(G')$  such that, if $\gamma\in G'(F_0)_\rs$ is matched with  $g\in G_{W_1}(F_0)_\rs$, then
\begin{equation*}
    \left\langle \tnr, g\tnr\right\rangle_{\nr\times\n}  \cdot\log q= -\frac{1}{2}\del\big(\gamma,  \varphi' \big) - \Orb\big(\gamma,  \fp'_\corr \big).
\end{equation*}
\end{altenumerate}
\end{conjecture}

Note that by  the density conjecture \cite[Conj. 5.16]{RSZ1}, part (ii) follows from part (i). Something analogous holds for all further conjectures in this section; in the interest of brevity, we have omitted these variants of these conjectures in the statements below.

Let us comment on Conjecture \ref{conjrodd all}. The case $r=n=1$ has been treated in the last section (Theorem \ref{thm n=r=1}). In the next simplest case when $r=1, n=2$, we can show that the natural candidate  $\mathbf{1}_{{K}_2^{\prime [1]}\times K_{3}'}$ is a transfer of  $({\bf 1}_{K_2^{[1]}\times K_3},0)$ (up to a constant multiple). However, even in this case we did not prove the AT conjecture above. 
In fact, beyond the case ($r=1, n=2$), we know nothing about the AT conjecture at the moment. We will pursue this direction in a future paper.

\subsection{Type $(0,r)$ }
\label{ss type 0r}
The heuristics for the explicit transfer in this case comes from Lemma \ref{lem orb lat}. We keep the notation $K_{n+1}^{[r]}$ and $K_{n+1}^{[r], +}$ and $K_{n+1}^{[r], -}$
and $h K_{n+1}^{[r],+}$ of that lemma.
\begin{conjecture}\label{conj 0 r +}
Let $r$ be such that $0\leq r\leq n+1$, with parity $\varepsilon=\varepsilon(r)$. Recall from \S \ref{ss: cpt open 0 r} the hermitian space $W_\varepsilon$; we denote by $W_{\varepsilon+1}$ the hermitian space of the same dimension $n+1$ and with opposite invariant. As before, the perp-spaces $W_\varepsilon^\flat$ and $W_{\varepsilon+1}^\flat$ are formed using special vectors of length $\varpi^\varepsilon$. Also, recall the function $\varphi_r^{[\varepsilon]}\in \CH_{\U(W_\varepsilon)}$ from \eqref{def phir}.

 There exists $\varphi'\in C_c^\infty(G')$ with transfer $({\bf 1}_{K_n^{[0]}}\otimes \varphi_r^{[\varepsilon]},0)\in C_c^\infty(G_{W_\varepsilon})\times C_c^\infty(G_{W_{\varepsilon+1}})$ such that, if $\gamma\in G'(F_0)_\rs$ is matched with  $g\in G_{W_{\varepsilon+1}}(F_0)_\rs$, then
 \begin{equation*}
 \left\langle \wt\CM_n^{[r]}, g\wt\CM_n^{[r]}\right\rangle_{\CN_n^{[0]}\times\CN_{n+1}^{[r]}} 
 \cdot\log q=-\frac{1}{2}\del\big(\gamma,  \varphi' \big).
\end{equation*}
\end{conjecture}
Note that in the case $r$ even we have the more precise Conjecture \ref{conjreven 0 r}, in that we can give $\varphi'$ explicitly. However, when $r$ is even and $2\leq r\leq n$, there is the following refinement of Conjecture \ref{conjreven 0 r}, in which we cannot give the test function explicitly.

We recall from \S\ref{ss:Mnr} that, when $r$ is even and $2\leq r\leq n$, the space $\Mnr$ has two closed formal schemes, $\wt{\CM}_n^{[r], +}=\CN_n^{[0,r]}$ and   $\wt{\CM}_n^{[r], -}$ (see \eqref{eq MnrNr0} and the paragraph after it). 
As the notation suggests, the generic fiber of $\wt{\CM}_n^{[r], +}$  is the member of the RZ tower associated with the compact open subgroup $K_{n+1}^{[r],+}=\U( \Lambda^+ )$. To be parallel to the formulation of Conjecture \ref{conjrodd all},
we introduce the volume constants analogous to \eqref{defc},
\begin{equation}
c_{r}^\pm:=\vol(K_n\cap K^{[r], \pm}_{n+1})^{-1}=[K_n:K_n\cap K^{[r], \pm}_{n+1} ].
\end{equation}
We then have the following  AT conjectures for the plus space and the minus space and the mixed case respectively. We know nothing about them at this moment. 
\begin{conjecture}\label{conj 0 r mixed}
Let $r$ be even such that $2\leq r\leq n$.
\begin{altenumerate}
\item\label{item:conjr1even1} There exists $\varphi'\in C_c^\infty(G')$ with transfer $((c_{r}^+)^{2}\,{\bf 1}_{K_n^{[0]}\times K_{n+1}^{[r],+}},0)\in C_c^\infty(G_{W_0})\times C_c^\infty(G_{W_1})$ such that, if $\gamma\in G'(F_0)_\rs$ is matched with  $g\in G_{W_1}(F_0)_\rs$, then
 \begin{equation*}
 \left\langle \wt\CM_n^{[r],+}, g\wt\CM_n^{[r],+}\right\rangle_{\CN_n^{[0]}\times\CN_{n+1}^{[r]}} 
 \cdot\log q=-\frac{1}{2}\del\big(\gamma,  \varphi' \big).
\end{equation*}
\item\label{item:conjr1even2} There exists $\varphi'\in C_c^\infty(G')$ with transfer $((c_{r}^-)^{2}\, {\bf 1}_{K_n^{[0]}\times K_{n+1}^{[r],-}},0)\in C_c^\infty(G_{W_0})\times C_c^\infty(G_{W_1})$ such that, if $\gamma\in G'(F_0)_\rs$ is matched with  $g\in G_{W_1}(F_0)_\rs$, then
 \begin{equation*}
 \left\langle \wt\CM_n^{[r],-}, g\wt\CM_n^{[r],-}\right\rangle_{\CN_n^{[0]}\times\CN_{n+1}^{[r]}} 
 \cdot\log q=-\frac{1}{2}\del\big(\gamma,  \varphi' \big).
\end{equation*}
\item\label{item:conjr1even3} There exists $\varphi'\in C_c^\infty(G')$ with transfer $(c_{r}^+c_{r}^-\,{\bf 1}_{K_n^{[0]}\times hK_{n+1}^{[r],+}},0)\in C_c^\infty(G_{W_0})\times C_c^\infty(G_{W_1})$ such that, if $\gamma\in G'(F_0)_\rs$ is matched with  $g\in G_{W_1}(F_0)_\rs$, then
 \begin{equation*}
 \left\langle \wt\CM_n^{[r],+}, g\wt\CM_n^{[r],-}\right\rangle_{\CN_n^{[0]}\times\CN_{n+1}^{[r]}} 
 \cdot\log q=-\frac{1}{2}\del\big(\gamma,  \varphi' \big).
\end{equation*}
\end{altenumerate}
\end{conjecture}

\part{The geometry of $\tN$}

In this part we specialize the situation of \S\ref{sec:geometric-side} to $r=1$. Our aim is to prove Theorem \ref{Introstrr=1} from the Introduction. In \S \ref{ss: k pt} and \S\ref{sec:bruh-tits-strat}, we give a description of the $\kb$-points of $\Na$ and its Bruhat-Tits stratification (using Dieudonn\'e theory) and determine the singularities (using the theory of local models). 
In \S \ref{ss: k pt N(n+1)} and \S\ref{sec:bruh-tits-strataself}, we do the same for $\CN^{[1]}_{n+1}$. In  \S \ref{ss: Zx}, we describe the singularities of  the special cycle $\CZ(u)$. In  \S \ref{ss: tN}, we formulate the main result Theorem \ref{conj:KRSZ} on the structure of $\wit\CN_n^{[1]}$. The remaining sections \S \ref{s:defthy} and \S \ref{s:Pfs} are devoted to the proofs of Theorem \ref{conj:KRSZ} and of Conjecture \ref{conjreg}.

Write $\mathbb{V}=\mathbb{V}_{n+1}$ and $\mathbb{W}=\mathbb{W}_n^{[1]}$ for short.  

\section{The space $\tN$}\label{s:spaceN}

\subsection{$\kb$-points of $\Na$} 
\label{ss: k pt}

Let $(Y, \iota, \lambda,\rho)\in \Na(\kb)$. Let $\mathbb{D}(Y)$ be the (relative) Dieudonn\'e crystal of $Y$. The (relative) Dieudonn\'e module $\mathbb{D}(Y)(\OFb)$ is a free $\OFb$-module of rank $2n$, equipped with the action of the $\sigma$-linear  Frobenius $\F$ and the $\sigma^{-1}$-linear Verschiebung $\V$.  The almost principal polarization $\lambda$ induces a non-perfect alternating $\OFb$-bilinear form on the Dieudonn\'e module $$\langle\ ,\ \rangle: \mathbb{D}(Y)(\OFb)\times \mathbb{D}(Y)(\OFb)\rightarrow \OFb.$$ It satisfies $\langle \F x,y\rangle=\langle x, \V y\rangle^\sigma$ for any $x,y\in \mathbb{D}(Y)(\OFb)$. The $O_F$-action $\iota$ induces a $\mathbb{Z}/2 \mathbb{Z}$-grading 
$$\mathbb{D}(Y)(\OFb)=\mathbb{D}(Y)(\OFb)_0\oplus \mathbb{D}(Y)(\OFb)_1,
$$
 where each $\mathbb{D}(Y)(\OFb)_i$ is a free $\OFb$-module of rank $n$. Then $\F$ (resp.  $\V$) is of degree 1 with respect this $\mathbb{Z}/2 \mathbb{Z}$-grading. The compatibility of $\iota$ with the polarization $\lambda$ gives an $O_F$-action on $\mathbb{D}(Y)(\OFb)$ commuting with $\F,\V$ such that $\langle \iota(a)x,y\rangle=\langle x, \iota(\sigma(a))y\rangle$ for any $x,y\in \mathbb{D}(Y)(\OFb)$ and $a\in O_F$.

Let $\tau=\V^{-1}\F$, a $\sigma^2$-linear operator on the $\Fb$-isocrystal $\mathbb{D}(Y)(\OFb) \otimes {\Fb}$ which is of degree 0 with respect to the $\mathbb{Z}/2 \mathbb{Z}$-grading. The space of $\tau$-invariants $C(Y):=(\mathbb{D}(Y)(\OFb)_{0} \otimes \Fb)^{\tau=1}$ is a $F$-vector space of dimension $n$. Define a pairing on the $\Fb$-isocrystal $$(\ ,\ ): \mathbb{D}(Y)(\OFb)\otimes\Fb\times \mathbb{D}(Y)(\OFb) \otimes \Fb\rightarrow \Fb,\quad (x,y):=(\varpi \delta)^{-1}\langle x, \F y\rangle.$$  It satisfies
\begin{equation}
  \label{eq:tau}
  (x,y)=(y,\tau^{-1}(x))^\sigma
\end{equation}
 and so $(\ ,\ )$ restricts to  an $F/F_0$-hermitian form on $C(Y)$. Via the quasi-isogeny $\rho$ we may identify $C(Y)$ with the hermitian space $C(\mathbb{Y})$, which we further identify with the hermitian space $\mathbb{W}$, cf. \cite[Lem. 3.9]{Kudla2011}.

For any $\OFb$-lattice $\A\subseteq \mathbb{W}_{\Fb}\simeq C(Y)_{\Fb}$ of rank $n$, define the dual lattice $$\A^\vee:=\{ x\in \mathbb{W}_{\Fb}: (x, \A)\subseteq \OFb\}.$$ Then by \eqref{eq:tau} we have $$(\A^\vee)^\vee=\tau(A).$$

\begin{definition}
  A pair $(\A,\B)$ of $\OFb$-lattices $\A,\B\subseteq\mathbb{W}_{\Fb}$ of rank $n$ is \emph{special} if $$\B^\vee\subseteq^1 \A\subseteq^1 \B,\quad \B^\vee\subseteq^1 \A^\vee\subseteq^1 \B.$$
\end{definition}

By  \cite[Prop. 2.4, $h=1$]{Cho2019}\footnote{Note that the dual lattice in \cite{Cho2019}  is taken with respect to the form $\{\ ,\ \}=\varpi\delta(\ ,\ )$ and thus $\varpi \A^\vee$ (resp. $\varpi \B^\vee$) in \cite{Cho2019} is our $\A^\vee$ (resp. $\B^\vee$).}, the pair $(\A,\B)$ of $\OFb$-lattices of $\mathbb{W}_{\Fb}$ given by $$\A:=\mathbb{D}(Y)(\OFb)_0,\quad\B:=(\mathbf{V}(\mathbb{D}(Y)(\OFb)_1))^\vee\cong \mathbb{D}(Y^\vee)(\OFb)_0$$ is special, and the association $(Y, \iota, \lambda,\rho)\mapsto (\A,\B)$ gives a bijection  $$\Na(\kb)\simeq\{ (\A,\B) \text{ special}:  \A,\B\subseteq \mathbb{W}_{\Fb} \}.$$

\subsection{Bruhat-Tits stratification of $\Na$}\label{sec:bruh-tits-strat}

By \cite[Thm. 1.1]{Cho2019}, we have a \emph{Bruhat-Tits stratification} with closed strata $$\Nared=\bigcup_{\Lambda\in\Ver(\mathbb{W})}\mathcal{V}(\Lambda).$$
\begin{altitemize}
\item For $\Lambda\in \Ver^0(\mathbb{W})$, by \cite[Rem. 2.15]{Cho2019} we have 
$$
\mathcal{V}(\Lambda)(\bar k)=\{(\A,\B)\text{ special}: \A=\Lambda_{\OFb}\}= \{(\A,\B)\text{ special}: \A=\A^\vee\}.
$$
In this case  $\mathcal{V}(\Lambda)\simeq\mathbb{P}^{n-1}=\mathbb{P}(\Lambda_{\kb})=\mathbb{P}(\varpi^{-1}\A/\A)$, where $(\A,\B)$ corresponds to the line given by the image of $\B$ in $\varpi^{-1}\A/\A$. Moreover $\mathcal{V}(\Lambda)$'s for $\Lambda\in \Ver^0(\mathbb{W})$ are all disjoint.
\item For $\Lambda\in \Ver^{t}(\mathbb{W})$ ($2\le t\le n$, necessarily even), by \cite[Def. 2.9 (2)]{Cho2019} we have $$\mathcal{V}(\Lambda)(\kb)=\{(\A,\B)\text{ special}: \Lambda\subseteq \B^\vee\},$$ and by \cite[Prop. 3.9 (1)]{Cho2019}, $\mathcal{V}(\Lambda)$ is a closed generalized Deligne--Lusztig variety for the finite even unitary group $\U(\Lambda^\vee/\Lambda)$ (for the hermitian form induced by $(\, ,\, )$),  which has dimension $t/2$.   It has an open Deligne--Lusztig subvariety $\mathcal{V}(\Lambda)^{\circ}\subseteq \mathcal{V}(\Lambda)$ with $$\mathcal{V}(\Lambda)^\circ(\kb)=\{(\A,\B)\text{ special}: \Lambda\subseteq \B^\vee, \A\ne \A^\vee\}.$$ 
  In particular, for $(\A,\B)\in\mathcal{V}(\Lambda)^\circ(\kb)$, we have $\B^\vee=\A\cap \A^\vee$ and thus $\B$ is uniquely determined by $\A$.  For $\Lambda\in\Ver^{2}(\mathbb{W})$, we have $\CV(\Lambda)\simeq \BP^1$.
\item For $\Lambda_0\in \Ver^0(\mathbb{W})$ and $\Lambda_t\in \Ver^t(\mathbb{W})$ with $t\ge2$, by \cite[Prop. 2.18]{Cho2019} we have $\mathcal{V}(\Lambda_0)\cap \mathcal{V}(\Lambda_t)\ne\varnothing$ if and only if $\Lambda_t\subseteq \Lambda_0$. In this case \cite[Rem. 2.19]{Cho2019} we have an isomorphism $$\mathcal{V}(\Lambda_0)\cap \mathcal{V}(\Lambda_t)\simeq\mathbb{P}(\Lambda_t^\vee/\Lambda_0)\simeq\mathbb{P}^{t/2-1},$$ which at the level of $\kb$-points is given by
  \begin{equation}
    \label{eq:Lambda0t}
\mathcal{V}(\Lambda_0)(\kb)\cap \mathcal{V}(\Lambda_t)(\kb)=\{(\A,\B)\text{ special}: \A=\Lambda_{0,\OFb},\  \A\subseteq^1\B\subseteq \Lambda_{t,\OFb}^\vee\}=\mathbb{P}(\Lambda_t^\vee/\Lambda_0)(\kb).    
  \end{equation}
 In particular, when $t=2$, we have $$\mathcal{V}(\Lambda_0)(\kb)\cap \mathcal{V}(\Lambda_2)(\kb)=\{(\A,\B)=(\Lambda_{0,\OFb},\Lambda_{2,\OFb}^\vee)\}$$ consisting of a single $\kb$-point.
\item For $\Lambda\in \Ver^t(\mathbb{W})$ with $t\ge2$, by \eqref{eq:Lambda0t} we have $$\mathcal{V}(\Lambda)^\circ=\mathcal{V}(\Lambda)\setminus \bigcup_{\Lambda'\in\Ver^0(\mathbb{W})}(\mathcal{V}(\Lambda')\cap \mathcal{V}(\Lambda)).$$
\end{altitemize}

  \begin{figure}[h]
      \centering
      \includegraphics[scale=.8]{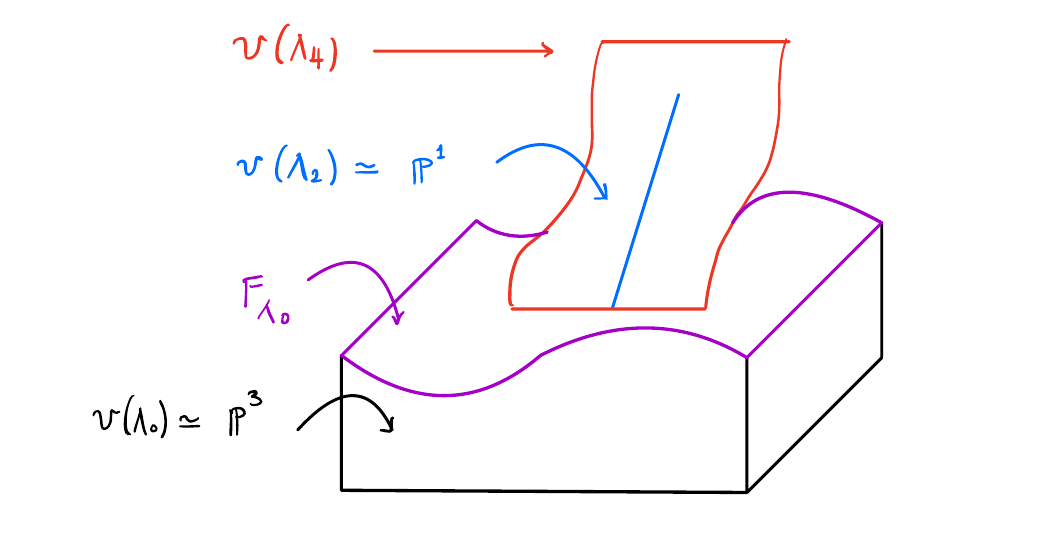}
      \caption{$\CN_{4}^{[1]}$}
      \label{fig:n=4}
    \end{figure}
\begin{definition}
  For $\Lambda\in\Ver^0(\mathbb{W})$, define $$F_\Lambda:=\{\ell\in \mathbb{P}(\Lambda_{\kb}): \ell\subseteq \ell^\perp \}\subseteq \mathcal{V}(\Lambda),$$ a Fermat hypersurface of degree $q+1$, comp. \cite[Ex. 5.6]{VW}. By \eqref{eq:Lambda0t} we know that  for $\Lambda\in \Ver^0(\mathbb{W})$, 
  $$ \bigcup_{\Lambda'\in \Ver^{\ge2}(\mathbb{W})}(\mathcal{V}(\Lambda)\cap\mathcal{\mathcal{V}}(\Lambda'))\subset F_\Lambda.$$
\end{definition}
The following terminology is due to Kudla \cite{Kudla2012}.
\begin{definition}\label{strata}
The \emph{closed balloon locus} of $\Na$ is the reduced closed subscheme of $\Na\otimes \bar k$ given by  $$\CN_n^{[1], \bullet}:=\bigsqcup_{\Lambda\in \Ver^0(\mathbb{W})}\mathcal{V}(\Lambda)\simeq\bigsqcup_{\Lambda\in \Ver^0(\mathbb{W})} \mathbb{P}^{n-1}.$$ 
Define the \emph{link stratum} to be 
$$
\Nlk=\bigsqcup_{\Lambda\in \Ver^0(\mathbb{W})}F_\Lambda,
$$
 a $(n-2)$-dimensional reduced closed subscheme of $\CN_n^{[1], \bullet}$.  Define the open \emph{balloon stratum} to be $$\Nbl:=\CN_n^{[1],\bullet}\setminus \Nlk,$$ a reduced locally closed subscheme of $\Na$.  Define the  \emph{non-special locus} to be
  $$\Nns:=\Na\setminus \CN_n^{[1], \bullet},$$
   an open formal subscheme of $\Na$.  Note that $$\Nbl=\bigsqcup_{\Lambda\in \Ver^0(\mathbb{W})} (\mathcal{V}(\Lambda)\setminus F_\Lambda),\quad (\Nns)_\mathrm{red}=\bigcup_{\Lambda\in \Ver^{\ge2}(\mathbb{W})} \mathcal{V}(\Lambda)^\circ.$$
\end{definition}
\begin{remark}
We emphasize that the loci $\CN_n^{[1], \bullet}$, $\Nlk$ and $\Nbl$ are schemes contained in the underlying reduced scheme of $\Na$. In \cite[\S 5.3]{ZZha}, Z.~Zhang represents the special fiber $\Na\otimes_{O_{\breve F}}\bar k$ as the union of two formal Cartier divisors $\CN_n^{[1], \bullet}$ and $\CN_n^{[1], \circ}$ of  $\Na$ and introduces their intersection $\Nlk$. The dictionary between his definitions and ours is as follows: his balloon stratum $\CN_n^{[1], \circ}$ is a scheme and equals our $\CN_n^{[1], \bullet}$; his ground stratum $\CN_n^{[1], \bullet}$ is not a scheme but has as underlying reduced scheme the union of $\bigcup_{\Lambda\in \Ver^{\ge2}(\mathbb{W})} \mathcal{V}(\Lambda)^\circ$ and $\bigsqcup_{\Lambda\in \Ver^0(\mathbb{W})} F_\Lambda$;  his link stratum $\Nlk$ is a scheme and coincides with our $\Nlk$.
\end{remark}

\begin{proposition}[Singularities of $\Na$]\label{prop:localmodel}
 The complete local ring of $\Na$ at $z\in \Na(\kb)$ is isomorphic to $\OFb[[X_1,\ldots, X_n]]/(X_1X_2-\varpi)$ if $z\in\Nlk(\kb)$, resp. to $\OFb[[X_1,\ldots, X_{n-1}]]$ if $z\not\in \Nlk(\kb)$.
\end{proposition}

\begin{proof}
Let $\mathcal{M}^\mathrm{loc}$ be the standard Drinfeld local model over $\Spec \OFb$ for the periodic lattice chain associated to two adjacent lattices: for an $\OFb$-scheme $S$, $\mathcal{M}^\mathrm{loc}(S)$ is the set of isomorphism classes of commutative diagrams of locally free $O_S$-modules $$\xymatrix{\Lambda_{0,S} \ar[r]^{\phi_0} & \Lambda_{1,S} \ar[r]^{\phi_1} & \Lambda_{0,S}\\ \mathcal{F}_0 \ar[u] \ar[r]  & \mathcal{F}_1 \ar[u] \ar[r]  &  \mathcal{F}_0 \ar[u]},$$ where
  \begin{altenumerate}
  \item  $\Lambda_0=\langle e_1,\ldots, e_n\rangle$, $\Lambda_1=\langle \varpi^{-1}e_1, e_2,\ldots e_n\rangle$ are free $\OFb$-modules of rank $n$, 
  \item $\phi_0=\id$, $\phi_1=\varpi$,
  \item $\mathcal{F}_i$ are locally free $O_S$-submodules of rank $n-1$ which Zariski-locally on $S$ are direct summands of $\Lambda_{i,S}$.
  \end{altenumerate}
For $z'\in \mathcal{M}^\mathrm{loc}(\kb)$, the complete local ring $\hat O_{z'}$ is isomorphic to $\OFb[[X_1,\ldots, X_{n}]]/(X_1X_2-\varpi)$ if $e_1\in \mathcal{F}_0$ and $\Lambda_{1,\kb}/\mathcal{F}_1$ is generated by $\varpi^{-1}e_1$,  and is isomorphic to $\OFb[[X_1,\ldots, X_{n-1}]]$ otherwise (\cite[ Thm. 4.12, \S4.4.5 with $\kappa=1$, $r=n-1$]{Go}). The local model of $\Na$ over $\OFb$ is isomorphic to $\mathcal{M}^\mathrm{loc}$ (\cite[\S5]{RSZ2}).

Let $z\in \Na(\kb)$. The corresponding $\kb$-point $z'\in \mathcal{M}^\mathrm{loc}(\kb)$ with isomorphic complete local ring $\hat O_z\simeq \hat O_{z'}$ is given by the diagram $$\xymatrix{\mathbb{D}(X_z)(\kb)_0=\A/\varpi \A \ar[r]^-{\phi_0} & \mathbb{D}(X_z^\vee)(\kb)_0=\B/\varpi \B \ar[r]^-{\phi_1} & \A/\varpi \A\\ \mathbf{V}(\mathbb{D}(X_z)(\kb)_1)=\tau^{-1}(\B^\vee)/\varpi \A \ar[u] \ar[r]  & \mathbf{V}(\mathbb{D}(X_z^\vee)(\kb)_1)=\tau^{-1}(\A^\vee)/\varpi \B \ar[u] \ar[r]  &  \tau^{-1}(\B^\vee)/\varpi \A. \ar[u]}$$

At a singular point $z$ (equivalently at a singular point $z'$),  the condition  that $\Lambda_{1,\kb}/\mathcal{F}_1$ is generated by $\varpi^{-1}e_1$ becomes the condition that $\B/\tau^{-1}(\A^\vee)$ is generated by $\varpi^{-1}e_1$, and the condition $e_1\in \mathcal{F}_0$ becomes the condition that $\tau^{-1}(\B^\vee)/\varpi \A$ contains $e_1$. Since $\B/\A$ is generated by $\varpi^{-1}e_1$, we know that $\tau^{-1}(\A^\vee)=\A$, and hence $\tau(\A)=\A=\Lambda_{\OFb}$ for some $\Lambda\in\Ver^0(\mathbb{W})$. Moreover, as $e_1\in \tau^{-1}(\B^\vee)/\varpi \A$, the $\kb$-line $\langle e_1\rangle=\varpi\B/\varpi\A\subseteq \Lambda_{\kb} $ satisfies $\langle e_1\rangle\subseteq \langle e_1\rangle^\perp$, and hence corresponds to a point on the Fermat hypersurface $F_\Lambda$. Hence $\hat O_z\simeq \OFb[[X_1,\ldots, X_m]]/(X_1X_2-\varpi)$ exactly when $z\in \Nlk(\kb)$.
\end{proof}

\subsection{$\kb$-points of $\CN_{n+1}^{[1]}$} 
\label{ss: k pt N(n+1)}
Let $(X, \iota, \lambda,\rho)\in \CN_{n+1}^{[1]}(\kb)$.  The (relative) Dieudonn\'e module $\mathbb{D}(X)(\OFb)$ is a free $\OFb$-module of rank $2(n+1)$, equipped with the action of the $\sigma$-linear  Frobenius $\F$ and the $\sigma^{-1}$-linear Verschiebung $\V$.  The principal polarization $\lambda$ induces a perfect alternating $\OFb$-bilinear form on the Dieudonn\'e module $$\langle\ ,\ \rangle: \mathbb{D}(X)(\OFb)\times \mathbb{D}(X)(\OFb)\rightarrow \OFb.$$ It satisfies $\langle \F x,y\rangle=\langle x, \V y\rangle^\sigma$ for any $x,y\in \mathbb{D}(X)(\OFb)$. The $O_F$-action $\iota$ induces a $\mathbb{Z}/2 \mathbb{Z}$-grading $$\mathbb{D}(X)(\OFb)=\mathbb{D}(X)(\OFb)_0\oplus \mathbb{D}(X)(\OFb)_1,
$$ 
where each $\mathbb{D}(X)(\OFb)_i$ is a free $\OFb$-module of rank $n$. Then $\F$ (resp.  $\V$) is of degree 1 with respect this $\mathbb{Z}/2 \mathbb{Z}$-grading. The compatibility of $\iota$ with the polarization $\lambda$ gives an $O_F$-action on $\mathbb{D}(X)(\OFb)$ commuting with $\F,\V$ such that $\langle \iota(a)x,y\rangle=\langle x, \iota(\sigma(a))y\rangle$ for any $x,y\in \mathbb{D}(X)(\OFb)$ and $a\in O_F$.

Let $\tau=\V^{-1}\F$, a $\sigma^2$-linear operator on the $\Fb$-isocrystal $\mathbb{D}(X)(\OFb) \otimes {\Fb}$ which is of degree 0 with respect to the $\mathbb{Z}/2 \mathbb{Z}$-grading. The space of $\tau$-invariants $C(X):=(\mathbb{D}(X)(\OFb)_{0} \otimes \Fb)^{\tau=1}$ is a $F$-vector space of dimension $n$. Define a pairing on the $\Fb$-isocrystal $$(\ ,\ ): \mathbb{D}(X)(\OFb)\otimes\Fb\times \mathbb{D}(X)(\OFb) \otimes \Fb\rightarrow \Fb,\quad (x,y):=(\varpi \delta)^{-1}\langle x, \F y\rangle.$$  It satisfies
\begin{equation}
  \label{eq:tau'}
  (x,y)=(y,\tau^{-1}(x))^\sigma
\end{equation}
 and so $(\ ,\ )$ restricts to  an $F/F_0$-hermitian form on $C(X)$. Via the quasi-isogeny $\rho$ we may identify $C(X)$ with the hermitian space $C(\mathbb{X})$, which we further identify with the hermitian space $\mathbb{V}$, cf. \cite[Lemma 3.9]{Kudla2011}.

For any $\OFb$-lattice $A\subseteq \mathbb{W}_{\Fb}\simeq C(X)_{\Fb}$ of rank $n$, define the dual lattice $$A^\vee:=\{ x\in \mathbb{W}_{\Fb}: (x, A)\subseteq \OFb\}.$$ Then by \eqref{eq:tau'} we have $$(A^\vee)^\vee=\tau(A).$$

\begin{definition}
An $\OFb$-lattice $A\subseteq \mathbb{V}_{\Fb}$ of rank $n+1$ is \emph{special} if $$\varpi A\subseteq^{n} A^\vee\subseteq^1 A.$$
\end{definition}
 For $(X,\iota,\lambda,\rho)\in \CN_{n+1}^{[1]}(\kb)$, by \cite[Prop. 1.10]{Vol10}  the $\OFb$-lattice  $$A:=\mathbb{D}(X)(\OFb)_0\subseteq \mathbb{V}_{\Fb}$$ is special, and the association $(X,\iota,\lambda,\rho)\mapsto A$ gives a bijection
\begin{equation}
  \label{eq:eq:bijN}
  \CN_{n+1}^{[1]}(\kb)\simeq\{\text{special lattices } A\subseteq \mathbb{V}_{\Fb} \}.
\end{equation}

\subsection{Bruhat--Tits stratification of $\CN_{n+1}^{[1]}$}\label{sec:bruh-tits-strataself}

By \cite[Thm. B]{VW}, we have a \emph{Bruhat--Tits stratification} of $\mathcal{N}^{[1],\mathrm{red}}$ with closed strata $$\mathcal{N}^{[1],\mathrm{red}}=\bigcup_{\Lambda\in \Ver(\mathbb{V})}\mathcal{V}(\Lambda).$$ Each closed Bruhat--Tits stratum $\mathcal{V}(\Lambda)$ is a generalized Deligne--Lusztig variety associated to the finite odd unitary group $\U(\Lambda^\vee/\Lambda)$, which has dimension $(t(\Lambda)-1)/2$. It has $\kb$-points $$\mathcal{V}(\Lambda)(\kb)=\{A\text{ special}: \Lambda\subseteq A^\vee\}.$$ In particular, for $\Lambda\in \Ver^1(\mathbb{V})$, we have $A\in \mathcal{V}(\Lambda)(\kb)$ if and only if 
  $A^\vee=\Lambda_{\OFb}$.
\begin{definition}
  Define the \emph{superspecial locus} $\CN_{n+1}^\mathrm{ss}$ of $\CN_{n+1}^{[1]}$ to be $$\N^\mathrm{ss}:=\bigsqcup_{\Lambda\in \Ver^1(\mathbb{V})}\mathcal{V}(\Lambda),$$ a closed reduced subscheme of dimension zero of $\CN_{n+1}^{[1]}$. 

\end{definition}

\subsection{The special divisor $\Zx$ on $\mathcal{N}_{n+1}^{[1]}$} 
 \label{ss: Zx}
Let $\Zx\subseteq \mathcal{N}_{n+1}^{[1]}$ be the special divisor considered in \S\ref{sec:relat-with-spec}.   The bijection \eqref{eq:eq:bijN}  restricts to   a bijection (see \cite[Prop.~3.1]{Kudla2011}) $$\mathcal{Z}(\x)(\kb)\simeq \{\text{special lattices }A\subseteq \mathbb{V}_{\Fb}: \x\in A^\vee\}.$$
 Under the identification $\BW=\langle u\rangle^\perp$,  each vertex lattice $\Lambda\in \Ver^t(\mathbb{W})$ gives rise to a vertex lattice $$\Lambda(\x):=\Lambda\obot \langle \x\rangle\in \Ver^{t+1}(\mathbb{V}).$$

\begin{definition}\label{def:centloc}
Define the \emph{center point locus} of $ \Zx$ to be $$\Zx^\mathrm{cent}:=\bigsqcup_{\Lambda\subseteq \Ver^0(\mathbb{W})}\mathcal{V}(\Lambda(\x)),$$ a 0-dimensional closed reduced subscheme of $\Zx$. In particular, $\Zx^\mathrm{cent}\subseteq \Zx\cap \mathcal{N}_{n+1}^\mathrm{ss}$. Note that $z\in \Zx^\mathrm{cent}$ if and only if $A^\vee=\Lambda(\x)_{\OFb}$, for some $\Lambda\in \Ver^0(\mathbb{W})$.

Define the \emph{non-special locus} 
$$\Zns:=\Zx\setminus\Zx^\mathrm{cent},$$ an open formal subscheme of $\Zx$.
\end{definition}
The proof of the following theorem is given in \S \ref{ss:structZ}.
\begin{theorem}\label{structZ}
The formal scheme $\CZ(u)$ is regular of dimension $n$, and formally smooth over $\Spf O_{\breve F}$ outside the  zero-dimensional closed subset $\Zx^\mathrm{cent}$ of $\CZ(u)^{\rm red}$. 
\end{theorem}
\subsection{Geometry of $\tN$}  \label{ss: tN}
We first introduce the following loci in $\tN$.

\begin{definition}
    Define the \emph{exceptional locus} $\wt{\CN}_n^{[1], {\rm exc}}:=\pi_2^{-1}(\Zx^\mathrm{cent})$, a closed formal subscheme of $\tN$. Define the \emph{non-special locus} $\tNns:=\tN\setminus\wt{\CN}_n^{[1], {\rm exc}}$, an open formal subscheme of $\tN$. 
  \end{definition}

Recall the projection morphisms $\pi_1, \pi_2$ from (\ref{eq:tNr3}). The following Theorem was conjectured in \cite{
Kudla2012}. The proof will occupy us in the next sections. The end of the proof is  in \S\ref{sec:proof-conj-refc}.

\begin{theorem}[Geometry of $\tN$]\label{conj:KRSZ}
\quad
    \begin{altenumerate}
    \item\label{item:KR1} The formal scheme $\tN$ is regular of dimension $n$. 
  \item\label{item:KR2} The morphism $\pi_1$ is finite flat of degree $q+1$, \'etale away from $\CN_n^{[1], \bullet}$, and totally ramified along $\CN_n^{[1], \bullet}$.      
   \item\label{item:KR3} The morphism $\pi_2$ is proper. Its restriction to $\tNns$   induces an isomorphism $\tNns\simeq \Zns$.
   \item\label{item:KR4} The closed formal subscheme $\wt{\CN}_n^{[1], {\rm exc}}$ of  $\tN$ is a reduced Cartier divisor and isomorphic to $\CN_n^{[1], \bullet}$ under $\pi_1$.  In particular, for $\Lambda\in \Ver^0(\mathbb{W})$, $\PL:=\pi_2^{-1}(\mathcal{V}(\Lambda(\x)))$ is isomorphic to $\mathcal{V}(\Lambda)\simeq \mathbb{P}^{n-1}$ under $\pi_1$ and we have a decomposition \begin{equation*}
  \wt{\CN}_n^{[1], {\rm exc}}=\bigsqcup_{\Lambda\in \Ver^0(\mathbb{W})}\PL.
\end{equation*}

   \item\label{item:KR5} For $\Lambda_1,\cdots,\Lambda_n\in\Ver^0(\BW)$,  we have
   \begin{align}\label{eq:int Lambda i}
\chi\bigl(\tN,\BP_{\Lambda_1}\cap^\BL\cdots\cap^\BL\BP_{\Lambda_n}\bigr)=\begin{cases}
(-1)^{n-1},&  \Lambda_1=\cdots=\Lambda_n,\\
0,&\text{otherwise}.\end{cases}
\end{align}
    \end{altenumerate}
  \end{theorem}

    \begin{example}
    When $n=2$, we have
    \begin{itemize}
    \item $\Na$ is isomorphic to the Drinfeld half plane, whose special fiber is a union of $\mathbb{P}^1$'s with dual graph a $(q+1)$-valent tree. The $\BP^1$'s are either of the form $\mathcal{V}(\Lambda_0)$ for a $\Lambda_0\in \Ver^0(\mathbb{W})$ or of the form $\mathcal{V}(\Lambda_2)$ for a $\Lambda_2\in \Ver^2(\mathbb{W})$. 
    \item $\CN_n^{[1], \bullet}$ consists of the $\mathcal{V}(\Lambda_0)$ for a $\Lambda_0\in \Ver^0(\mathbb{W})$, with nonreduced preimage $\pi_1^{-1}(\mathcal{V}(\Lambda_0))$   (a ``fat'' $\mathbb{P}^1$ with multiplicity $q+1$).  The preimage $\pi_1^{-1}(\mathcal{V}(\Lambda_2))$ of $\mathcal{V}(\Lambda_2)$ is a Fermat curve of degree $q+1$.
    \item The special fiber of $\Zx$ consists of Fermat curves of degree $q+1$ intersecting at points in $\Zx^\mathrm{cent}$. Each Fermat curve contains  $q+1$ intersection points and $q+1$ Fermat curves pass through each intersection point.
    \item       The preimage of $z=\mathcal{V}(\Lambda_0(\x))\in\Zx^\mathrm{cent}$ under $\pi_2$ is an exceptional divisor $\pi_1^{-1}(\mathcal{V}(\Lambda_0))^\mathrm{red}\simeq\mathbb{P}^1$.
    \end{itemize}
    Figure       \ref{fig:n=2} in the Introduction illustrates the morphisms $\pi_1$ and $\pi_2$ in (\ref{eq:tNr3}) (for $n=2$ and $r=1$) on the special fibers locally around a superspecial point of $\Zx$.
  \end{example}

\subsection{The morphisms $\pi_1,\pi_2$ on $\kb$-points}  As a first step towards proving Theorem \ref{conj:KRSZ}, we study the properties of $\pi_1$ and $\pi_2$ at the level of $\kb$-points.

\begin{lemma}\label{lem:bijectionpoints} The following assertions hold.
  \begin{altenumerate}
  \item\label{item:1}   $\pi_2$ induces a bijection $\pi_1^{-1}(\Nns)(\kb)\simeq \Zns(\kb)$.
  \item\label{item:2} $\pi_2$ maps $\pi_1^{-1}(\CN_n^{[1], \bullet})(\kb)$ onto $\Zx^\mathrm{cent}(\kb)$ with fibers isomorphic to $\mathbb{P}^{n-1}(\kb)$.
  \item\label{item:3} $\pi_1$ induces a bijection $\wt{\CN}_n^{[1], {\rm exc}}(\kb)\simeq \CN_n^{[1], \bullet}(\kb)$.   \item\label{item:4} $\pi_1$ maps $\tNns(\kb)$ onto $\Nns(\kb)$ with fibers of size $q+1$.
  \end{altenumerate}
\end{lemma}

\begin{proof}
  Let $z=((\A, \B), A)\in (\Na\times\CN_{n+1}^{[1]})(\kb)$. Then $z\in \tN(\kb)$ if and only if  $$\A(\x)\subseteq^1 A,\quad\B^\vee(\x)\subseteq^1 A^\vee.$$  
In this case, we have $\A=A\cap \mathbb{W}_\Fb$ and $\B^\vee=A^\vee \cap \mathbb{W}_\Fb$. For $z\in \tN(\kb)$, by definition we have $z\in \pi_1^{-1}(\Nns)$ if and only if $\A\ne \A^\vee$.

  On the other hand, we have $$\Zx(\kb)=\{A\text{ special}: \x\in A^\vee \},$$ and in this case by definition $A\in \Zns(\kb)$ if and only if $A^\vee\cap \mathbb{W}_\Fb\text{ does not contain any }\Lambda_0\in\Ver^0(\mathbb{W})$, if and only if $A\cap \mathbb{W}_\Fb$ is not $\tau$-invariant. 
  \begin{altenumerate}
  \item Let $z=((\A,\B),A)\in \pi_1^{-1}(\Nns)(\kb)$. Then $A\cap \mathbb{W}_{\Fb}=\A$ satisfies $\A\ne \A^\vee$, hence $A\cap \mathbb{W}_{\Fb}$ is not $\tau$-invariant and thus $\pi_2(z)\in \Zns(\kb)$. Conversely,   let $A\in \Zns(\kb)$. Then $A\cap \mathbb{W}_\Fb$ is not $\tau$-invariant. Let $\A=A\cap \mathbb{W}_{\Fb}$. Then $\A\ne\A^\vee$ and hence determines a unique $z=((\A,\B),A)\in \pi_1^{-1}(\Nns)(\kb)$ such that $\pi_2(z)=A$.
  \item If $z\in \pi_1^{-1}(\CN_n^{[1], \bullet})(\kb)$, then $\A=\A^\vee=\Lambda_{\OFb}$ for some $\Lambda\in \Ver^0(\mathbb{W})$. Hence $\A(\x)=\Lambda_{\OFb}(\x)\subseteq^1 A$ and $A^\vee \subseteq^1 A$ implies that $A^\vee=\Lambda_{\OFb}(\x)$ and thus $\pi_2(z)=A\in\mathcal{V}(\Lambda(\x))(\kb)\subseteq \Zx^\mathrm{cent}(\kb)$. Conversely, for a fixed $A\in\mathcal{V}(\Lambda(\x))(\kb)$, we have $\pi_2(z)=A$ if and only if $A^\vee=\Lambda_{\OFb}(\x)$, if and only if $\A=\Lambda_{\OFb}$, and thus $\pi_2^{-1}(A)(\kb)\simeq\{\B: \A\subseteq^1 \B\}=\mathbb{P}(\Lambda_{\kb})(\kb)$.
  \item This follows from the description of the fibers of $\pi_2$ in (ii).
  \item Let $(\A,\B)\in \Nns(\kb)$. Then $\A\ne \A^\vee$ and $\B^\vee=\A\cap \A^\vee$. We have $\pi_1(z)=(\A,\B)$ if and only if $\B^\vee(\x)\subseteq A^\vee$. Since $$\B^\vee(\x)\subseteq^1 A^\vee\subseteq^1 A\subseteq^1 \B(\varpi^{-1}\x),\quad  \B^\vee(\x)\subseteq^1 \A(\x)\subseteq^1 A\subseteq^1 \B(\varpi^{-1}\x),$$ we know that the choices of $A$ such that $z=((\A,\B),A)\in \tN$ are in bijection with isotropic lines $\ell:=A^\vee/\B^\vee(\x)$ in the 3-dimensional space $\B(\varpi^{-1}\x)/\B^\vee(\x)$ such that $\ell$ is orthogonal to the anisotropic line $\ell':=\A(\x)/\B^\vee(\x)$. Hence the number of choices of $A$ is equal to the number of isotropic lines $\ell$ in the 2-dimensional space $(\ell')^\perp$, which is equal to $q+1$. \qedhere
  \end{altenumerate}
\end{proof}

\section{Deformation theory}\label{s:defthy}
  In this section we collect necessary deformation theoretic facts needed for the study of $\tN$. Let $R$ be a local noetherian $\OFb$-algebra on which $\varpi$ is nilpotent. Let $I\subseteq R$ be an ideal such that $I^2=0$. Let $S=R/I$. Then $R$ is a thickening of $S$, equipped with the trivial nilpotent divided power structure on $I$.

\subsection{The spaces $\N$ and $\Zx$}\label{sec:spaces-n} Let $(X,\iota, \lambda,\rho)\in \N(S)$. Then $\mathbb{D}(X)(S)$ is a free $S$-module of rank $2(n+1)$. We have the Hodge filtration $\Fil^1\mathbb{D}(X)(S)\subseteq \mathbb{D}(X)(S)$, a free $S$-module of rank $n+1$ with free factor module. The polarization $\lambda$ induces an $S$-alternating pairing on $\mathbb{D}(X)(S)$ such that $\Fil^1\mathbb{D}(X)(S)$ is totally isotropic.

Denote by $\{\iota_F^i: O_F\rightarrow\OFb, i\in \mathbb{Z}/ 2 \mathbb{Z}\}$ the two conjugate embeddings that are the  identity on $O_{F_0}$. Define $\mathbb{D}(X)(S)_i\subseteq \mathbb{D}(X)(S)$ to be the maximal $S$-submodule on which $\iota(O_F)$ acts via the the map $O_F\xrightarrow{\iota_F^i} \OFb\rightarrow S$, a free $S$-module of rank $n+1$. We have a Hodge exact sequence of free $S$-modules $$0\rightarrow \Fil^1\mathbb{D}(X)(S)_i\rightarrow \mathbb{D}(X)(S)_i\rightarrow \Lie X(S)_i\rightarrow0.$$ By the signature condition we know that $\Fil^1\mathbb{D}(X)(S)_i$ has rank $n$ for $i=0$ and rank 1 for $i=1$. The principal polarization $\lambda$ induces a perfect $S$-bilinear pairing $$\langle\ ,\ \rangle_i: \mathbb{D}(X)(S)_i \times \mathbb{D}(X)(S)_{i+1}\rightarrow S,$$ which induces a perfect pairing $$\Fil^1\mathbb{D}(X)(S)_i\times \Lie X(S)_{i+1}\rightarrow S.$$ In particular, $\Fil^1\mathbb{D}(X)(S)_0$ determines $\Fil^1\mathbb{D}(X)(S)_1$ and vice versa.

Note that $\mathbb{D}(X)(R)$ is a free $R$-module of rank $2(n+1)$ and $\mathbb{D}(X)(S)=\mathbb{D}(X)(R)\otimes_R S$. By Grothendieck--Messing theory, a lifting of $(X,\iota,\lambda,\rho)\in \N(S)$ to $\N(R)$ corresponds to a free $R$-module $$\Fil^1\mathbb{D}(X)(R)_0\subseteq \mathbb{D}(X)(R)_0$$ lifting $\Fil^1\mathbb{D}(X)(S)_0$.

For $(X,\iota,\lambda,\rho)\in \Zx(S)$, the special homomorphism $\x\in \Hom_{O_F}(\ov\CE, X)$ induces a homomorphism of free $R$-modules $$\x_{*}: \mathbb{D}(\ov\CE)(R)_0\rightarrow \mathbb{D}(X)(R)_0,$$ which preserves the Hodge filtration when base changing to $S$,
 $$\x_{*}(\Fil^1\mathbb{D}(\ov\CE)(S)_0)\subseteq \Fil^1\mathbb{D}(X)(S)_0 ,$$
  where the source and target have dimensions 1 and $n$ respectively. The lifting $(X,\iota,\lambda,\rho)\in \N(R)$ lies in $\Zx(R)$ if and only if $\x_{*}$ preserves the Hodge filtration: $$\x_{*}(\Fil^1\mathbb{D}(\ov\CE)(R)_0)\subseteq \Fil^1\mathbb{D}(X)(R)_0.$$

\begin{lemma}\label{lem:Zss}
Let $z\in \Zx(\kb)$.  Then  the dimension of the tangent space is given by
$$\dim_{\kb}T_z\Zx_{\kb}=
\begin{cases}
  n-1, & z \in \Zns(\kb), \\
  n, & z \in \Zx^\mathrm{cent}(\kb).
\end{cases}
$$
In particular, $\Zns$ is formally smooth over $\Spf \OFb$.
\end{lemma}

\begin{proof}
 Let $z\in \Zx(\kb)$. Let $\hat O_z$ be the complete local ring of $\Zx$ at $z$. Let $A\subseteq \mathbb{V}_{\Fb}$ be the special lattice associated to $z$.   Recall that the reduction $A_{\kb}=A/\varpi A$ is equipped with the Hodge filtration $\Fil^1 A_{\kb}\subseteq A_{\kb}$, a hyperplane given by the image of $A^\vee$. By Grothendieck--Messing theory, for any local Artinian $\OFb$-algebra $R$ such that the kernel $R\rightarrow \kb$ is equipped with a nilpotent divided power structure, the set $\Hom_{\OFb}(\hat O_z, R)$ is in bijection with $R$-hyperplanes $H_R\subseteq A_R$ lifting $\Fil^1A_{\kb}$ such that $\x\in H_R$. Let $\mathfrak{m}$ be the maximal ideal of $R$. 

 Since $z\in \Zx(\kb)$, we know that $\x\in A^\vee$. Since $A^\vee\subseteq A$ and $\val(\x)=1$, we know that $\varpi^{-1}\x\not\in A^\vee$. We distinguish two cases.
  \begin{altenumerate}
  \item If $\varpi^{-1}\x\in A$, then we know that the line $A/A^\vee$ is generated by the image of $\varpi^{-1}\x$, and hence $$A=\Lambda_0 \obot \langle \varpi^{-1}\x\rangle, \quad A^\vee=\Lambda_0 \obot \langle \x\rangle,$$ for $\Lambda_0\subseteq \mathbb{V}_{\Fb}$ a self-dual lattice of rank $n$. By $$A^\vee=\Lambda_0 \obot \langle \x\rangle\subseteq^1\tau(A)=\tau(\Lambda_0)\obot \langle \varpi^{-1}\x\rangle,$$ we obtain that $\Lambda_0=\tau(\Lambda_0)$, so $A=\tau(A)$, and hence $z\in \Zx^\mathrm{ss}$. In this case, we may extend $\varpi^{-1}\x$ to an $\OFb$-basis $\{e_0=\varpi^{-1}\x, e_1,\ldots e_{n}\}$ of $A$. The hyperplane $\Fil^1 A_{\kb}$ is given by the equation $$e_0^*=0.$$ Then a hyperplane $H_R\subseteq A_R$ lifting $\Fil^1 A_{\kb}$ is given by an equation of the form
    \begin{equation}\label{eq:hyperplane}
      e_0^*+ X_1 e_1^*\cdots + X_{n} e_{n}^*=0,\quad  X_1,\ldots, X_{n}\in \mathfrak{m}.
    \end{equation}
    Since $e_0^*(\x)=e_0^*(\varpi e_0)=\varpi$, $e_i^*(\x)=0$ ($i\ge1$), the condition $\x\in H_R$ becomes $$\varpi =0,$$ and it follows that $$\Hom_{\OFb}(\hat O_z,R)=\Hom_\OFb(\kb[[X_1,\ldots,X_{n}]], R).$$ (This is not enough information to determine $\hat O_z$ since the kernel $R\rightarrow k$ is required to have a nilpotent divided power structure).  In particular, taking $R=k[\varepsilon]/\varepsilon^2$ we obtain the tangent space of the special fiber has dimension $$\dim T_z \Zx_{\kb}=n.$$

  \item If $\varpi^{-1}\x\not\in A$, then $z\not\in \Zx^\mathrm{cent}$.  We may extend $\x$ to an $\OFb$-basis $\{e_0=\x,e_1,\ldots, e_{n}\}$ of $A$. After changing basis we may assume that the hyperplane $\Fil^1 A_{\kb}$ is given by $$e_{n}^*=0.$$  Then a hyperplane $H_R\subseteq A_R$ lifting $\Fil^1 A_{\kb}$ is given by an equation of the form
    \begin{equation}\label{eq:hyperplane}
      X_0 e_0^*+ X_1e_1^*\cdots + X_{n-1} e_{n-1}^*+e_{n}^*=0,\quad  X_0,\ldots, X_{n-1}\in \mathfrak{m}.
    \end{equation} But now $e_0^*(\x)=1$, so the condition $\x\in H_R$ becomes $$X_0=0,$$ and it follows that $$\Hom_\OFb(\hat O_z, R)=\{(X_1,\ldots, X_{n-1}), X_i\in \mathfrak{m}\}=\Hom_\OFb(\OFb[[X_1,\ldots, X_{n-1}]],R).$$ In particular, taking $R=k[\varepsilon]/\varepsilon^2$ we obtain $$\dim T_z \Zx_{\kb}=n-1,$$ and hence the special fiber $\Zx_{\kb}$ is formally smooth at $z$ and so $\hat O_z\simeq\OFb[[X_1,\ldots, X_{n-1}]]$.\qedhere
  \end{altenumerate}    
\end{proof}  
\subsection{Proof of Theorem \ref{structZ}}\label{ss:structZ}

The regularity of $\CZ(u)$ follows from \cite{Ter}, since $\CZ(u)$ is its own difference divisor. The second assertion follows from Lemma \ref{lem:Zss} and the fact that $ \Zx^\mathrm{cent}$ is zero-dimensional, cf. Definition \ref{def:centloc}. 

\subsection{The space $\Na$}
Let $(Y,\iota, \lambda,\rho)\in \Na(S)$. Then $\mathbb{D}(Y)(S)$ is a free $S$-module of rank $2n$. We have a Hodge filtration $\Fil^1\mathbb{D}(Y)(S)\subseteq \mathbb{D}(Y)(S)$, a free $S$-module of rank $n$. The polarization $\lambda$ induces an $S$-alternating pairing on $\mathbb{D}(Y)(S)$ such that $\Fil^1\mathbb{D}(Y)(S)$ is totally isotropic. 

We have the Hodge exact sequence of free $S$-modules $$0\rightarrow \Fil^1\mathbb{D}(Y)(S)_i\rightarrow \mathbb{D}(Y)(S)_i\rightarrow \Lie Y(S)_i\rightarrow0.$$ By the signature condition we know that $\Fil^1\mathbb{D}(Y)(S)_i$ has rank $n-1$ for $i=0$ and rank 1 for $i=1$. The polarization $\lambda$ induces an $S$-bilinear pairing $$\langle\ ,\ \rangle_i: \mathbb{D}(Y)(S)_i \times \mathbb{D}(Y)(S)_{i+1}\rightarrow S.$$

When $\varpi=0$ in $S$, let $\mathbb{D}(Y)(S)_i^\perp\subseteq \mathbb{D}(Y)(S)_{i+1}$ be the orthogonal complement under $\langle\ ,\ \rangle$, which has rank 1 by the almost principal assumption on $\lambda$. Let $z=(Y,\iota,\lambda,\rho_Y)\in \Na(S)$, then
\begin{altenumerate}

\item $z\in \Nbl(S)$ if and only if $\mathbb{D}(Y)(S)_0^\perp=\Fil^1\mathbb{D}(Y)(S)_1$ and  $\mathbb{D}(Y)(S)_1^\perp\not\subseteq \Fil^1\mathbb{D}(Y)(S)_0$,
\item $z\in \Nns(S)$if and only if $\mathbb{D}(Y)(S)_0^\perp\ne\Fil^1\mathbb{D}(Y)(S)_1$ and $\mathbb{D}(Y)(S)_1^\perp\subseteq \Fil^1\mathbb{D}(Y)(S)_0$,
\item $z\in \Nlk(S)$ if and only if $\mathbb{D}(Y)(S)_0^\perp=\Fil^1\mathbb{D}(Y)(S)_1$ and $\mathbb{D}(Y)(S)_1^\perp\subseteq \Fil^1\mathbb{D}(Y)(S)_0$.
\end{altenumerate}

Note that $\mathbb{D}(Y)(R)$ is a free $R$-module of rank $2n$ and $\mathbb{D}(Y)(S)=\mathbb{D}(Y)(R)\otimes_R S$. By Grothendieck--Messing theory, a lifting of $(Y,\iota, \lambda,\rho_Y)\in \Na(S)$ to $\Na(R)$ corresponds to free $R$-modules $$\Fil^1\mathbb{D}(Y)(R)_i\subseteq \mathbb{D}(Y)(R)_i$$ lifting $\Fil^1\mathbb{D}(Y)(S)_i$ $i\in \mathbb{Z}/2 \mathbb{Z}$, such that $\Fil^1\mathbb{D}(Y)(R)_0$ and $\Fil^1\mathbb{D}(Y)(R)_1$ are orthogonal under $\langle\ ,\ \rangle_0$.

\subsection{The space $\tN$}\label{sec:space-hn}
Let $(X,\iota_X, \lambda_X,\rho_X, Y, \iota_Y,\lambda_Y,\rho_Y)\in \tN(S)$. The isogeny $\wit\alpha: Y\times \ov\CE \rightarrow X$ induces a homomorphism of free $R$-modules of rank $n+1$ $$\alpha_*: \mathbb{D}(Y)(R)_i \oplus \mathbb{D}(\ov\CE)(R)_i\rightarrow \mathbb{D}(X)(R)_i,$$ whose cokernel is a free $R/\varpi R$-module of rank 1.  The condition $\alpha^*\lambda_X= \lambda_Y\times \varpi \lambda_{\ov\CE}$ translates to the compatibility of $\langle\ ,\ \rangle_i$ on $\mathbb{D}(Y\times \ov\CE)(R)$ and $\langle \ , \  \rangle_i$ on $\mathbb{D}(X)(R)$ under $\alpha_*$,
\begin{equation}
  \label{eq:compatiblity}
  \langle\ , \ \rangle_{\mathbb{D}(Y)(R)_i} \oplus \varpi\langle\ , \ \rangle_{\mathbb{D}(\ov\CE)(R)_i}=\langle \alpha_*(\ ), \alpha_*(\ )\rangle_{\mathbb{D}(X)(R)_i}.
\end{equation}
 The homomorphism $\alpha_*$ preserves Hodge filtrations when base changing to $S$: $$\alpha_*(\Fil^1\mathbb{D}(Y)(S)_i \times\Fil^1\mathbb{D}(\ov\CE)(S)_i)\subseteq \Fil^1\mathbb{D}(X)(S)_i,$$ where both the source and target have rank $n$ when $i=0$ and rank 1 when $i=1$.

By Grothendieck--Messing theory, a lifting of $(X,\iota_X, \lambda_X,\rho_X, Y, \iota_Y,\lambda_Y,\rho_Y)\in \tN(S)$ to $\tN(R)$ corresponds to liftings $\Fil^1\mathbb{D}(X)(R)_0$, $\Fil^1\mathbb{D}(Y)(R)_i$ such that
\begin{altenumerate}
\item $\Fil^1\mathbb{D}(Y)(R)_0$  and  $\Fil^1\mathbb{D}(Y)(R)_1$ are orthogonal under $\langle\ ,\ \rangle_0$.  
\item $\alpha_*$ preserves Hodge filtrations: $$\alpha_*(\Fil^1\mathbb{D}(Y)(R)_i \times\Fil^1\mathbb{D}(\ov\CE)(R)_i)\subseteq \Fil^1\mathbb{D}(X)(R)_i.$$ Here $\Fil^1\mathbb{D}(X)(R)_1$ is determined by $\Fil^1\mathbb{D}(X)(R)_0$ (see \S\ref{sec:spaces-n}). 
\end{altenumerate}

\section{Proof of Theorem \ref{conj:KRSZ} and Conjecture \ref{conjreg}}\label{s:Pfs}

\subsection{The exceptional divisor $\wt{\CN}_n^{[1], {\rm exc}}$ of $\pi_2$}

\begin{proposition}\label{prop:ZssCartier}
    $\wt{\CN}_n^{[1], {\rm exc}}$ is a Cartier divisor in $\tN$.
\end{proposition}

\begin{proof}
 We write $\mathcal T:=\wt{\CN}_n^{[1], {\rm exc}}$ for brevity. Let $z\in \mathcal{T}(\kb)$. Let $O_z$ be the local ring of $\tN$ at $z$ with maximal ideal $\mathfrak{m}$. Let $J\subseteq O_z$ be the ideal defining $\mathcal{T}$ at $z$. Let $R=O_z/\mathfrak{m}J$ and $I=J/\mathfrak{m}J$. Then $R$ is a local noetherian $\OFb$-algebra on which $\varpi$ is nilpotent and $I^2=0$. By Nakayama's lemma, to show that $J$ is principal it suffices to show that $I$ is principal. It remains to show the following more general assertion: for  any local noetherian $\OFb$-algebra  $R$ on which $\varpi$ is nilpotent, a nonzero ideal $I\subseteq R$ such that $I^2=0$  and $S=R/I$, the condition that a lifting of $\tilde z\in\tN(R)$ of $z\in \CT(S)$ lies in $\CT(R)$ is given by the vanishing of one nonzero element in $I$.

  Let $z=(X,\iota_X, \lambda_X,\rho_X, Y, \iota_Y,\lambda_Y,\rho_Y)\in \mathcal T(S)$.  Write $z_i=\pi_i(z)$.  By the definition of $\mathcal{T}$, we know that $z_2=(X, \iota_X, \lambda_X,\rho_X)\in \Zx^\mathrm{cent}(S)$. Hence $\varpi=0$ in $S$ and $X= X_{\kb} \times_{\kb} S$ for a unique $\Lambda\in\Ver^0(\mathbb{W})$ and the unique point $(X_{\kb}, \iota_{X_{\kb}},\lambda_{X_{\kb}}, \rho_{X_{\kb}})\in \mathcal{V}(\Lambda(\x))(\kb)$. Therefore $$\mathbb{D}(X)(R)=\mathbb{D}(X_{\kb})(\OFb) \otimes_{\OFb}R,\quad \mathbb{D}(X)(S)=\mathbb{D}(X_{\kb})(\OFb)\otimes_\OFb S=\mathbb{D}(X_{\kb})(\kb) \otimes_{\kb}S$$ and $$\Fil^1\mathbb{D}(X)(S)_0=\Fil^1\mathbb{D}(X)(\kb)_0 \otimes_{\kb}S.$$

  A lifting of $z\in \mathcal{T}(S)$ to $\tilde z\in\tN(R)$ corresponds to liftings $\Fil^1\mathbb{D}(X)(R)_0$, resp. $\Fil^1\mathbb{D}(Y)(R)_i$, of $\Fil^1\mathbb{D}(X)(S)_0$, resp. $\Fil^1\mathbb{D}(Y)(S)_i$, as in \S\ref{sec:space-hn}.  Note that $\tilde z\in \mathcal{T}(R)$ if and only if $\varpi=0$ in $R$ and $$\Fil^1\mathbb{D}(X)(R)_0=\Fil^1\mathbb{D}(X)(\kb)_0 \otimes_{\kb} R.$$  We would like to show that the condition that $\tilde z\in\mathcal{T}(R)$ is given by the vanishing of one nonzero element in $I$.

  Let $\{e_1,\ldots, e_n\}$ be an $O_F$-basis of $\Lambda$ and let $e_0=\varpi^{-1}\x$. Then by the first case of the proof of Lemma \ref{lem:Zss}, we have $$\mathbb{D}(X_{\kb})(\OFb)_0=\langle e_0,\ldots,e_n\rangle_{\OFb},\quad \Fil^1\mathbb{D}(X_{\kb})(\kb)_0=\langle e_1,\ldots,e_n\rangle_{\kb}.$$ Hence $$\mathbb{D}(X)(R)_0=\langle e_0,\ldots, e_n\rangle_R, \quad \Fil^1\mathbb{D}(X)(S)_0=\langle e_1,\ldots,e_n\rangle_S.$$ A lifting of $z_2=(X, \iota_X, \lambda_X,\rho_X)\in\Zx^\mathrm{cent}(S)$ to $\tilde z_2\in\Zx(R)$ then corresponds to an $R$-hyperplane $\Fil^1\mathbb{D}(X)(R)_0$ in $\langle e_0,\ldots,e_n\rangle_R$ given by an equation $$e_0^*+\lambda_1e_1^*+\cdots+\lambda_n e_n^*=0, \quad \lambda_i\in I, i=1,\ldots, n$$ such that $\x\in \Fil^1\mathbb{D}(X)(R)_0$, i.e., $\lambda_i\in I$ for $i=1,\ldots, n$ and $\varpi=0$ in $R$.

Since $z_1=(Y,\iota_Y,\lambda_Y,\rho_Y)\not\in\Nns(S)$, we know that $\Fil^1\mathbb{D}(Y)(S)_1=\mathbb{D}(Y)(S)_0^\perp$ is determined by $\mathbb{D}(Y)(S)_0$. Since the cokernel of $$\alpha_*: \mathbb{D}(Y)(R)_0\times \mathbb{D}(\ov\CE)(R)_0\rightarrow \mathbb{D}(X)(R)_0$$ is a free $R/\varpi R$-module of rank $1$, we know that $\alpha_*$ induces an isomorphism $$\mathbb{D}(Y)(R)_0\cong \langle e_1,\ldots,e_n\rangle_R.$$  After changing the basis $\{e_1,\ldots,e_n\}$ we may assume that $$\mathbb{D}(Y)(R)_1\cong \langle f_1,\ldots, f_n\rangle_R$$ with $\langle e_1,f_1\rangle_0=\varpi$ and $\langle e_i,f_j\rangle_0=\delta_{ij}$ for $(i,j)\ne(1,1)$. Then $\mathbb{D}(Y)(S)_1^\perp=\langle e_1\rangle_S$, and $\Fil^1\mathbb{D}(Y)(S)_1=\mathbb{D}(Y)(S)_0^\perp=\langle f_1\rangle_S$. 

  \begin{altenumerate}
  \item   First consider the case that $z_1=(Y,\iota_Y,\lambda_Y,\rho_Y)\in \Nbl(S)$ lies in the balloon stratum, cf. Definition \ref{strata}. Then $\mathbb{D}(Y)(S)_1^\perp=\langle e_1\rangle_S\not\subseteq \Fil^1\mathbb{D}(Y)(S)_0$. We may assume that the $S$-hyperplane $\Fil^1\mathbb{D}(Y)(S)_0$ in $\mathbb{D}(Y)(S)_0=\langle e_1,\ldots,e_n\rangle_S$ is given by the equation $$e_1^*+a_2e_2^*+\cdots+a_ne_n^*=0,\quad a_i\in S.$$ A lifting of $z_1=(Y, \iota_Y, \lambda_Y,\rho_Y)\in \Na(S)$ to $\tilde z_1\in \Na(R)$ then corresponds to an $R$-hyperplane $\Fil^1\mathbb{D}(Y)(R)_0$ in $\mathbb{D}(Y)(R)_0=\langle e_1,\ldots,e_n\rangle_R$ given by an equation
    \begin{equation}
      \label{eq:hyperplane1}
      e_1^*+\mu_2e_2+\cdots +\mu_{n-1}e_{n-1}^*+\mu_ne_n^*=0, \quad \mu_i\in a_i+ I , 
          \end{equation}
 and an $R$-line $\Fil^1\mathbb{D}(Y)(R)_1$ in  $ \mathbb{D}(Y)(R)_1$ generated by $$f_1+\nu_2f_2+\cdots+\nu_nf_n,\quad \nu_i\in I, i=2,\ldots, n$$ such that $\Fil^1\mathbb{D}(Y)(R)_0$ and $\Fil^1\mathbb{D}(Y)(R)_1$ are orthogonal under $\langle\ ,\ \rangle_0$. The orthogonality condition is equivalent to the $R$-hyperplane (\ref{eq:hyperplane1}) being contained in the $R$-subspace defined by $$\varpi e_1^*+\nu_2e_2^*+\cdots +\nu_{n-1}e_{n-1}^*+\nu_ne_n^*=0.$$ Hence $$\nu_2=\varpi \mu_2,\cdots, \nu_{n-1}=\varpi \mu_{n-1},\nu_n=\varpi\mu_n,$$ and the liftings $\tilde z_1\in \Na(R)$ are parametrized by $\mu_2,\ldots,\mu_{n}\in I$.

 For $\tilde z_1\in \Na(R)$ lifting $z_1$ and $\tilde z_2\in \Zx(R)$ lifting $z_2$, we have $\tilde z=(\tilde z_1, \tilde z_2)\in \tN(R)$ if and only if $\alpha_*$ preserves the Hodge filtrations, namely $\alpha_*(\Fil^1\mathbb{D}(Y)(R)_i)$ is contained in $\Fil^1\mathbb{D}(X)(R)_i$ for $i\in \mathbb{Z}/ 2 \mathbb{Z}$. For $i=0$, the preservation of the Hodge filtrations means that $$\lambda_1 e_1^*+\cdots+\lambda_ne_n^*=\lambda_1(e_1^*+\mu_2e_2^*+\cdots +\mu_{n-1}e_{n-1}^*+\mu_ne_n^*).$$ It follows that $$\lambda_2=\lambda_1\mu_2,\ldots, \lambda_{n-1}=\lambda_1\mu_{n-1}, \lambda_n=\lambda_1\mu_n.$$ For $i=1$, the preservation of the Hodge filtrations means that the $R$-line $\varpi f_1+\nu_2 f_2+\cdots +\nu_n f_n$ is contained in the $R$-line $f_0+\lambda_1f_1+\cdots+ \lambda_nf_n$. Hence $\varpi=0$ in $R$ and $$\nu_2=\cdots=\nu_n=0.$$ Therefore the liftings of $z\in \mathcal{T}(S)$ to $\tilde z\in \tN(R)$ are parametrized by $\mu_1,\ldots,\mu_{n-1},\lambda_1\in I$. 

Now for $\tilde z=(\tilde z_1, \tilde z_2)\in \tN(R)$ lifting $z\in \mathcal{T}(S)$, we have $\tilde z\in \mathcal{T}(R)$ if and only if $\Fil^1\mathbb{D}(X)(R)_0$ is given by the equation $e_0^*=0$, which is cut out by one equation $\lambda_1=0$, as desired.
\item Next consider the case that $z_1\in \Nlk(S)$ lies in the link stratum, cf. Definition \ref{strata}. Then $\mathbb{D}(Y)(S)_1^\perp=\langle e_1\rangle_S\subseteq \Fil^1\mathbb{D}(Y)(S)_0$. Without loss of generality we may assume that the $S$-hyperplane $\Fil^1\mathbb{D}(Y)(S)_0$ in $\mathbb{D}(Y)(S)_0=\langle e_1,\ldots,e_n\rangle_S$ is given by the equation $e_n^*=0$. A lifting of $z_1=(Y, \iota_Y, \lambda_Y,\rho_Y)\in \Na(S)$ to $\tilde z_1\in \Na(R)$ then corresponds to an $R$-hyperplane $\Fil^1\mathbb{D}(Y)(R)_0$ in $\mathbb{D}(Y)(R)_0=\langle e_1,\ldots,e_n\rangle_R$ given by an equation
    \begin{equation}\label{eq:hyperplane2}
      \mu_1e_1^*+\cdots +\mu_{n-1}e_{n-1}^*+e_n^*=0, \quad \mu_i\in I, i=1,\ldots,n-1
    \end{equation}
    and an $R$-line $\Fil^1\mathbb{D}(Y)(R)_1$ in  $ \mathbb{D}(Y)(R)_1$ generated by $$f_1+\nu_2f_2+\cdots+\nu_nf_n,\quad \nu_i\in I, i=2,\ldots, n$$ such that $\Fil^1\mathbb{D}(Y)(R)_0$ and $\Fil^1\mathbb{D}(Y)(R)_1$ and orthogonal under $\langle\ ,\ \rangle_0$. The orthogonality condition is equivalent to the $R$-hyperplanes  (\ref{eq:hyperplane2}) being contained in the $R$-subspace $$\varpi e_1^*+\nu_2e_2^*+\cdots +\nu_{n-1}e_{n-1}^*+\nu_ne_n^*=0,$$ which is equivalent to $$\nu_n(\mu_1e_1^*+\cdots +\mu_{n-1}e_{n-1}^*+e_n^*)=\varpi e_1^*+\nu_2 e_2^*+\cdots+\nu_{n}e_n^*,$$ i.e., $$\varpi=\nu_n\mu_1, \nu_2=\nu_n\mu_2,\ldots,\nu_{n-1}=\nu_n\mu_{n-1}.$$ Since $I^2=0$ we know that $\varpi=0$ in $R$ and $$\nu_2=\cdots =\nu_{n-1}=0.$$
    Hence the liftings $\tilde z_1\in \Na(R)$ are parametrized by $\mu_1,\ldots,\mu_{n-1},\nu_n\in I$.

    For $\tilde z_1\in \Na(R)$ lifting $z_1$, and $\tilde z_2\in \Zx(R)$ lifting $z_2$, we have $\tilde z=(\tilde z_1, \tilde z_2)\in \tN(R)$ if and only if $\alpha_*$ preserves the Hodge filtrations. For $i=0$, the preservation of the Hodge filtrations means that $$\lambda_1 e_1^*+\cdots+\lambda_ne_n^*=\lambda_n(\mu_1e_1^*+\cdots +\mu_{n-1}e_{n-1}^*+e_n^*).$$ It follows that $$\lambda_1=\lambda_n\mu_1,\ldots, \lambda_{n-1}=\lambda_n\mu_{n-1}.$$ For $i=1$, the preservation of the Hodge filtrations means that the $R$-line $\varpi f_1+\nu_2 f_2+\cdots +\nu_n f_n$ is contained in the $R$-line $f_0+\lambda_1f_1+\cdots+ \lambda_nf_n$. Hence $\varpi=0$ in $R$ and $$\nu_2=\cdots=\nu_n=0.$$
    Therefore the liftings of $z\in \mathcal{Z}(u)(S)$ to $\tilde z\in \tN(R)$ are parametrized by $\mu_1,\ldots,\mu_{n-1},\lambda_n\in I$.

    Now for $\tilde z=(\tilde z_1, \tilde z_2)\in \tN(R)$ lifting $z\in \mathcal{Z}(u)(S)$, we have $\tilde z\in \mathcal{T}(R)$ if and only if $\Fil^1\mathbb{D}(X)(R)_0$ is given by the equation $e_0^*=0$, which is cut out by one equation $\lambda_n=0$, as desired.
    \qedhere
  \end{altenumerate}
\end{proof}

\begin{corollary}\label{cor:PLtangent}
  Let $z\in \wt{\CN}_n^{[1], {\rm exc}}(\kb)$. Then $$\dim_{\kb}T_z\,\wt{\CN}_n^{[1], {\rm exc}}=n-1.$$
\end{corollary}

\begin{proof}
  By the proof of Proposition \ref{prop:ZssCartier} applied to $R=\kb[\varepsilon]/\varepsilon^2$, $I=(\varepsilon)$ and $S=\kb$, we know that the liftings of $z\in \mathcal{T}(\kb)$ to $\mathcal{T}(\kb[\varepsilon]/\varepsilon^2)$  are parametrized by $n-1$ free variables in $I=(\varepsilon)$. Hence $\dim_{\kb}T_z(\mathcal{T})=n-1$.
\end{proof}

\begin{definition}\label{defPL}
For $\Lambda\in \Ver^0(\mathbb{W}_n)$, write $\PL:=\pi_2^{-1}(\mathcal{V}(\Lambda(\x)))$, a closed $\kb$-subscheme of $\tN$. 
\end{definition}

\begin{proposition}\label{prop:pi2ss}
 
For $\Lambda\in \Ver^0(\mathbb{W}_n)$,  the morphism $\pi_1$ induces an isomorphism $\PL\isoarrow \mathcal{V}(\Lambda)\simeq\mathbb{P}^{n-1}$. In particular, there is a decomposition
\begin{equation}
  \label{eq:pi2ss}
  \wt{\CN}_n^{[1], {\rm exc}}=\bigsqcup_{\Lambda\in \Ver^0(\mathbb{W}_n)}\PL\simeq\bigsqcup_{\Lambda\in \Ver^0(\mathbb{W}_n)}\mathbb{P}^{n-1}.
\end{equation}
\end{proposition}

\begin{proof}
  By Lemma \ref{lem:bijectionpoints} (\ref{item:3}), we know that $\pi_1$ restricts to a morphism of $\kb$-schemes $\PL\rightarrow \mathcal{V}(\Lambda)$ which induces a bijection on $\kb$-points. By working systematically with a Cohen ring instead of the Witt ring, we obtain that $\PL\rightarrow \mathcal{V}(\Lambda)$ induces a bijection on $k'$-points for any field extension $k'/\kb$. Thus $\PL\rightarrow \mathcal{V}(\Lambda)$ is birational and universally bijective. Since $\pi_1$ is finite (Proposition \ref{propfinflat} below), we know that $\PL\rightarrow \mathcal{V}(\Lambda)$ is proper and therefore a universal homeomorphism. Hence $\PL$ is irreducible of dimension $n-1$. It follows from Corollary \ref{cor:PLtangent} that $\PL$ is smooth and hence reduced. Now the morphism $\PL\rightarrow \mathcal{V}(\Lambda)$ is a birational, bijective and proper morphism with an integral source and a normal target, hence it is  an isomorphism by the Zariski main theorem.
\end{proof}

\begin{corollary}\label{cor:sslocus}
 $\wt{\CN}_n^{[1], {\rm exc}}=\pi_1^{-1}(\CN_n^{[1], \bullet})^\mathrm{red}$.
\end{corollary}

\begin{proof}
By Proposition \ref{prop:pi2ss}, we know that $\wt{\CN}_n^{[1], {\rm exc}}$ is a reduced closed subscheme of $\tN$. The result then follows as it has the same set of $\kb$-points as the reduced closed subscheme $\pi_1^{-1}(\CN_n^{[1], \bullet})^\mathrm{red}$, by Lemma \ref{lem:bijectionpoints}~(\ref{item:3}).
\end{proof}

\subsection{The morphism $\pi_2$ and the regularity of $\tN$}\label{sec:morph-pi_2-regul}

\begin{proposition}\label{prop:pi2ns}
The restriction of $\pi_2$ to $\tNns$ induces an isomorphism $\tNns\simeq \Zns$. 
\end{proposition}

\begin{proof} 
  By Lemma \ref{lem:bijectionpoints} (\ref{item:1}, \ref{item:2}), we know that $\pi_2$ induces a bijection $\tNns(\kb)\simeq\Zns(\kb)$. Since $\tN$ is formally locally of finite type, it remains to show that the restriction of $\pi_2$ to $\tNns$ is formally \'etale. Let $R$ be a local noetherian $\OFb$-algebra on which $\varpi$ is nilpotent. Let $I\subseteq R$ be an ideal such that $I^2=0$. Let $S=R/I$. Let $z=(X,\iota_X, \lambda_X,\rho_X, Y, \iota_Y,\lambda_Y,\rho_Y)\in \tNns(S)$ and $z_2:=\pi_2(z)=(X, \iota_X,\lambda_X,\rho_X)\in \Zns(S)$. To show the formal \'etaleness, we need to show that for any lift $\tilde z_2 \in \Zns(R)$ of $z_2$, there exists a unique lift $\tilde z\in \tNns(R)$  of $z$ such that $\pi_2(\tilde z)=\tilde z_2$. Without loss of generality we may assume that $R$ has residue field $\kb$.

  By the second case of the proof of Lemma \ref{lem:Zss}, there exists a $\kb$-basis $\{\bar e_0, \bar e_1,,\ldots, \bar e_n\}$ of $\mathbb{D}(X_{\kb})(\kb)_0$ such that $$\Fil^1\mathbb{D}(X_{\kb})(\kb)_0=\langle \bar e_0, \bar e_1,\ldots, \bar e_{n-1}\rangle_{\kb},\quad \x_{*}(\mathbb{D}(\ov\CE)(\kb)_0)=\langle \bar e_0\rangle_{\kb}.$$ Since $$\alpha_*: \mathbb{D}(Y)(R)_0\oplus \mathbb{D}(\ov\CE)(R)_0\rightarrow \mathbb{D}(X)(R)_0$$ has cokernel a free $R/\varpi R$-module of rank 1, we may lift $\{\bar e_0,\ldots, \bar e_n\}$ to an $R$-basis $\{e_0, e_1,\ldots,  e_n\}$ of $\mathbb{D}(X)(R)_0$ and find an $R$-basis $\{f_1,\ldots, f_n\}$ of $\mathbb{D}(Y)(R)_0$ such that $$\alpha_*(f_i)=e_i,\ i=1,\ldots, n-1,\quad \alpha_*(f_n)=\varpi e_n,\quad \alpha_*(\mathbb{D}(\ov\CE)(R)_0)=\langle e_0\rangle_R.$$  Assume that the $S$-hyperplane $\Fil^1\mathbb{D}(Y)(S)_0\subseteq \mathbb{D}(Y)(S)_0$ is defined by an equation $$\lambda_1f_1^*+\cdots+ \lambda_{n-1}f_{n-1}^*+\lambda_nf_n^*=0,\ \lambda_i\in S.$$    Since $z\in\tNns(S)$, we know by Lemma \ref{lem:bijectionpoints} (\ref{item:4}) that $\pi_1(z)\in \Nns(S)$ and hence $$\mathbb{D}(Y_{\kb})(\kb)_1^\perp=\langle f_n\rangle_{\kb}\subseteq \Fil^1\mathbb{D}(Y_{\kb})(\kb)_0.$$ In particular, we know that $\lambda_n\not\in S^\times$ and thus there exists some $1\le i \le n-1$ such that $\lambda_i\in S^\times$. Without loss of generality we may assume that $\lambda_1=1$. The fact that $\alpha_*$ preserves the Hodge filtrations over $S$ implies that the $S$-hyperplane $\Fil^1\mathbb{D}(X)(S)_0\subseteq \mathbb{D}(Y)(S)_0$ is defined by an equation $$e_1^*+\lambda_2e_2^*+\cdots +\lambda_{n-1}e_{n-1}^*+\lambda'_{n}e_n^*=0,\quad \lambda_n'\in S, \varpi\lambda_n'=\lambda_n.$$

  The lift $\tilde z_2$ of $z_2$ corresponds a hyperplane $\Fil^1\mathbb{D}(X)(R)_0\subset \mathbb{D}(X)(R)_0$ lifting $\Fil^1\mathbb{D}(X)(S)_0$, defined by an equation $$e_1^*+\mu_2e_2^*+\cdots +\mu_ne_n^*=0,\quad  \mu_i\in \lambda_i+I, i=2,\ldots,n-1,\ \mu_n\in \lambda_n'+I.$$ A lift $\tilde z$ of $z$ such that $\pi_2(\tilde z)=\tilde z_2$ corresponds to an $R$-hyperplane $\Fil^1\mathbb{D}(Y)(R)_0\subseteq \mathbb{D}(Y)(R)_0$ lifting $\Fil^1\mathbb{D}(Y)(S)_0$ defined by an equation $$f_1^*+\nu_2f_2^*+\cdots+\nu_nf_n^*=0,\quad \nu_i\in \lambda_i+I, i=2,\ldots, n$$ and an $R$-line $\Fil^1\mathbb{D}(Y)(R)_1\subseteq \mathbb{D}(Y)(R)_1$ lifting $\Fil^1\mathbb{D}(Y)(R)$ such that $\alpha_*$ preserves the Hodge filtrations over $R$. For $i=0$, the preservation of the Hodge filtrations means that $$\nu_i=\mu_i, i=2,\ldots,n-1,\ \nu_n=\varpi\mu_n.$$ Hence such a lift $\Fil^1\mathbb{D}(Y)(R)_0$ exists and is uniquely determined by $\Fil^1\mathbb{D}(X)(R)_0$. A similar argument shows that $\Fil^1\mathbb{D}(Y)(R)_1$ also exists and is uniquely determined by $\Fil^1\mathbb{D}(X)(R)_1$. Thus such a lift $\tilde z$ exists and is uniquely determined by $\tilde z_2$.
\end{proof}

\begin{corollary}\label{correg}
The formal scheme  $\tN$ is regular of dimension $n$.
\end{corollary}

\begin{proof}
 By (\ref{eq:pi2ss}), we know that $\wt{\CN}_n^{[1], {\rm exc}}$ is regular. It follows from Proposition \ref{prop:ZssCartier} that $\tN$ admits a regular Cartier divisor $\wt{\CN}_n^{[1], {\rm exc}}$, hence $\tN$ is regular at all points $z\in \wt{\CN}_n^{[1], {\rm exc}}(\kb)$.  By Proposition \ref{prop:pi2ns} and the fact that $\Zx$ is regular, we know that $\tN$ is also regular at all points $z\in \tNns(\kb)$. Therefore $\tN$ is regular at all points $z\in \tN(\kb)$.
\end{proof}

\subsection{The morphism $\pi_1$}

\begin{proposition}\label{propfinflat}
  The morphism $\pi_1: \tN\rightarrow \Na$ is finite flat of degree $q+1$, \'etale along $\Nns$, and totally ramified along $\CN_n^{[1], \bullet}$.  
\end{proposition}

\begin{proof}
  Let $S$ be a noetherian $\OFb$-algebra on which $\varpi$ is nilpotent.  Note that any $S$-point $z=(X,\iota_X, \lambda_X,\rho_X, Y, \iota_Y,\lambda_Y,\rho_Y)\in\tN(S)$ is determined by $z_1=(Y,\iota_Y, \lambda_Y,\rho_Y)\in\Na(S)$ together with $\ker (\tilde \alpha: Y\times\ov\CE_S\rightarrow X) \subseteq (Y\times\ov\CE_S)[\varpi]$. Moreover, the condition for a subscheme of the projective $S$-scheme $(Y\times \ov\CE_S)[\varpi]$ to appear as $\ker (\tilde \alpha: Y\times\ov\CE_S\rightarrow X)$ for some $z\in \pi_1^{-1}(z_1)$ is a closed condition. Hence by the theory of Hilbert schemes \cite[Thm. 3.1]{Grothendieck1995}, the morphism $\pi_1$ is relatively representable by a projective scheme  (hence proper). Since $\pi_1$ is quasi-finite by Lemma~\ref{lem:bijectionpoints}~(\ref{item:3})(\ref{item:4}), we know that $\pi_1$ is finite.

  Since $\tN$ is regular (hence Cohen--Macaulay) by Corollary \ref{correg}  and $\Na$ is regular by Proposition \ref{prop:localmodel}, we know that $\pi_1$ is flat by the miracle flatness theorem.

  The generic degree of $\pi_1$ is equal to the number of type 0 lattices containing a fixed type 2 lattice in an $F/F_0$-hermitian space of dimension $n+1$, which is $q+1$. Hence $\pi_1$ is finite flat of degree $q+1$. Comparing the degree $q+1$ with the size of fibers at $\kb$-points in Lemma~\ref{lem:bijectionpoints}~(\ref{item:3})(\ref{item:4}) it follows that $\pi_1$ is \'etale along $\Nns$ and totally ramified along $\CN_n^{[1], \bullet}$.
\end{proof}

\subsection{The self-intersection number of the exceptional divisor}
Recall from Definition \ref{defPL} the closed subscheme $\BP_\Lambda$ of $\tN$. 
\begin{proposition}\label{prop:self-inters-numb}
For $\Lambda\in \Ver^0(\mathbb{W}_n)$, the normal bundle $N_{\mathbb{P}_\Lambda/\tN}$ is isomorphic to $\mathcal{O}_{\mathbb{P}_\Lambda}(-1)$. In particular, the $n$-fold self-intersection of $\mathbb{P}_\Lambda$ in $\tN$ is equal to $$\chi(\tN, \PL\cap^\BL \cdots\cap^\BL \PL)=(-1)^{n-1}.$$
\end{proposition}

\begin{proof}
  By Proposition \ref{prop:ZssCartier}, we know that the normal bundle $N_{\mathbb{P}_\Lambda/\tN}$ is a line bundle on $\mathbb{P}_\Lambda\simeq \mathbb{P}^{n-1}$, and hence $N_{\mathbb{P}_\Lambda/\tN}\simeq \mathcal{O}_{\mathbb{P}_\Lambda}(m)$ for a unique integer $m$. Let $\Lambda^\flat\subseteq \Lambda$ be a type 0 lattice of rank $n-2$. Then we have a closed immersion $\delta: \mathcal{N}_3\rightarrow \CN_{n+1}^{[1]}$ (\cite[\S 2.11]{LZ22}) which identifies $\mathcal{N}_3$ with the Kudla--Rapoport cycle $\mathcal{Z}(\Lambda^\flat)\subseteq \CN_{n+1}^{[1]}$. Let $\mathcal{Z}^\flat(\x)=\Zx\cap \mathcal{Z}(\Lambda^\flat)\subseteq \mathcal{N}_3$ be a valuation one Kudla--Rapoport divisor on $\mathcal{N}_3$ and $\pi_2^\flat: \wit{\mathcal{N}}^{[1]}_2\rightarrow \mathcal{Z}^\flat(\x)$ the natural projection. Then $\delta$ induces a cartesian diagram $$\xymatrix{\wit{\mathcal{N}}^{[1]}_2 \ar@{}[rd]|{\square}  \ar[r]^-{\pi_2^\flat}    \ar[d]^-{\tilde\delta} & \mathcal{Z}^\flat(\x) 
  \ar[d]^-{\delta} \\ \tN \ar[r]^-{\pi_2} & \Zx   
}$$ Let $\Lambda_2$ be the orthogonal complement of $\Lambda^\flat$ in $\Lambda$ and $\mathbb{W}_2:=\Lambda_{2,F}$. Then $\Lambda_2\in \Ver^0(\mathbb{W}_2)$ and $\tilde \delta$ identifies $\mathbb{P}_{\Lambda_2}\subseteq \wit{\mathcal{N}}^{[1]}_2$ with a projective line in $\PL$. Hence $$N_{\mathbb{P}_{\Lambda_2}/\wit{\mathcal{N}}^1_2}=\tilde \delta^*(N_{\PL/\tN})\simeq \tilde\delta^*(\mathcal{O}_{\PL}(m))=\mathcal{O}_{\mathbb{P}_{\Lambda_2}}(m).$$ Thus to show that $m=-1$ we are reduced to the case $n=2$.

Now assume that $n=2$. Since $\pi_1$ is finite flat of degree $q+1$ and totally ramified along $\mathcal{V}(\Lambda)$ (Proposition \ref{propfinflat}), by the projection formula we have $$(q+1)\cdot\chi(\tN,\PL\cap^\BL \PL)=\chi(\Na,\mathcal{V}(\Lambda)\cap^\BL \mathcal{V}(\Lambda)).$$ Since  $\Na$ is a regular formal surface, whose special fiber is a strict normal crossing divisor with exactly $q+1$ irreducible curves intersecting $\mathcal{V}(\Lambda)$, we know that $$\chi(\Na,\mathcal{V}(\Lambda)\cap^\BL \mathcal{V}(\Lambda))=-(q+1).$$ Hence $\chi(\tN,\PL\cap^\BL \PL)=-1$,  which is equivalent to $m=-1$ when $n=2$, as desired.
\end{proof}

\subsection{Proof of Theorem \ref{conj:KRSZ}}\label{sec:proof-conj-refc}

Item (\ref{item:KR1}) is proved in Corollary \ref{correg}.  Item  (\ref{item:KR2})  is proved in Proposition \ref{propfinflat}. Item  (\ref{item:KR3}) is proved in Proposition \ref{prop:pi2ns}. Item  (\ref{item:KR4})  is proved in Proposition \ref{prop:ZssCartier} and Proposition \ref{prop:pi2ss}. Item  (\ref{item:KR5})  is proved in Proposition \ref{prop:self-inters-numb}.

\subsection{Singularities of $\tN$ and proof of Conjecture \ref{conjreg}}
By Corollary  \ref{cor:sslocus}, we have $\tNns(\kb)=\pi_1^{-1}(\Nns)(\kb)$.
Also let $\tNlk=\pi_1^{-1}(\Nlk)$ and $\tNbl=\pi_1^{-1}(\Nbl)$ be the inverse images of the link stratum, resp. the balloon stratum, cf. Definition \ref{strata}.

\begin{theorem}\label{thm:singtN}
Let $\hat O_{\tilde z}$ be the complete local ring of $\tN$ at $\tilde z\in \tN(\kb)$. Then $$\hat O_{\tilde z}\simeq
\begin{cases}
  \OFb[[T_1,\ldots, T_{n-1}]], & \tilde z\in \tNns(\kb),\\
  \smallskip
  \OFb[[T_1,\ldots, T_n]]/(T_1^{q+1}-\varpi), & \tilde z \in \tNbl(\kb),\\
  \smallskip
  \OFb[[T_1,\ldots, T_n]]/(T_1^{q+1}T_2-\varpi), & \tilde z\in\tNlk(\kb).\\
\end{cases}
$$ In particular, Conjecture \ref{conjreg} holds when $r=1$.
\end{theorem}

\begin{proof}
  Let $z=\pi_1(\tilde z)\in \Na(\kb)$. Let $\hat O_z$ be the complete local ring of $\Na$ at $z$. Then $\pi_1$ induces a morphism $\hat O_z\rightarrow \hat O_{\tilde z}$. When $z\in\Nns(\kb)$, the morphism $\hat O_z\rightarrow \hat O_{\tilde z}$ is an isomorphism  by Proposition \ref{propfinflat}. Hence when $z\in\Nns(\kb)$, we have $\hat O_{\tilde z}\simeq\OFb[[T_1,\ldots, T_{n-1}]]$ by Proposition \ref{prop:localmodel}.

  When $z\in\CN_n^{[1], \bullet}(\kb)$,  by Proposition \ref{prop:localmodel} we have $\hat O_z\simeq\OFb[[X_1,\ldots, X_{n}]]/(X_1-\varpi)$ (resp. $\hat O_z\simeq \allowbreak\OFb[[X_1,\ldots, X_n]]/(X_1X_2-\varpi)$ when $z\in\Nbl(\kb)$ (resp. when $z\in\Nlk(\kb)$). Here we choose $X_1=0$ to be a local equation defining the Cartier divisor $\CN_n^{[1], \bullet}\subseteq \Na$ at $z$ and $X_2,\ldots, X_n$ to be a regular system of parameters for $\CN_n^{[1], \bullet}$ at $z$. Let $T_i\in \hat O_{\tilde z}$ be the image of $X_i$ under $\hat O_z\rightarrow \hat O_{\tilde z}$ for $i\ge2$. Let $T_1\in \hat O_{\tilde z}$ such that the local equation $T_1=0$ defines the Cartier divisor $\pi_2^{-1}(\Zx^\mathrm{ss})$ at $\tilde z$ (cf. Proposition \ref{prop:ZssCartier}). Then by Proposition \ref{propfinflat} and Corollary \ref{cor:sslocus}, we know that $T_2,\ldots, T_{n}$ form a regular system of parameters for $\pi_2^{-1}(\Zx^\mathrm{ss})$ at $\tilde z$ and the ideal $(T_1^{q+1})\subseteq \hat O_{\tilde z}$ equals the image of $(X_1)$. The result then follows.
\end{proof}

\end{document}